\documentclass[12pt]{article}
\usepackage[font=footnotesize,labelfont=bf]{caption}
\usepackage[normalem]{ulem}
\usepackage{amssymb,amsmath,bm}
\usepackage{mathrsfs}
\usepackage{bbm}
\usepackage{stackengine}
\usepackage{enumitem} 
\usepackage{constants,color}
\usepackage{appendix}
\usepackage{relsize}
\usepackage{tikz}
\def\qed{\hfill \mbox{\rule{0.5em}{0.5em}}}
\usepackage{subcaption}

\usepackage{graphicx}
\usepackage{amssymb}
\usepackage{bookmark}
\usepackage{indentfirst}
\usepackage{cleveref}

\newcommand{\be}{\begin{equation}}
\newcommand{\ee}{\end{equation}}
\newcommand{\bes}{\begin{equation*}}
\newcommand{\ees}{\end{equation*}}
\newcommand{\ba}{\begin{aligned}}
\newcommand{\ea}{\end{aligned}}
\newcommand{\bi}{\begin{itemize}}
\newcommand{\ei}{\end{itemize}}

\newcommand{\EE}{\mathbb{E}}

\allowdisplaybreaks

\numberwithin{equation}{section}
\newtheorem{assumption}{Assumption}[section]
\newtheorem{theorem}{Theorem}%[section]
\newtheorem{proposition}{Proposition}%[section]
\newtheorem{lemma}{Lemma}%[section]

\title{Parameter Estimation in Recurrent Tumor Evolution with Finite Carrying Capacity}
\author{Kevin Leder$^{1}$ \and\hspace*{-6pt}  Zicheng Wang$^{2}$ \and\hspace*{-6pt}Xuanming Zhang$^{1}$}
\date{%
    \footnotesize $^1$Department of Industrial and Systems Engineering, University of Minnesota, Twin Cities, MN 55455, USA. \\[2pt]
    $^2$School of Data Science, The Chinese University of Hong Kong, Shenzhen (CUHK-Shenzhen), China
}
\begin{document}
\maketitle
\begin{abstract}
In this work, we investigate the population dynamics of tumor cells under therapeutic pressure. Although drug treatment initially induces a reduction in tumor burden, treatment failure frequently occurs over time due to the emergence of drug resistance, ultimately leading to cancer recurrence. To model this process, we employ a two-type branching process with state-dependent growth rates. The model assumes an initial tumor population composed predominantly of drug-sensitive cells, with a small subpopulation of resistant cells. Sensitive cells may acquire resistance through mutation, which is coupled to a change in cellular fitness. Furthermore, the growth rates of resistant cells are modulated by the overall tumor burden. Using stochastic differential equation techniques, we establish a functional law of large numbers for the scaled populations of sensitive cells, resistant cells, and the initial resistant clone. We then define the stochastic recurrence time as the first time the total tumor population regrows to its initial size following treatment. For this recurrence time, as well as for measures of clonal diversity and the size of the largest resistant clone at recurrence, we derive corresponding law of large number limits. These asymptotic results provide a theoretical foundation for constructing statistically consistent estimators for key biological parameters, including the cellular growth rates, the mutation rate, and the initial fraction of resistant cells.
% In this work, we study the dynamics of a tumor cell population under pharmacological treatment. While drug treatment initially reduces tumor size, it often fails over time as tumor cells develop resistance, ultimately leading to cancer recurrence. We model this process using a two-type, state-dependent branching process, \textcolor{magenta}{assuming the tumor initially consists predominantly of drug-sensitive cells and a small fraction of resistant cells.} Drug-sensitive cells may acquire resistance through mutation, accompanied by fitness changes, while the overall tumor burden influences cellular growth rates. Using stochastic differential equation techniques, we establish a functional law of large numbers for the populations of resistant and sensitive cells as well as the initial resistant clone. Additionally, by defining the stochastic time of cancer recurrence—when the tumor regrows to its initial size—we derive law of large number limits for recurrence time, clonal diversity, and the size of the largest clone. These results enable the construction of consistent estimators for key parameters, such as cell growth rate, mutation rate, and the initial fraction of resistant cells.
% \\

\end{abstract}

\section{Introduction}
Despite substantial advances in cancer therapy, including chemotherapy, immunotherapy, and radiotherapy, initial antitumor responses are often transient, and disease relapse remains a common and formidable challenge. For example, in glioblastoma, the vast majority of patients experience relapse, with approximately 90\% recurring within two years and a median progression-free survival of only $\sim$7 months under contemporary care \cite{rogers2024re,vaz2023recurrent}. Similarly, in advanced epithelial ovarian cancer, around 85\% of cases recur within a decade \cite{perez2024seom}. Mechanistically, relapse is primarily driven by minimal residual disease that evades therapeutic elimination through intrinsic or acquired resistance. This adaptive process is underpinned by Darwinian selection of pre-existing resistant subclones alongside therapy-induced adaptations, such as genetic mutations and phenotypic plasticity \cite{johnson2014mutational,samur2023high,youk2024quantitative,zhang2012activation,beltran2019role}. Consequently, recurrent tumors exhibit pronounced intratumor heterogeneity at genomic, transcriptomic, and phenotypic levels \cite{ding2012clonal,wang2022single,simpson2025challenges,hobor2024mixed}. This intratumor heterogeneity substantially complicates the development of effective subsequent treatments \cite{dagogo2018tumour,raatz2021impact}, highlighting the critical need to understand the dynamics of relapse.

While intratumor heterogeneity in recurrent tumors undermines therapeutic durability, it simultaneously encodes valuable information on tumor evolution. This information presents an opportunity to infer key evolutionary parameters from genomic data using mathematical and computational frameworks.
%%%
A growing body of literature seeks to harness this opportunity. Building on branching-process models, Leder and colleagues \cite{leder2024parameter, leder2021clonal} analyzed the Simpson index (a measure of diversity based on the second moment of subclone-size distributions) to estimate tumor growth and mutation rates from single-time-point sequencing data of recurrent tumors. In a related approach, Gunnarsson et al. \cite{gunnarsson2025limit} examined the site-frequency spectrum of neutral mutations in exponentially growing populations and, through limit theorems, derives estimators for mutation rates and extinction probabilities. Williams et al. \cite{williams2018quantification} employed a branching process framework to model variant allele frequencies in bulk sequencing data, enabling the quantification of subclonal selection, relative fitness, and the timing of subclone emergence.
% There is a body of literature on this direction.
% Building on branching-process models, \cite{leder2024parameter} and \cite{leder2021clonal} analyze the Simpson Index—the second moment of subclone-size distributions—to estimate tumor growth and mutation rates from single-patient, single-time-point sequencing data, particularly in recurrent tumors. Along similar lines, \cite{gunnarsson2025limit} examines the site-frequency spectrum (SFS) of neutral mutations in exponentially growing populations and, via limit theorems, derives estimators for mutation rates and extinction probabilities. Likewise, \cite{williams2018quantification} employs a branching-process framework to model variant allele frequencies in bulk sequencing data, enabling quantification of subclonal selection, relative fitness, and the timing of subclone emergence. 
In another direction, cloneRate \cite{johnson2023clonerate} leveraged coalescent theory to analyze the distribution of shared mutations (those present in more than one but not all cells). This enables the rapid estimation of single-cell clonal growth rates and dynamics. Saleh et al. \cite{salehi2021clonal} introduced fitClone, which applies a diffusion approximation to the $K$-allele Wright--Fisher model with selection. By utilizing longitudinal measurements of clonal abundances from single-cell whole-genome sequencing, the method generates posterior probability densities for fitness values, thereby mapping clonal fitness landscapes over time. Collectively, these studies demonstrate how heterogeneity can be harnessed as a quantitative signal for inferring tumor evolutionary dynamics. However, a common limitation among these studies is the assumption of constant cellular growth rates, independent of the tumor microenvironment, which constrains their biological interpretability.

% In another direction, cloneRate \cite{johnson2023clonerate} leverages coalescent theory to analyze the distribution of shared mutations—those present in more than one but not all cells—facilitating rapid estimation of single-cell clonal growth rates and dynamics. Finally, \cite{salehi2021clonal} introduces fitClone, which applies a diffusion approximation to the $K$-allele Wright--Fisher model with selection, using longitudinal measurements of clonal abundances from single-cell whole-genome sequencing to generate posterior probability densities for fitness values and thereby map clonal fitness landscapes over time. Collectively, these studies demonstrate how heterogeneity can be transformed into a quantitative signal for reconstructing tumor evolutionary dynamics.

% While the studies above leverage diverse mathematical frameworks to recover tumor dynamics from clonal heterogeneity, most treat subclone fitness as constant—typically assuming exponential growth. Such simplifications, though convenient, conflict with the resource- and space-limited tumor microenvironment, where diffusion barriers (e.g., oxygen and nutrients), vascular dysfunction, immune pressure, and solid stress drive a progressive decline in overall growth rates, resulting in the decelerating kinetics observed in vivo. In practice, tumors face a finite carrying capacity rather than an unlimited niche. 

In practice, the limited space and resources inside a tumor, imposed by diffusion barriers (e.g., for oxygen and nutrients), vascular dysfunction, immune surveillance, and solid stress, collectively drive a progressive decline in net proliferation rates. This ultimately results in the decelerating growth kinetics characteristic of in vivo tumors.
A growing body of work explicitly incorporates resource constraints into models of tumor dynamics. For example, Benzekr et al. \cite{benzekry2014classical} established that capacity-dependent models, such as Gompertzian and logistic-type growth, provide more accurate descriptions and predictions of experimental tumor growth deceleration (e.g., in breast and lung carcinoma) than exponential models. This work laid a foundation for forecasting tumor trajectories under bounded resource conditions. In a related approach, Lambert \cite{lambert2005branching} introduced a stochastic branching process with logistic growth, incorporating density-dependent regulation where birth rates decline with population size due to resource competition. This model offers a probabilistic framework for studying population dynamics under carrying-capacity constraints. More recently, Lewinsohn et al. \cite{lewinsohn2023state} developed SDevo, a multi-type birth-death process that classifies solid tumor cells into ``edge'' and ``core'' states based on spatial location. By assigning state-dependent growth rates, this framework helps reveal evolutionary patterns of tumor expansion under both spatial and resource constraints.
% A substantial body of work explicitly incorporates this constraint into tumor dynamics. For instance, \cite{benzekry2014classical} demonstrates that Gompertzian, logistic-type, and other capacity-dependent models more accurately describe and predict the deceleration of experimental tumor growth—such as in breast and lung carcinoma—than exponential models, thereby providing a foundation for forecasting tumor trajectories under bounded resources. Building on this concept, \cite{lambert2005branching} introduces a branching process with logistic growth, a stochastic model that incorporates density-dependent regulation, where birth rates decrease with population size due to pairwise interactions or resource competition, offering a probabilistic framework for population dynamics under carrying-capacity limits. Lewinsohn and co-authors \cite{lewinsohn2023state} presents SDevo, a multi-type birth–death process that partitions solid tumor cells into “edge” and “core” states based on spatial location, assigning state-dependent growth rates to reveal evolutionary modes of tumor expansion under spatial and resource constraints. 
Evolutionary game theory provides another perspective grounded in limited capacity. For example, Zhang et al. \cite{zhang2017integrating} applied Lotka–Volterra competition dynamics to model subclones with distinct phenotypes in metastatic castrate-resistant prostate cancer; this framework is subsequently integrated into treatment simulations to predict evolutionary outcomes. 
The study of competitive interactions under resource constraints has further inspired the development of modern adaptive therapy. For example, Gatenby et al. \cite{gatenby2009adaptive} proposed a strategy that leverages these interactions between sensitive and resistant lineages. By employing modulated dosing based on state-feedback (e.g., PSA levels, ctDNA, or tumor volume thresholds), the approach intentionally preserves a population of therapy-sensitive cells to suppress the expansion of resistant ones, thereby delaying disease progression while minimizing cumulative drug toxicity. Ultimately, incorporating carrying-capacity constraints into mathematical models provides a more biologically realistic framework for interpreting tumor evolution and for designing resilient, evolutionarily-informed therapeutic strategies.

% For example, \cite{zhang2017integrating} applies Lotka–Volterra competition dynamics to model subclones with distinct phenotypes in metastatic castrate-resistant prostate cancer, integrating these into treatment strategies to simulate evolutionary outcomes. This emphasis on capacity constraints has also inspired modern adaptive therapy paradigms: \cite{gatenby2009adaptive} proposes leveraging competitive interactions between sensitive and resistant lineages through modulated dosing based on state-feedback (e.g., PSA, ctDNA, or volume thresholds), intentionally preserving sensitive populations to suppress resistant ones and thereby delay progression while minimizing cumulative drug exposure. Ultimately, incorporating carrying-capacity limits into mathematical models offers a more biologically grounded lens for interpreting tumor evolution and designing resilient therapeutic strategies.

To incorporate carrying capacity into tumor dynamics, we model the system as a multi-type branching process with state-dependent growth rates. Our objective is to quantify tumor evolution by establishing a functional law of large numbers (FLLN) for this process. The FLLN for density-dependent stochastic systems has been extensively studied in probability theory. Ethier and Kurtz \cite{ethier2009markov} developed a general framework for establishing FLLN and central limit theorems for density-dependent Markov processes, demonstrating that their trajectories can be approximated by solutions to ordinary differential equations over finite time intervals. More recently, Prodhomme \cite{prodhomme2023strong} improved these results by extending the time horizon to depend on and grow unbounded with the carrying capacity. In a related work, Bansaye et al. \cite{bansaye2024sharp} analyzed a multi-type birth–death process with density-dependent rates that models mutant invasion into an equilibrium resident population, providing limit approximations across different population phases. In the specific context of hematopoietic cell proliferation, Wang et al. \cite{wang2025stochastic} derived both the FLLN and functional central limit theorem for a regulated stochastic two-compartment model, demonstrating convergence of scaled densities to ODE dynamics and, under appropriate rescaling, to a time-inhomogeneous diffusion process.

Building upon our earlier model \cite{leder2024parameter} that did not account for carrying-capacity constraints, we extend the analysis to incorporate density-dependent regulation. Specifically, we examine the joint dynamics of two tumor subpopulations, sensitive and resistant cells, each evolving according to a birth–death process, where the proliferation of resistant cells is modulated by system-wide resource limitations. A fundamental distinction between our framework and the classical model \cite{ethier2009markov} concerns the transition mechanism: we introduce a mutation rate from sensitive to resistant cells that scales with total population size via a power-law relationship. This formulation is especially relevant for modeling tumor evolution, where large population sizes and rare mutation events make such scaling biologically well-motivated. However, this modeling choice introduces significant theoretical challenges for the analysis and the derivation of a FLLN. Specifically, the presence of this state-dependent transition term prevents direct application of the standard FLLN framework \cite{ethier2009markov}, as that limiting ordinary differential equation will not account for density-driven mutation dynamics. To address these challenges, we define a stochastic stopping time corresponding to tumor recurrence and establish a novel FLLN for the subpopulation trajectories. Moreover, we derive asymptotic results of three key clinical biomarkers: recurrence time, clonal diversity, and pre-existing resistant clone sizes. These results enable the construction of consistent estimators for key parameters, including growth rates, mutation rates, and initial resistant population size.

The remainder of this paper is organized as follows. In Section \ref{sec:model}, we introduce the mathematical model for tumor evolution under therapeutic pressure, including trajectory representations of density-dependent birth–death processes for sensitive and resistant cell populations, their deterministic ODE approximations, and formal definitions of key biological and mathematical quantities such as recurrence time and clonal diversity metrics. In Section \ref{sec:theoretical results}, we present our main theoretical results: the asymptotic analysis of the deterministic system (Section \ref{sec:theo_determ}), functional law of large numbers results for population size trajectories and related quantities up to the time of tumor recurrence (Section \ref{sec:theo_stoch}), and the construction of consistent estimators for key parameters (Section \ref{sec:theo_estima}). In Section \ref{sec:simu}, we conduct numerical studies to corroborate our theoretical findings and assess the finite-sample properties and robustness of the proposed estimators.

% The remainder of this paper is organized as follows. Section \ref{sec:model} introduces the mathematical model of tumor evolution under drug treatment, including trajectory representations of density-dependent birth–death processes for sensitive and resistant cells, deterministic ODE approximations, and definitions of key quantities such as recurrence times and clonal metrics. Section \ref{sec:theoretical results} presents our main theoretical results: the asymptotic behavior of the deterministic system (Section \ref{sec:theo_determ}), law of large numbers results for stochastic trajectories and quantities up to recurrence (Section \ref{sec:theo_stoch}), and the construction of consistent estimators for intratumoral parameters (Section \ref{sec:theo_estima}). Finally, Section \ref{sec:simu} reports simulation studies that validate the convergence of stochastic processes to their ODE counterparts and evaluate the performance and robustness of the proposed estimators.

% \section{Mathematical Model and Quantities Definition}
% \label{sec:model}
\section{Model}
\label{sec:model}
We propose a stochastic model to describe the evolutionary dynamics of a tumor under therapeutic pressure. The tumor population is composed of two distinct cell subpopulations: \textit{sensitive cells} and \textit{resistant cells}. Let \( Z_0(t) \) and \( Z_1(t) \) denote the population sizes of sensitive and resistant cells at time \( t \), respectively. We assume the tumor is initially dominated by sensitive cells, accompanied by a small population of pre-existing resistant cells. The initial conditions are given by \( Z_0(0) = n \) and \( Z_1(0) = n^{\beta} \), with \( 0 < \beta < 1 \). 

We model the population dynamics through continuous-time birth-death processes. Each sensitive cell proliferates at a birth rate of \( r_0 \) and dies at a death rate of \( d_0 \), yielding a net growth rate \( \lambda_0 := r_0 - d_0 <0\). Each sensitive cell also gives birth to a resistant cell and a sensitive cell at a mutation rate which follows a power law, \( n^{-\alpha} \), where \( \alpha \in (0,1) \) \cite{brouard2023genetic}. Each resistant cell proliferates at a state-dependent birth rate, modulated by population size relative to the carrying capacity, and dies at a death rate of \( d_1 \). Specifically, the carrying capacity is defined as \( K(n) = k n \), where \( k > 1 \) is a fixed constant. The birth rate of resistant cells is denoted as \( f(Z_0/K, Z_1/K) \). The net growth rate of resistant cells is then \( \phi(Z_0/K, Z_1/K) = f(Z_0/K, Z_1/K) - d_1 \). We define \( r_1 = f(0, 0) \) as the intrinsic birth rate of resistant cells in the absence of competitive pressures, which yields an intrinsic net growth rate of \( \lambda_1 = r_1 - d_1 \). For notational convenience, we let \( K = K(n) \), \( Z(t) = (Z_0(t), Z_1(t)) \), and introduce the normalized process \( X(t) = (X_0(t), X_1(t)) \), with \( X_0(t):= Z_0(t)/K \) and \( X_1(t):= Z_1(t)/K \).

Following Chapter 2.4 of \cite{bansaye2015stochastic}, the system dynamics admit the following trajectorial representation:
% By Chapter 2.4 of \cite{bansaye2015stochastic}, we have the following trajectorial representation of the described birth-death processes:
\begin{align}
    Z_0(t) &= Z_0(0) 
    + \int_0^t \int_0^{\infty} \mathbbm{1}_{\{u \leq Z_0(s-) r_0 \}} \mathcal{N}_0^b(ds, du)
    - \int_0^t \int_0^{\infty} \mathbbm{1}_{\{u \leq Z_0(s-) d_0 \}} \mathcal{N}_0^d(ds, du), \\
    Z_1(t) &= Z_1(0) 
    + \int_0^t \int_0^{\infty} \mathbbm{1}_{\{u \leq Z_1(s-) f(Z(s-)/K) \}} \mathcal{N}_1^b(ds, du)
    - \int_0^t \int_0^{\infty} \mathbbm{1}_{\{u \leq Z_1(s-) d_1 \}} \mathcal{N}_1^d(ds, du) \\
    &\quad + \int_0^t \int_0^{\infty} \mathbbm{1}_{\{u \leq Z_0(s-) n^{-\alpha} \}} \mathcal{N}_0^m(ds, du),
\end{align}
where \( \mathcal{N}_\bullet^\bullet(ds, du) \) are independent Poisson point measures with Lebesgue measure intensity.
% Correspondingly, the normalized system is:
Similarly, the dynamics of the normalized system are governed by:
\begin{align}
\label{SDE:X0}
    X_0(t) &= X_0(0)
    + \frac{1}{K} \int_0^t \int_0^{\infty} \mathbbm{1}_{\{u \leq K X_0(s-) r_0 \}} \mathcal{N}_0^b(ds, du)
    - \frac{1}{K} \int_0^t \int_0^{\infty} \mathbbm{1}_{\{u \leq K X_0(s-) d_0 \}} \mathcal{N}_0^d(ds, du), \\
    \label{SDE:X1}
    X_1(t) &= X_1(0)
    + \frac{1}{K} \int_0^t \int_0^{\infty} \mathbbm{1}_{\{u \leq K X_1(s-) f(X(s-)) \}} \mathcal{N}_1^b(ds, du) \\
    &\quad - \frac{1}{K} \int_0^t \int_0^{\infty} \mathbbm{1}_{\{u \leq K X_1(s-) d_1 \}} \mathcal{N}_1^d(ds, du)
    + \frac{1}{K} \int_0^t \int_0^{\infty} \mathbbm{1}_{\{u \leq K X_0(s-) n^{-\alpha} \}} \mathcal{N}_0^m(ds, du).\nonumber
\end{align}
Given the important role of pre-existing resistant cells in determining treatment response and evolutionary dynamics, we isolate these cells and their progeny from the overall resistant population. We denote their population process by \( Z_{\beta}(t) \), which is governed by the stochastic differential equation:
% We define the initial clone process \( Z_{\beta}(t) \) through the following stochastic differential equation:
\begin{align}
    Z_{\beta}(t) 
    &= Z_{\beta}(0) 
    + \int_0^t \int_0^{\infty} \mathbbm{1}_{\{ u \leq Z_{\beta}(s-) f(Z(s-)/K) \}} \mathcal{N}_1^b(ds, du)
    - \int_0^t \int_0^{\infty} \mathbbm{1}_{\{ u \leq Z_{\beta}(s-) d_1 \}} \mathcal{N}_1^d(ds, du),
\end{align}
where \( Z_{\beta}(0) = n^{\beta}\) is the initial population size of pre-existing resistant cells. 
% Consequently, \( Z_{\beta}(t) \) denotes the population size of this initial clone at time \( t \). 
The corresponding normalized process, defined as 
\( X_{\beta}(t) = Z_{\beta}(t)/K \), evolves according to the dynamics governed by the following stochastic differential equation:
% The corresponding normalized process \( X_{\beta}(t) = Z_{\beta}(t)/K \) evolves according to:
\begin{align}
\label{SDE:Xbeta}
    X_{\beta}(t) 
    &= X_{\beta}(0) 
    + \frac{1}{K} \int_0^t \int_0^{\infty} \mathbbm{1}_{\{ u \leq K X_{\beta}(s-) f(X(s-)) \}} \mathcal{N}_1^b(ds, du) \\
    &\quad - \frac{1}{K} \int_0^t \int_0^{\infty} \mathbbm{1}_{\{ u \leq K X_{\beta}(s-) d_1 \}} \mathcal{N}_1^d(ds, du).
\end{align}
To facilitate our analysis, we introduce the auxiliary processes
\[
Z_m(t) = Z_1(t) - Z_{\beta}(t), \quad X_m(t) = X_1(t) - X_{\beta}(t),
\]
which represent the population of resistant cells excluding the initial pre-existing clone. Biologically, \( Z_m(t) \) corresponds to resistant subclones originating from mutations acquired from sensitive cells after treatment initiation.
% so that \( Z_m(t) \) and \( X_m(t) \) represent the remaining population of resistant cells beyond the initial clone. Biologically, \( Z_m(t) \) denotes resistant cells present in clones initiated by mutations from sensitive cells.

We define the associated deterministic ordinary differential equation (ODE) system, which approximates the dynamics of the stochastic system under consideration, as follows:
% We define the deterministic ODE system approximating the stochastic system above  as follows:
\begin{equation}
\label{ODEs}
    \begin{cases}
        \dot{y}_0(t) &= \lambda_0 \cdot y_0(t), \\[6pt]
        \dot{y}_1(t) &= \phi( y(t)) \cdot y_1(t) + n^{-\alpha} \cdot y_0(t), \\[6pt]
        \dot{y}_{\beta}(t) &= \phi( y(t)) \cdot y_{\beta}(t),
    \end{cases}
\end{equation}
where \( y(t) = (y_0(t), y_1(t)) \) with initial condition \((y_0(0),y_1(0),y_{\beta}(0)) = (n/K, n^{\beta}/K,n^{\beta}/K)\).

It is well established \cite{ethier2009markov} that in the absence of mutations (i.e., when \(\alpha = \infty\)), the normalized processes \( X_0(t) \) and \( X_1(t) \) converge almost surely to their deterministic counterparts \( y_0(t) \) and \( y_1(t) \), respectively, as \( n \to \infty \) on any finite time interval. In this work, we consider a more biologically realistic scenario where mutations occur at a rate following a power law with exponent \(\alpha \in (0,1)\). Furthermore, rather than examining deterministic finite time horizons, we analyze stopping times corresponding to tumor recurrence, specifically, the first time at which the resistant cell population reaches the initial tumor size:
% It is well established by \cite{ethier2009markov} that \( X_0(t) \) and \( X_1(t) \) converge almost surely to \( y_0(t) \) and \( y_1(t) \), respectively, as \( n \to \infty \) on any finite time interval when mutations are absent (i.e., \(\alpha = \infty\)). However, we consider a more realistic scenario with mutations present, following a power law and assuming \(\alpha \in (0,1)\). Also, instead of a deterministic finite time \( t \), we consider stopping times representing tumor recurrence, when the tumor regrows to an observable level:
\begin{align}
\label{def:zeta}
    \zeta_n &:= \inf\left\{ t > 0 : y_1(t) = \frac{n}{K} \right\}, \\[6pt]
    \label{def:gamma}
    \gamma_n &:= \inf\left\{ t > 0 : X_1(t) = \frac{n}{K} \right\}.
\end{align}

Furthermore, under the infinite-sites model, we assume that each mutation event from sensitive cells gives rise to a distinct lineage (clone) of resistant cells characterized by a unique genotype. Recent advances in genomic sequencing technologies enable the detection and quantification of such distinct resistant clones. In this work, we aim to characterize the number of surviving resistant clones at tumor recurrence. We therefore define the following quantity:
% Furthermore, employing an infinite-sites model, each mutation from sensitive cells initiates a lineage of resistant cells possessing a unique genotype. We refer to groups of resistant cells sharing an identical genotype as clones. Advances in gene sequencing technologies now allow us to detect and quantify distinct surviving clones. Specifically, we wish to investigate the number of surviving resistant clones at critical times, particularly at tumor recurrence. To quantify this, we define:
\[
 I_n(t) := \int_0^t \int_0^{\infty} \mathbbm{1}_{\{B_s(t)>0\}} \, \mathbbm{1}_{\{ u \leq K X_0(s-) n^{-\alpha} \}} \mathcal{N}_0^m(ds, du),
\]
where \( B_s(t) \) denotes the population size at time \( t \) of the resistant clone originating from a mutation at time \( s \). Thus, \( I_n(t) \) corresponds to the number of resistant clones that have survived until time \( t \).

The goal of this work is to construct estimators for key evolutionary parameters, including the growth rates \(\lambda_0\), \(\lambda_1\), the mutation power-law exponent \(\alpha\), and the initial resistant fraction exponent \(\beta\), from observables such as recurrence time \(\gamma_n\), the number of surviving resistant clones \(I_n(\gamma_n)\), and population sizes \(Z_0(\gamma_n)\) and \(Z_{\beta}(\gamma_n)\). These quantities can be derived from gene sequencing data and medical imaging (e.g., CT scans) using state-of-the-art computational methods.
% With advancements in modern technologies such as gene sequencing, CT scans, and computational tools like PyClone for inferring clonal structures in tumors, clinical observations including \(\gamma_n\), \(I_n(\gamma_n)\), \(Z_0(\gamma_n)\), and \(Z_{\beta}(\gamma_n)\) become available upon tumor recurrence. Ultimately, our goal is to utilize this observable data to estimate key intratumoral parameters characterizing cancer progression, including \(\lambda_0\), \(\lambda_1\), \(\alpha\), and \(\beta\).
Before presenting our main results, we specify the assumptions on the density-dependent birth rate function \( f(x,y) \) to ensure analytical tractability.

\begin{assumption}
\label{assump:f}
\leavevmode
\begin{enumerate}[label=(A\arabic*), itemsep=4pt]
    % \item The function \( f(x,y) \) is defined on \( \mathbb{R}^+ \times \mathbb{R}^+ \) and is Lipschitz continuous with respect to both variables \( x \) and \( y \).
    \item The function \( f: \mathbb{R}^+ \times \mathbb{R}^+ \rightarrow  \mathbb{R}^+ \) is Lipschitz continuous in both variables.
    
    % \item Boundary conditions are specified by \( f(x,y) = r_1 \) if \( x+y = 0 \), and \( f(x,y) = d_1 \) if \( x+y = 1 \).

    \item The function $f$ satisfies the boundary conditions \( f(x,y) = r_1 \) when \( x+y = 0 \), and \( f(x,y) = d_1 \) when \( x+y = 1 \).
    
    % \item There exists a non-increasing function \( \Phi(z): \mathbb{R}^+ \to \mathbb{R}^+ \) satisfying \( \frac{d\Phi}{dz} \leq 0 \), such that \( f(x,y) = \Phi(x+y) \).
    \item There exists a non-increasing function \( \Phi(z): \mathbb{R}^+ \to \mathbb{R}^+ \) such that \( \frac{d\Phi}{dz} \leq 0 \) and \( f(x,y) = \Phi(x+y) \).
    
    \item The birth rate function vanishes at infinity: \( \lim\limits_{x \to \infty} f(x,y) = \lim\limits_{y \to \infty} f(x,y) = 0 \).
    
    \item The birth rate function admits the lower bound \( f(x,y) \geq \lambda_1\left(1 - (x+y) \right) + d_1 \).
\end{enumerate}
\end{assumption}

We note that the class of generalized logistic growth functions, defined as
\[
f(x,y) = \lambda_1\left(1 - (x+y)^{\nu}\right) + d_1, \quad \nu\geq 1,
\]
satisfies the conditions specified in Assumption~\ref{assump:f}.

\section{Theoretical Results}
\label{sec:theoretical results}
\subsection{Asymptotic Behavior of the Deterministic System}
\label{sec:theo_determ}
% Before analyzing the behavior of the stochastic processes, we begin by examining the ODE system \eqref{ODEs} and the stopping time definition \eqref{def:zeta} to quantify the deterministic values of \( \zeta_n \) and \( y_{\beta}(\zeta_n) \). We first characterize the asymptotic behavior of \( \zeta_n \):
Before analyzing the stochastic system, we first examine the deterministic counterparts given by the ODE system \eqref{ODEs} and the stopping time \eqref{def:zeta}. Our objective is to characterize the asymptotic behavior of \( \zeta_n \) and \( y_{\beta}(\zeta_n) \).

\begin{proposition}
\label{prop:zeta}
% As \( n \to \infty \), the deterministic recurrence time satisfies:
In the large population limit, the scaled deterministic recurrence time converges to:
\[
\lim_{n \to \infty} \frac{\zeta_n}{\log n} = \frac{\min\left\{ 1 - \beta, \alpha \right\}}{\lambda_1}.
\]
\end{proposition}

\begin{proof}
See Section~\ref{sec:pf_prop_zeta}.
\end{proof}

Next, we examine the asymptotic behavior of \( y_{\beta}(\zeta_n) \).
% Next, we consider the asymptotic behavior of \( y_{\beta}(\zeta_n) \):

\begin{proposition}
\label{prop:ybeta}
As \( n \to \infty \), the solution \( y_{\beta}(\zeta_n) \) of the ODE system \eqref{ODEs} satisfies:
\[
\lim_{n \to \infty} \frac{\log \log\left( \frac{n}{K y_{\beta}(\zeta_n)} \right)}{\log n} = 1 - \alpha - \beta.
\]
\end{proposition}

\begin{proof}
See Section~\ref{sec:pf_prop_ybeta}.
\end{proof}
\subsection{Asymptotic Behavior of the Stochastic System}
\label{sec:theo_stoch}
We now present our main convergence results. Specifically, we establish that the ratio between the solutions of the stochastic differential equations~\eqref{SDE:X0}, \eqref{SDE:X1}, \eqref{SDE:Xbeta} and their deterministic counterparts~\eqref{ODEs} converges uniformly to $1$ in probability over the time interval $[0, \zeta_n+\delta]$, for any fixed constant $\delta>0$.

% Now we are ready to present our main convergence results. Specifically, we demonstrate that in probability, the ratio between the solutions of the stochastic differential equations~\eqref{SDE:X0}, \eqref{SDE:X1}, \eqref{SDE:Xbeta} and their deterministic ODE counterparts~\eqref{ODEs} converges uniformly to 1 before the deterministic time \( \zeta_n \) plus an arbitrary constant.

\begin{theorem}
\label{thm:ratio_conv}
Let \( \epsilon, \delta > 0 \). Suppose \( \beta > 1 + \frac{\lambda_1}{\lambda_0} \) and \( \alpha + \beta > 1 \). Then, for any \( u_1 < \beta/2 \), \( u_2 < \min\{ \beta/2, \alpha + \beta - 1 \} \), we have:
\begin{align}
    &\lim_{n \to \infty} \mathbb{P} \left( \sup_{t \leq \zeta_n + \delta} \left| \frac{X_0(t)}{y_0(t)} - 1 \right| > \epsilon \right) = 0, \\[6pt]
    &\lim_{n \to \infty} \mathbb{P} \left( n^{u_1} \sup_{t \leq \zeta_n + \delta} \left| \frac{X_1(t)}{y_1(t)} - 1 \right| > \epsilon \right) = 0, \\[6pt]
    &\lim_{n \to \infty} 
    \mathbb{P} \left( 
        n^{u_2} \sup_{t \leq \zeta_n+\delta} \left| \frac{X_{\beta}(t)}{y_{\beta}(t)} - 1 \right| > \epsilon 
    \right) = 0.
\end{align}
\end{theorem}

\begin{proof}
See Section~\ref{sec:pf_thm_ratio_conv}.
\end{proof}

The parameters \( u_1 \) and \( u_2 \) in Theorem~\ref{thm:ratio_conv} govern the convergence rates of the ratios \( X_1(t)/y_1(t) \) and \( X_{\beta}(t)/y_{\beta}(t) \), respectively. Larger values of \( u_1 \) and \( u_2 \) correspond to faster convergence. The condition \( \beta > 1 + \lambda_1 / \lambda_0 \) ensures the persistence of sensitive cells at recurrence time \( \zeta_n \),  which is biologically supported by clinical observations that sensitive cells often remain detectable upon relapse \cite{chaft2011disease,piotrowska2015heterogeneity, ichihara2015shades}.

% This is biologically plausible, as studies have shown that sensitive cells may still constitute a portion of the tumor population upon relapse \cite{chaft2011disease,piotrowska2015heterogeneity, ichihara2015shades}.

The second condition, \( \alpha + \beta > 1 \), is biologically plausible given that mutation events are typically rare, often resulting in values of \( \alpha \) close to $1$. This inequality admits a natural biological interpretation: the parameter \( \beta \), governing the initial size of the resistant population, reflects the system's intrinsic \textit{stability}, while \( 1 - \alpha \), representing the intensity of mutations from sensitive to resistant cells, introduces external \textit{variability}. For the sample paths of the stochastic system to remain uniformly close to the deterministic trajectories over the relevant time scale, the inherent stability of the resistant population must exceed the variability introduced by mutations. Thus, the condition \( \alpha + \beta > 1 \) ensures that the stochastic fluctuations arising from mutations do not disrupt the mean-field dynamics dictating the system’s long-term behavior.

% The second condition, \( \alpha + \beta > 1 \), is also mild in biological contexts. Since mutation events are typically rare, \( \alpha \) is often close to 1. To better understand this condition, observe that \( \beta \), which governs the size of the initial resistant cell population, can be interpreted as a measure of the system of resistant cells' intrinsic \textit{stability}. In contrast, \( 1 - \alpha \), which quantifies the mutation intensity from sensitive to resistant cells, introduces external \textit{variability}. For the stochastic trajectories to closely track their deterministic counterparts, the system's inherent stability must dominate this variability. Hence, the inequality \( \alpha + \beta > 1 \) guarantees that the noise introduced by mutations does not overwhelm the  growth of the deterministic system. 

Theorem~\ref{thm:ratio_conv} establishes a strong asymptotic equivalence between the deterministic and stochastic systems, thereby justifying the use of deterministic trajectories as approximations for analyzing key stochastic quantities. A direct implication of this result is the convergence of the stochastic recurrence time \( \gamma_n \) to its deterministic counterpart \( \zeta_n \).

% Theorem~\ref{thm:ratio_conv} establishes a strong correspondence between the deterministic and stochastic systems, which motivates the use of deterministic trajectories as proxies for analyzing stochastic quantities. One immediate consequence is the convergence of the stochastic recurrence time \( \gamma_n \) to its deterministic counterpart \( \zeta_n \).

\begin{proposition}
\label{prop:convergence_of_gamma_n}
Let \( \gamma_n \) and \( \zeta_n \) be defined as in~\eqref{def:gamma} and~\eqref{def:zeta}, respectively. Then, under the condition \( \alpha + \beta > 1 \), for any \( \epsilon > 0 \) and \( u < \beta/2 \),
\[
\lim_{n \to \infty} \mathbb{P}\left( n^u |\gamma_n - \zeta_n| > \epsilon \right) = 0.
\]
\end{proposition}

\begin{proof}
See Section~\ref{sec:pf_prop_gamma_conv}.
\end{proof}

\vspace{1em}
We now focus on characterizing the resistant population at recurrence. A key quantity is the number of distinct resistant clones present at time \( \gamma_n \), denoted $I_n(\gamma_n)$. The following result shows that $I_n(\gamma_n)$
scales polynomially with exponent \( 1 - \alpha \).

% We now turn our attention to characterizing the resistant population at recurrence. A key quantity of interest is the number of distinct resistant clones alive at time \( \gamma_n \). The following result establishes that this number grows polynomially with exponent \( 1 - \alpha \):

\begin{proposition}
\label{prop:bound_of_In}
There exist positive constants $c_I$ and $C_I$ such that
\[
\lim_{n \to \infty} \mathbb{P} \left( c_I n^{1 - \alpha} \leq I_n(\gamma_n) \leq C_I n^{1 - \alpha} \right) = 1.
\]
\end{proposition}

\begin{proof}
See Section~\ref{sec:pf_prop_bound_In}. 
\end{proof}

\vspace{1em}
% In addition to the number of clones, we are also interested in the size of the clone initiated by the pre-existing resistant population. The next proposition provides precise asymptotics for this quantity at the time of recurrence:
In addition to the number of resistant clones, we are also interested in the the size of the pre-existing resistant clone. The following proposition establishes the asymptotic behavior of this population at recurrence.

\begin{proposition}
\label{prop:conv_Xbeta}
There exist positive constants $c$ and $C$ such that
\[
\lim_{n \to \infty} 
\mathbb{P} \left( 
    c n^{1 - \alpha - \beta} 
    < -\log \left( \frac{Z_{\beta}(\gamma_n)}{n} \right) 
    < C n^{1 - \alpha - \beta} 
\right) = 1.
\]
\end{proposition}

\begin{proof}
See Section~\ref{sec:pf_prop_conv_Xbeta}. 
\end{proof}
\subsection{Construction of Estimators}
\label{sec:theo_estima}
In Section \ref{sec:theo_stoch}, we have characterized the asymptotic behavior of key stochastic quantities at tumor recurrence time 
\( \gamma_n \). Specifically, we have established convergence results for: (i) the number of distinct resistant clones \( I_n(\gamma_n) \), (ii) the size of the pre-existing resistant clone \( Z_{\beta}(\gamma_n) \)\footnote{By Proposition~\ref{prop:conv_Xbeta}, the pre-existing resistant clone is, with high probability, the largest resistant clone at recurrence, making it clinically tractable.}, and (iii) the recurrence time \( \gamma_n \) itself. To facilitate parameter estimation, we additionally incorporate \( Z_0(\gamma_n) \), whose asymptotic properties are well-established in prior work~\cite{leder2024parameter, leder2021clonal}. These results provide the theoretical foundation for constructing estimators of key evolutionary parameters. We now define estimators for \( \lambda_0 \), \( \lambda_1 \), \( \alpha \), and \( \beta \) as follows:
% With all the efforts made above, we have adequately described the behavior of several important quantities at the tumor recurrence time \( \gamma_n \). In particular, we focus on the number of alive clones \( I_n(\gamma_n) \), the size of the clone initiated by the pre-existing resistant population \( Z_{\beta}(\gamma_n) \), and the recurrence time \( \gamma_n \) itself. To construct the estimators, we also include \( Z_0(\gamma_n) \), whose behavior has been well studied in the literature~\cite{leder2024parameter, leder2021clonal}.
% Moreover, by Proposition~\ref{prop:conv_Xbeta}, \( Z_{\beta}(\gamma_n) \) is, with high probability, the largest resistant clone at recurrence. This makes it a biologically and computationally accessible quantity, just like \( I_n(\gamma_n) \), \( Z_0(\gamma_n) \), and \( \gamma_n \).
% Based on these observable quantities, we define our estimators for the parameters \( \lambda_0 \), \( \lambda_1 \), \( \alpha \), and \( \beta \) as follows:
\begin{align}
\label{def:alpha_hat}
    \hat{\alpha} &:= 1 - \log_n\left( I_n(\gamma_n) \right), \\
    \hat{\beta} &:= 1 - \hat{\alpha} - \frac{\log\log\left( \frac{n}{Z_{\beta}(\gamma_n)} \right)}{\log n}, \\
    \hat{\lambda}_0 &:= \frac{1}{\gamma_n} \log\left( \frac{Z_0(\gamma_n)}{n} \right), \\
\label{def:lambda1_hat}
    \hat{\lambda}_1 &:= \frac{1 - \hat{\beta}}{\gamma_n} \log n.
\end{align}

We now state our main statistical result regarding the consistency of the proposed estimators:
% We are now ready to present our main statistical result:
\begin{theorem}
\label{thm: estimators}
Suppose \( \beta > 1 + \frac{\lambda_1}{\lambda_0} \) and \( \alpha + \beta > 1 \). Then the estimators $\hat{\alpha}$, $\hat{\beta}$, $\hat{\lambda}_0$, and $\hat{\lambda}_1$ are consistent.
% defined above are consistent.
\end{theorem}
\begin{proof}
See Section~\ref{sec:pf_thm_estimators}. 
\end{proof}

\section{Simulation Results}
\label{sec:simu}
\subsection{Convergence of the Stochastic System}

In this section, we perform numerical simulations to validate Theorem~\ref{thm:ratio_conv}, which establishes the convergence of the stochastic system to its mean-field approximation. Specifically, we demonstrate that the normalized population processes \( X_0(t) = Z_0(t)/K \), \( X_1(t) = Z_1(t)/K \), and \( X_{\beta}(t) = Z_{\beta}(t)/K \) converge in probability to their deterministic counterparts  \( y_0(t), y_1(t), y_{\beta}(t) \), uniformly over the interval \( [0, \zeta_n + \delta] \) as \( n \to \infty \).

% We conduct numerical simulations to illustrate Theorem~\ref{thm:ratio_conv}, which establishes convergence between the stochastic model and its mean-field approximation. In particular, we verify that the normalized processes \( X_0(t) = Z_0(t)/K \), \( X_1(t) = Z_1(t)/K \), and \( X_{\beta}(t) = Z_{\beta}(t)/K \) converge in probability to the deterministic ODE solutions \( y_0(t), y_1(t), y_{\beta}(t) \) uniformly on the interval \( [0, \zeta_n + \delta] \) as \( n \to \infty \).

We simulate the stochastic system using the Gillespie algorithm, which generates exact realizations of the event sequence (e.g., birth, death, mutation) and their precise occurrence times according to the model defined in Section~\ref{sec:model}. For the birth rate function, we employ a logistic growth form $f(x,y)=\lambda_1 \left(1-\left(x+y\right)\right)+d_1$. Mutations from sensitive to resistant cells occur at a rate of \( n^{-\alpha} Z_0(t) \). The recurrence time \( \gamma_n \) is recorded when the resistant population \( Z_1(t) \) reaches the initial tumor burden \( n \). In parallel, we numerically solve the ODE system~\eqref{ODEs} using the Runge–Kutta 45 (RK45) method to obtain the deterministic trajectories \( y_0(t), y_1(t), y_{\beta}(t) \).

% We use Gillespie Algorithm to simulate the exact sequence of events (e.g., birth, death, mutation) and the precise times at which they occur of the model described in Section \ref{sec:model}. We specify the growth function $f(x,y)$ to be logistic growth. Mutation from sensitive to resistant cells occurs continuously at rate \( n^{-\alpha} Z_0(t) \). The recurrence time \( \gamma_n \) is recorded when the resistant population \( Z_1(t) \) reaches the initial tumor burden \( n \). In parallel, we obtain the trajectories \( y_0(t), y_1(t), y_{\beta}(t) \) through solving the ODE system~\eqref{ODEs} numerically via a Runge–Kutta 45 (RK45) scheme to obtain the trajectories \( y_0(t), y_1(t), y_{\beta}(t) \).

% The stochastic system is simulated using a discrete-event implementation of the density-dependent birth-death processes defined in Section~\ref{sec:model}. We specify the growth function $f(x,y)$ to be logistic growth. Mutation from sensitive to resistant cells occurs continuously at rate \( n^{-\alpha} Z_0(t) \). The recurrence time \( \gamma_n \) is recorded when the resistant population \( Z_1(t) \) reaches the initial tumor burden \( n \). In parallel, the ODE system~\eqref{ODEs} is solved numerically via a Runge–Kutta 45 (RK45) scheme to obtain the trajectories \( y_0(t), y_1(t), y_{\beta}(t) \).

Figure~\ref{fig:thm1} compares stochastic and deterministic trajectories for increasing system sizes \( n = 10^3, 10^4, 10^5, 10^6 \). Solid lines depict the stochastic trajectories \( Z_0, Z_1, Z_{\beta} \), while dashed lines represent the scaled deterministic solutions \( K y_0, K y_1, K y_{\beta} \). As $n$ increases, stochastic fluctuations diminish and the trajectories converge uniformly to their deterministic counterparts, validating the convergence result established in Theorem~\ref{thm:ratio_conv}.

% Figure~\ref{fig:thm1} displays stochastic and deterministic trajectories across increasing system sizes \( n = 10^3, 10^4, 10^5, 10^6 \). Solid lines represent the raw stochastic trajectories \( Z_0, Z_1, Z_{\beta} \), while dashed lines represent the ODE solutions \( K y_0, K y_1, K y_{\beta} \). As \( n \) increases, stochastic variability decreases and convergence to the deterministic solutions becomes visually evident, confirming the uniform convergence result of Theorem~\ref{thm:ratio_conv}.

\begin{figure}[ht!]
\centering
\includegraphics[width=0.48\linewidth]{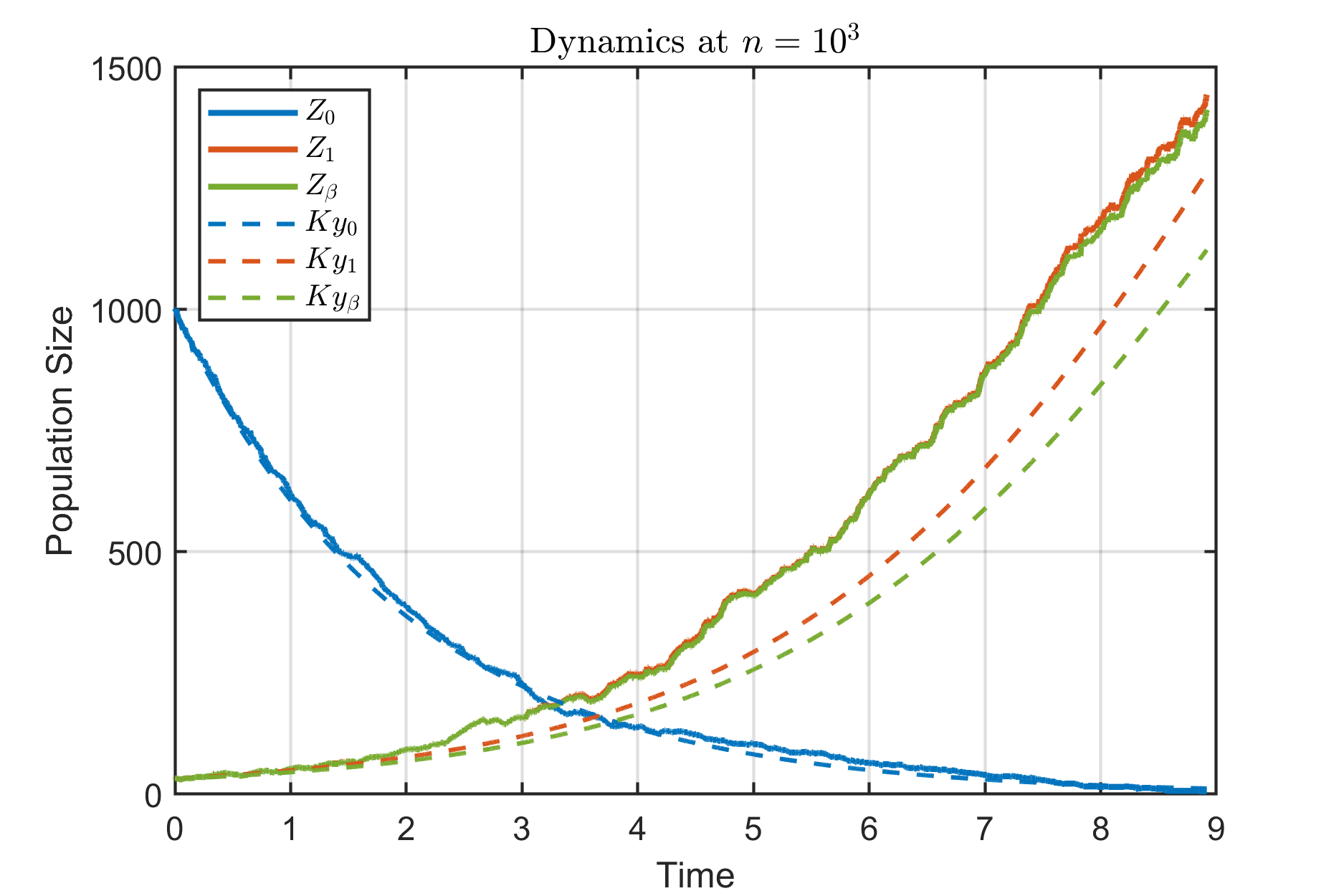}
\includegraphics[width=0.48\linewidth]{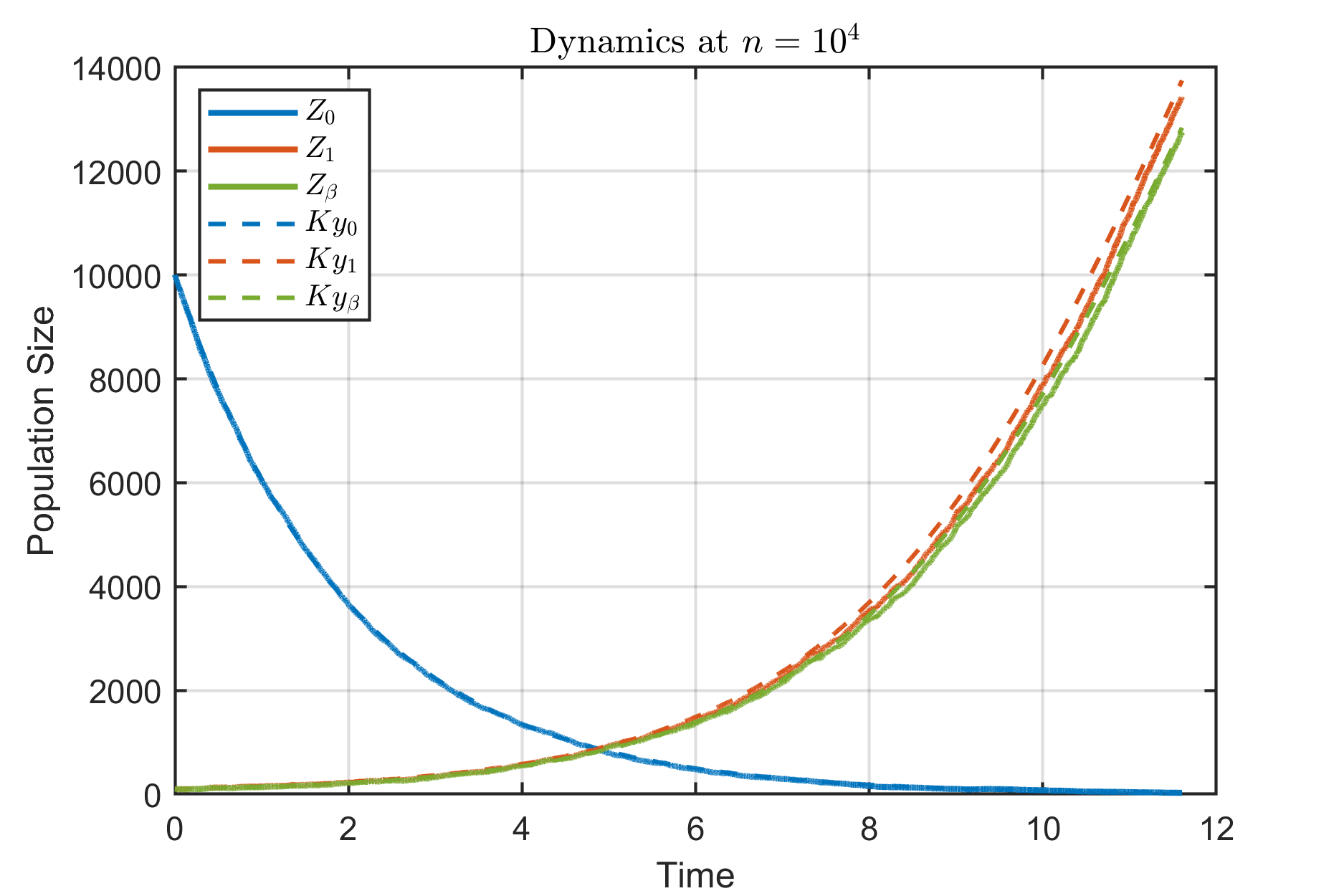} \\
\vspace{2mm}
\includegraphics[width=0.48\linewidth]{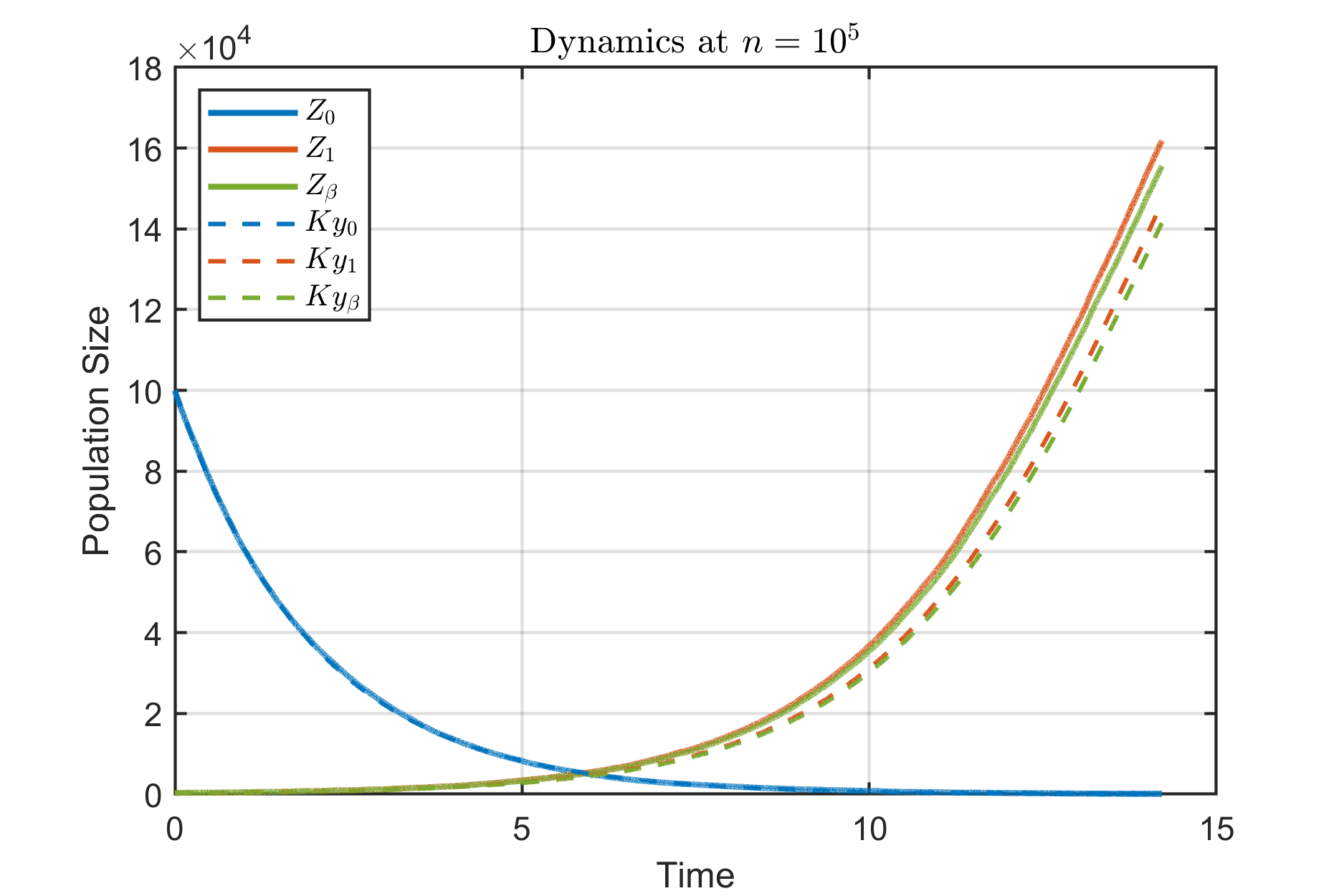}
\includegraphics[width=0.48\linewidth]{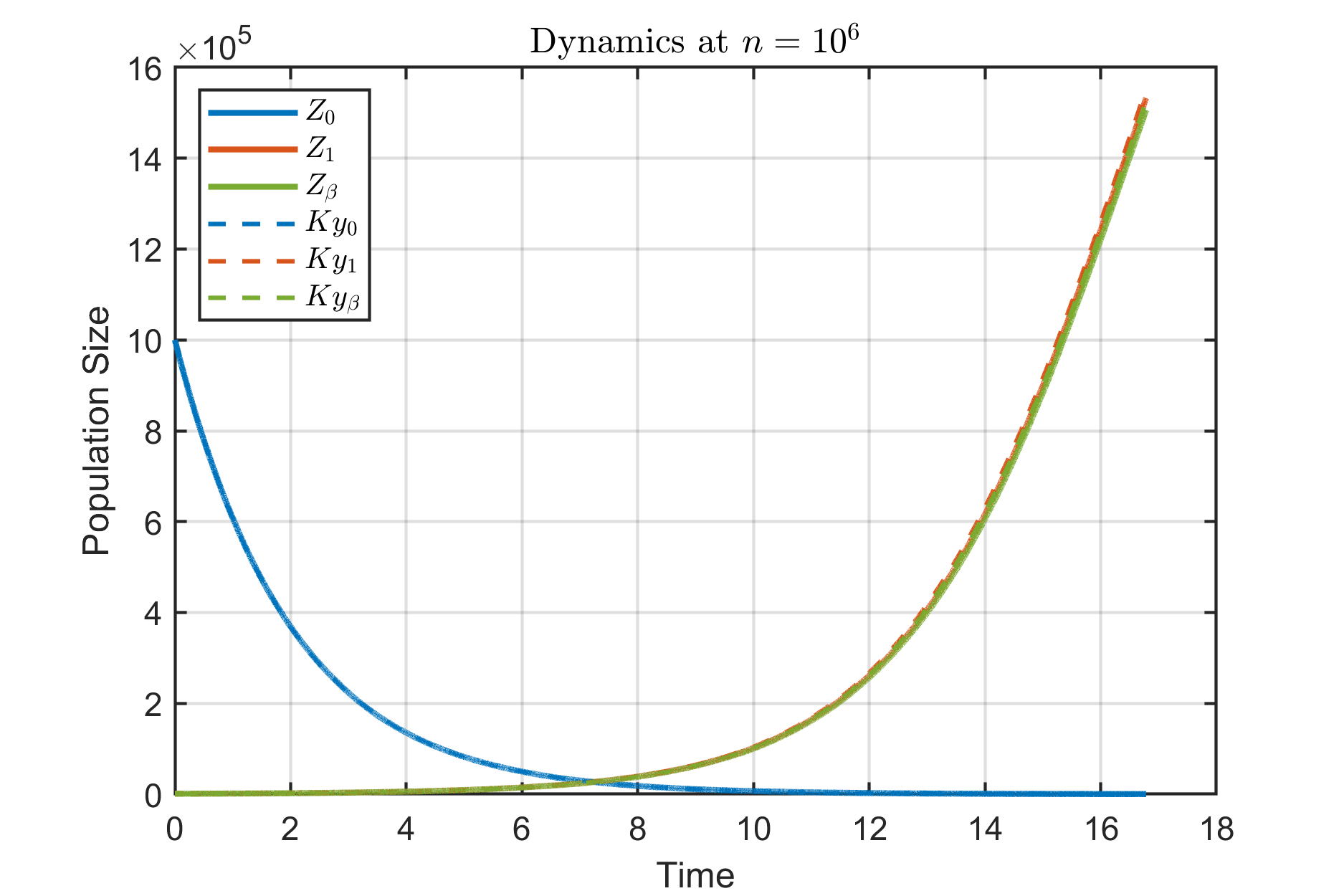}
% \caption{Simulated tumor dynamics for increasing system sizes \( n = 10^3, 10^4, 10^5, 10^6 \). Parameters: \( \alpha = 0.8 \), \( \beta = 0.5 \), \( \lambda_0 = -0.5 \), \( \lambda_1 = 0.5 \), \( k = 3 \). Solid lines: stochastic processes \( Z_0, Z_1, Z_{\beta} \). Dashed lines: corresponding ODE solutions \( K y_0, K y_1, K y_{\beta} \). As \( n \to \infty \), stochastic fluctuations diminish and convergence is observed.}
\caption{Simulated tumor dynamics under therapeutic pressure for increasing system sizes \( n = 10^3, 10^4, 10^5, 10^6 \). Parameter values: \( \alpha = 0.8 \), \( \beta = 0.5 \), \( \lambda_0 = -0.5 \), \( \lambda_1 = 0.5 \), \( k = 3 \). Solid lines represent stochastic trajectories ($Z_0$: sensitive cells, $Z_1$: total resistant cells, $Z_{\beta}$: pre-existing resistant clone). Dashed lines show corresponding scaled deterministic solutions (\( K y_0, K y_1, K y_{\beta} \)). As \( n \to \infty \), stochastic fluctuations diminish and trajectories converge uniformly to their deterministic limits.}
\label{fig:thm1}
\end{figure}

\subsection{Consistency of the Proposed Estimators}
Using the same parameter values as in Figure~\ref{fig:thm1}, we perform $10$ simulations for each system size $n$. At each stochastic recurrence time \( \gamma_n \), we record three key quantities: the number of surviving resistant clones \( I_n(\gamma_n) \), the sensitive cell population size \( Z_0(\gamma_n) \), and the size of the pre-existing resistant clone \( Z_{\beta}(\gamma_n) \). Following the estimator definitions in equations~\eqref{def:alpha_hat}–\eqref{def:lambda1_hat}, we compute the corresponding parameter estimates $\hat{\alpha}$, $\hat{\beta}$, $\hat{\lambda}_0$, and $\hat{\lambda}_1$ for each simulation trial. To quantify estimator accuracy, we compute the relative error for each parameter estimate as follows:
\[
\frac{|\hat{\alpha} - \alpha|}{\alpha}, \quad \frac{|\hat{\beta} - \beta|}{\beta}, \quad \frac{|\hat{\lambda}_0 - \lambda_0|}{|\lambda_0|}, \quad \frac{|\hat{\lambda}_1 - \lambda_1|}{\lambda_1}.
\]
The mean and standard deviation of these relative errors are then computed across simulation trials and plotted against the system size \( n \) (equivalently, against the carrying capacity \( K = 3n \)).
% Under the same parameter settings, we repeat the simulation 10 times for each value of \( n \). At the stochastic recurrence time \( \gamma_n \), we record key quantities: the number of surviving resistant clones \( I_n(\gamma_n) \), the sensitive cell population \( Z_0(\gamma_n) \), and the size of the pre-existing resistant clone \( Z_{\beta}(\gamma_n) \). Using the definitions in equations~\eqref{def:alpha_hat}--\eqref{def:lambda1_hat}, we compute the corresponding estimators \( \hat{\alpha}, \hat{\beta}, \hat{\lambda}_0, \hat{\lambda}_1 \) for each trial.

% To evaluate estimator performance, we compute the relative error in each case as 
% \[
% \frac{|\hat{\alpha} - \alpha|}{\alpha}, \quad \frac{|\hat{\beta} - \beta|}{\beta}, \quad \frac{|\hat{\lambda}_0 - \lambda_0|}{|\lambda_0|}, \quad \frac{|\hat{\lambda}_1 - \lambda_1|}{\lambda_1}.
% \]
% We then plot the mean and standard deviation of these relative errors against system size \( n \) (equivalently \( K = 3n \)).

As shown in Figure~\ref{fig:thm2}, the mean relative error decreases systematically with increasing system size for all estimated parameters. At \( n = 10^7 \), the relative error plus one standard deviation remains below $10\%$ for all estimators and below $2\%$ for \( \lambda_0 \) and \( \lambda_1 \). Given that clinically observed tumors frequently reach sizes on the order of \( 10^9 \) cells or larger, these results indicate strong potential for practical applicability. Furthermore, the narrowing variability (shaded regions) with increasing $n$ provides empirical support for the theoretical consistency established in Theorem~\ref{thm: estimators}.

% As shown in Figure~\ref{fig:thm2}, the mean relative error decreases with increasing system size for all parameters. When \( n = 10^8 \), the relative error—plus one standard deviation—falls below 10\% for all estimators, and below 2\% for \( \lambda_0 \) and \( \lambda_1 \). Since clinically observed tumors often reach populations on the order of \( 10^9 \) cells and above, these results suggest strong practical reliability. Additionally, the narrowing shaded regions support the consistency of our estimators as established in Theorem~\ref{thm: estimators}.

\begin{figure}[ht!]
\centering
    \includegraphics[width=0.48\linewidth]{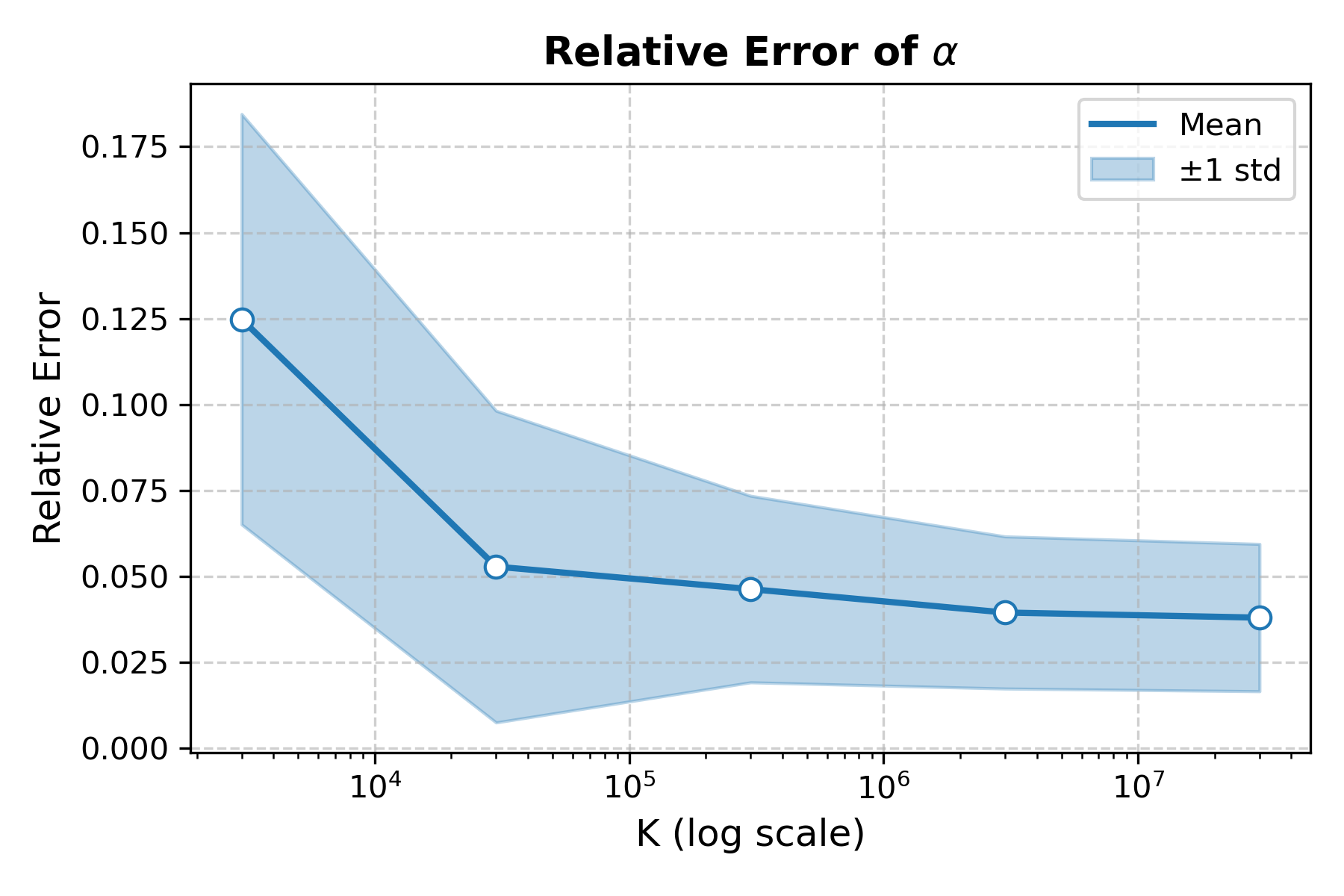}
    \includegraphics[width=0.48\linewidth]{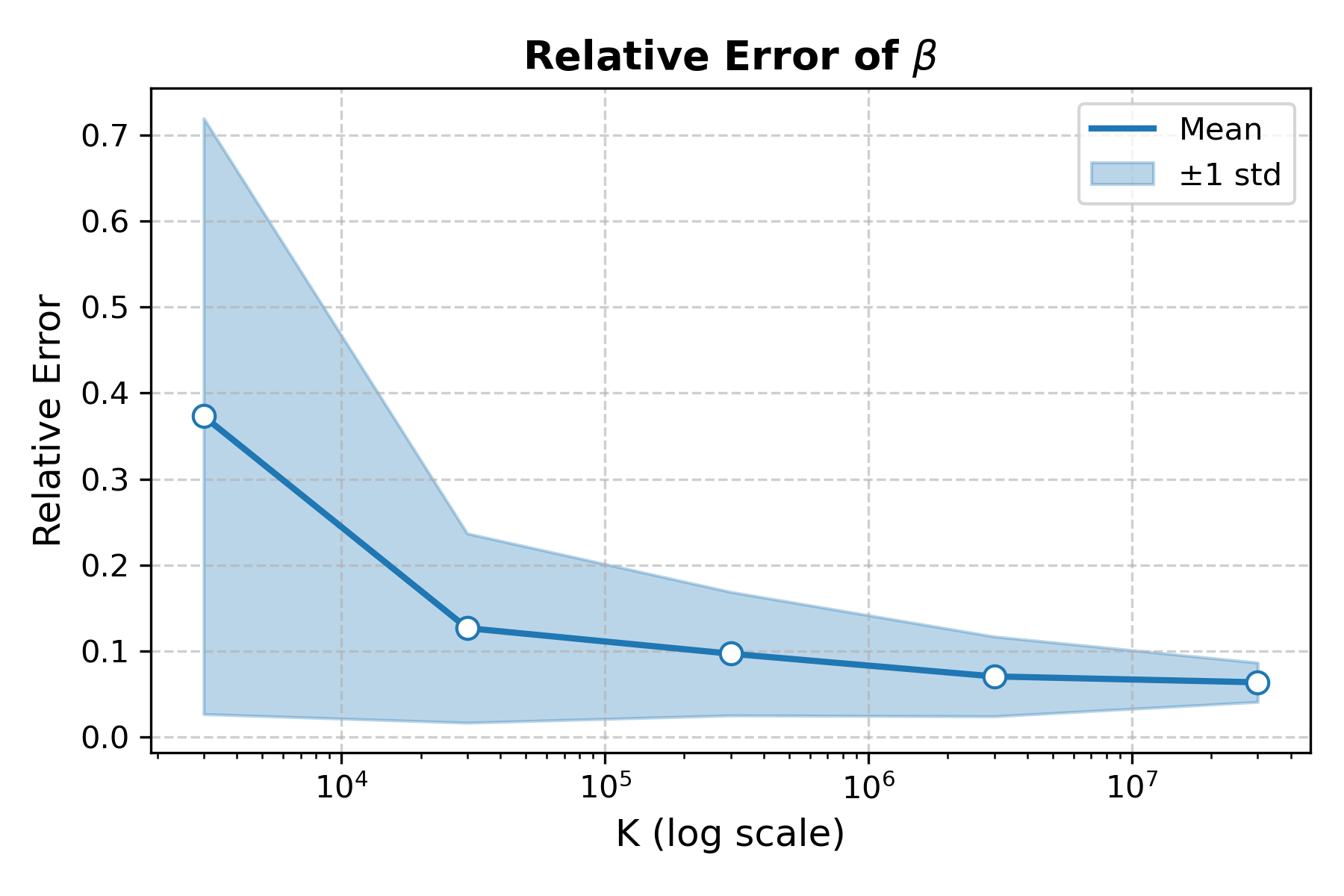} \\
    \vspace{2mm}
    \includegraphics[width=0.48\linewidth]{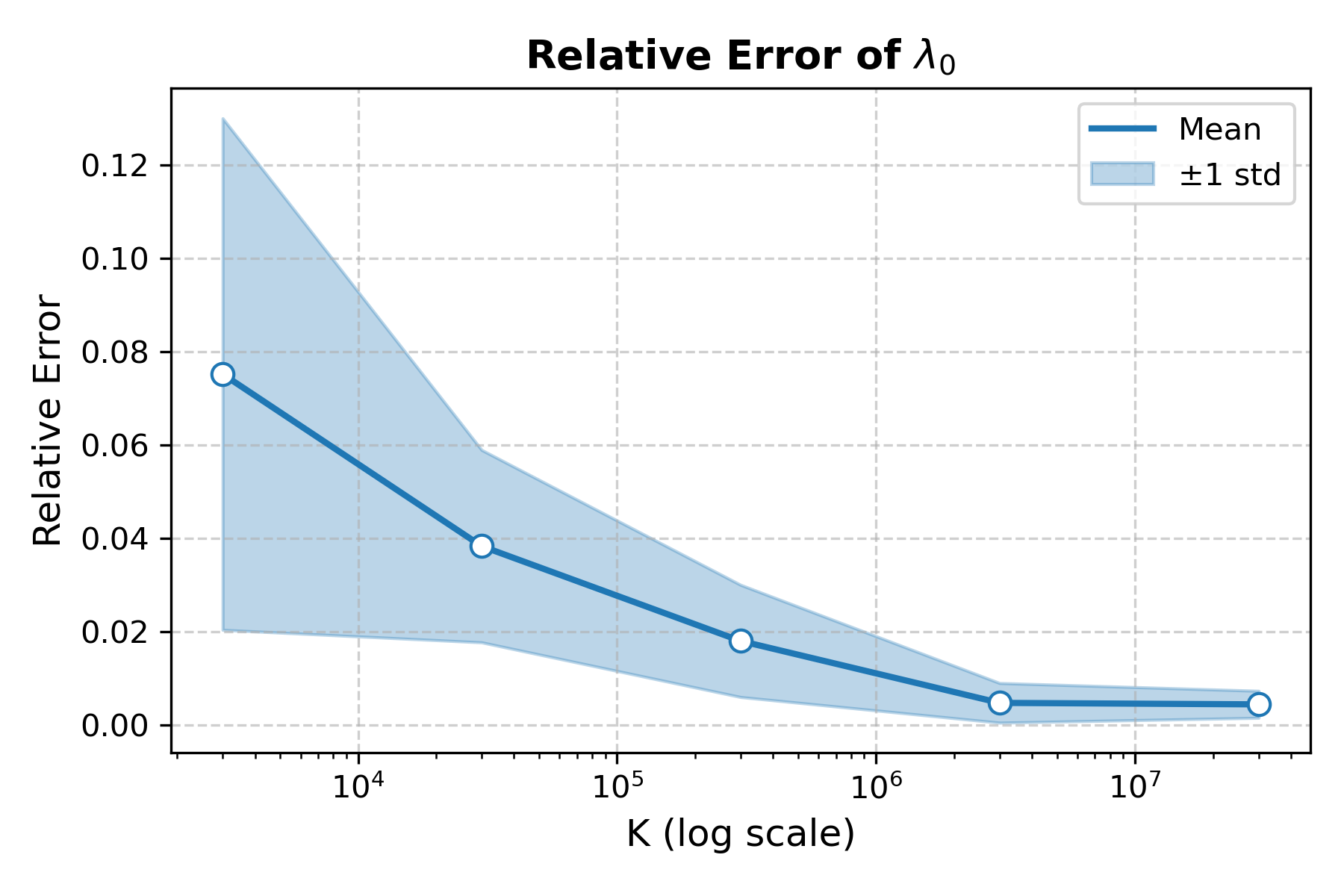}
    \includegraphics[width=0.48\linewidth]{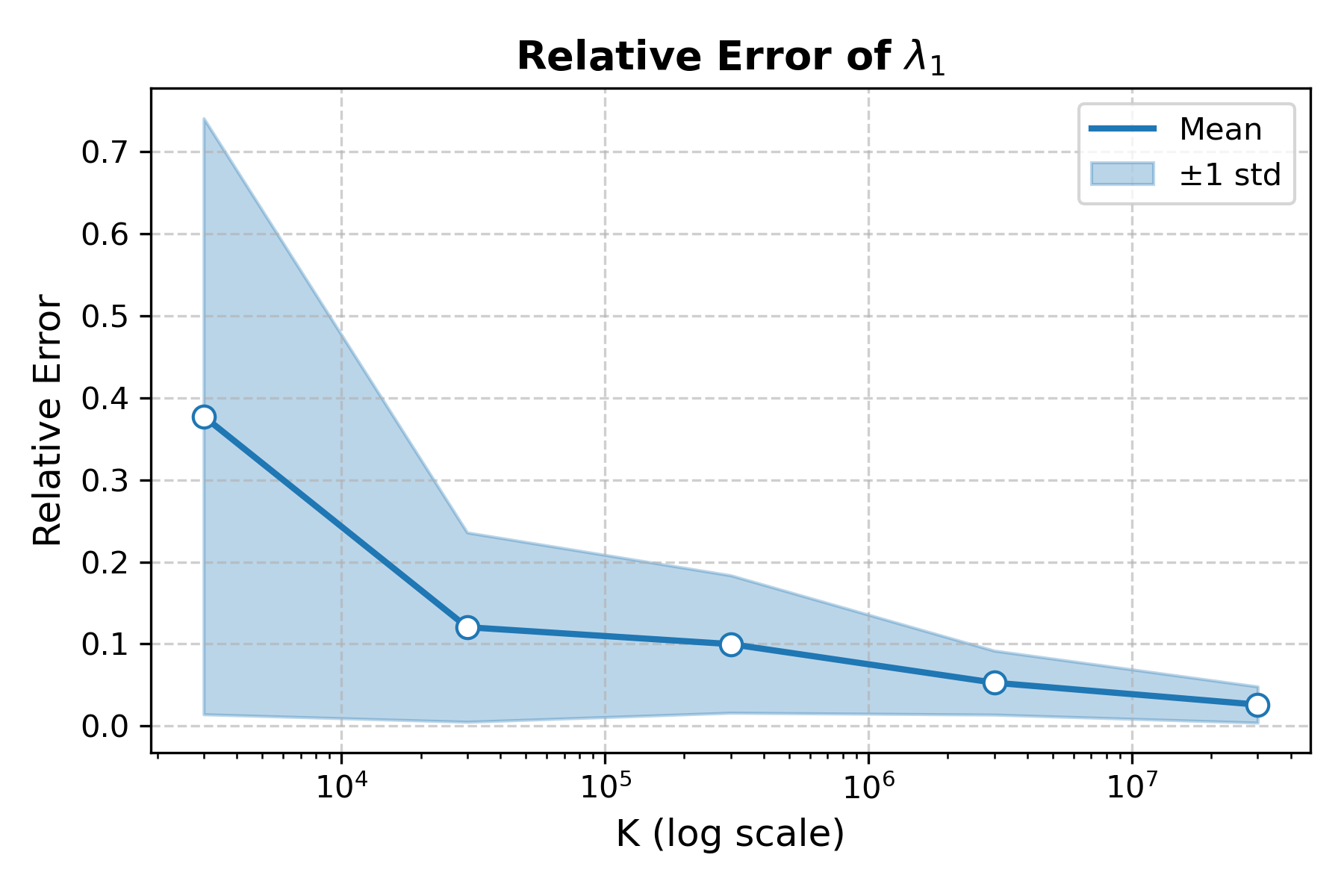}
\caption{Relative error of parameter estimators for increasing system sizes \( n = 10^3, 10^4, 10^5, 10^6, 10^7 \). Parameter values: \( \alpha = 0.8 \), \( \beta = 0.5 \), \( \lambda_0 = -0.5 \), \( \lambda_1 = 0.5 \), \( k = 3 \). Solid lines: mean relative error. Shaded areas: $\pm 1$ standard deviation.}
\label{fig:thm2}
\end{figure}

\subsection{Robustness Analysis}
To evaluate the robustness of the proposed estimators, we perform simulations with parameters sampled from the following ranges:  \( \lambda_0 \in (-0.9, -0.1) \), \( \lambda_1 \in (0.1, 0.9) \), \( \alpha \in (0.5, 0.9) \), \( \beta \in (0.1, 0.9) \), and \( k \in (1.5, 6.5) \). For each randomly generated parameter set, we impose the theoretical constraints required by Theorem~\ref{thm: estimators}, specifically \( \beta > 1 + \frac{\lambda_1}{\lambda_0} \) and \( \alpha + \beta > 1 \). Parameter combinations failing to satisfy these conditions are discarded and resampled. We fix the initial sensitive cell population at \( n = 5 \times 10^6 \) to balance computational tractability with biological realism and estimator accuracy. While moderate, this system size remains sufficient to capture statistically meaningful trends in estimator performance across diverse parameter regimes.

% To assess the robustness of our proposed estimators across a wide range of parameter settings, we conduct simulations under randomized parameter configurations. Specifically, we independently sample each parameter from the following intervals: \( \lambda_0 \in (-0.9, -0.1) \), \( \lambda_1 \in (0.1, 0.9) \), \( \alpha \in (0.5, 0.9) \), \( \beta \in (0.1, 0.9) \), and \( k \in (1.5, 6.5) \). For each sampled parameter set, we enforce the theoretical conditions required in Theorem~\ref{thm: estimators}, namely \( \beta > 1 + \frac{\lambda_1}{\lambda_0} \) and \( \alpha + \beta > 1 \); parameter sets that do not satisfy these conditions are discarded.

% We fix the initial sensitive cell population size at \( n = 5 \times 10^6 \) to strike a balance between computational feasibility and estimation accuracy. Although this value of \( n \) is moderate, it still allows us to capture meaningful trends in estimator performance across varying regimes.

% Each simulation records the relative error of all four estimators—\( \hat{\alpha} \), \( \hat{\beta} \), \( \hat{\lambda}_0 \), and \( \hat{\lambda}_1 \)—as defined earlier. In Figure~\ref{fig:robust_estimators}, we visualize these results using scatter plots, where each blue dot corresponds to the relative error from a single simulation instance. Histogram bars show the average relative error within parameter bins, and the red horizontal line indicates the global mean relative error.

For each simulation, we compute the relative error for all four estimators (\( \hat{\alpha} \), \( \hat{\beta} \), \( \hat{\lambda}_0 \), and \( \hat{\lambda}_1 \)). Figure~\ref{fig:robust_estimators} visualizes the simulation results using scatter plots: each blue point represents the relative error from an individual simulation run. Binned averages of relative errors are displayed as histogram bars, while the red horizontal line denotes the global mean relative error. The results demonstrate that the relative error remains consistently low across the full spectrum of tested parameter values. We observe no systematic bias or performance deterioration as parameters vary, suggesting that the estimators retain high accuracy and robustness. These findings provide strong empirical evidence for the reliability of our estimation framework across a biologically plausible parameter space.

% The results demonstrate that the relative error remains consistently low across the entire range of parameter values. There is no evident systematic bias or performance degradation as parameters vary, suggesting that the estimators maintain strong accuracy and stability. These findings provide empirical support for the robustness of our estimators across biologically plausible parameter ranges.

\begin{figure}[ht!]
\centering
    \includegraphics[width=0.45\linewidth]{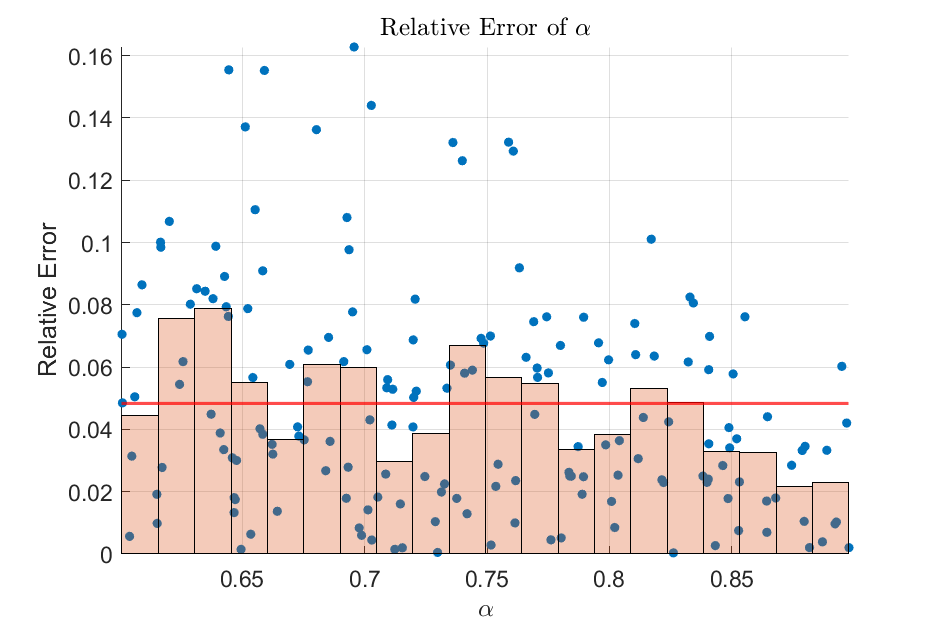}
    \includegraphics[width=0.45\linewidth]{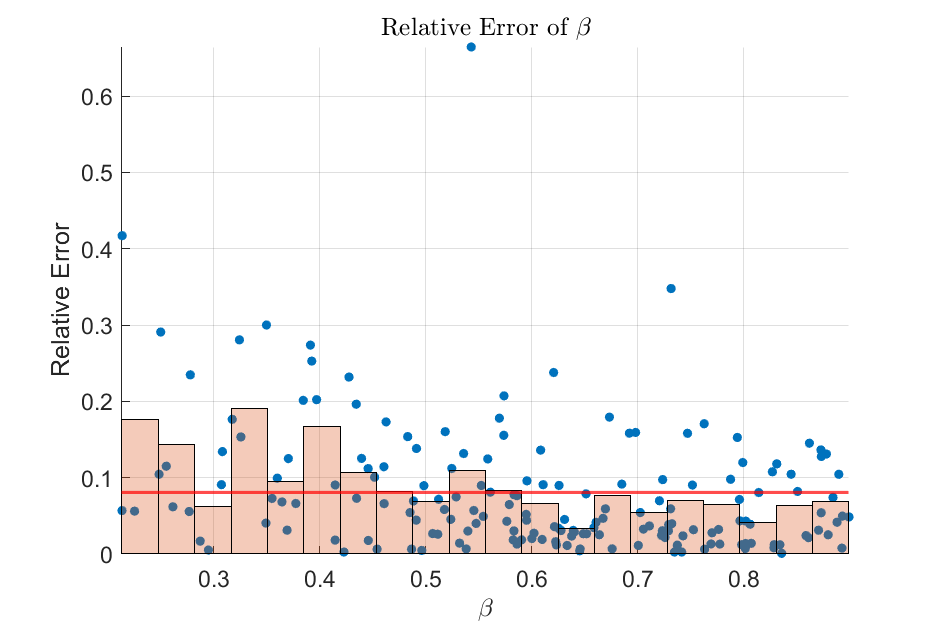} \\
    \vspace{2mm}
    \includegraphics[width=0.45\linewidth]{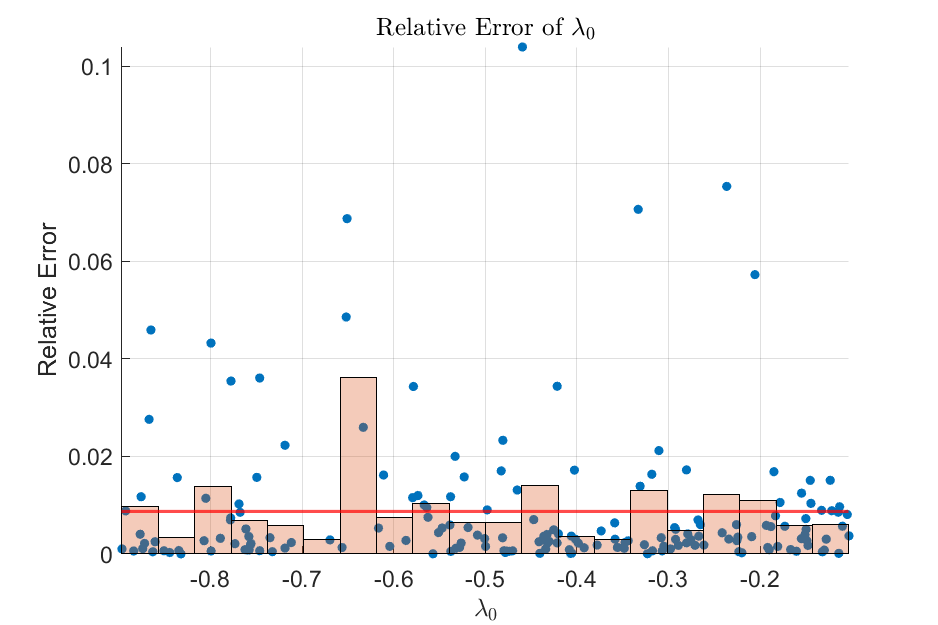}
    \includegraphics[width=0.45\linewidth]{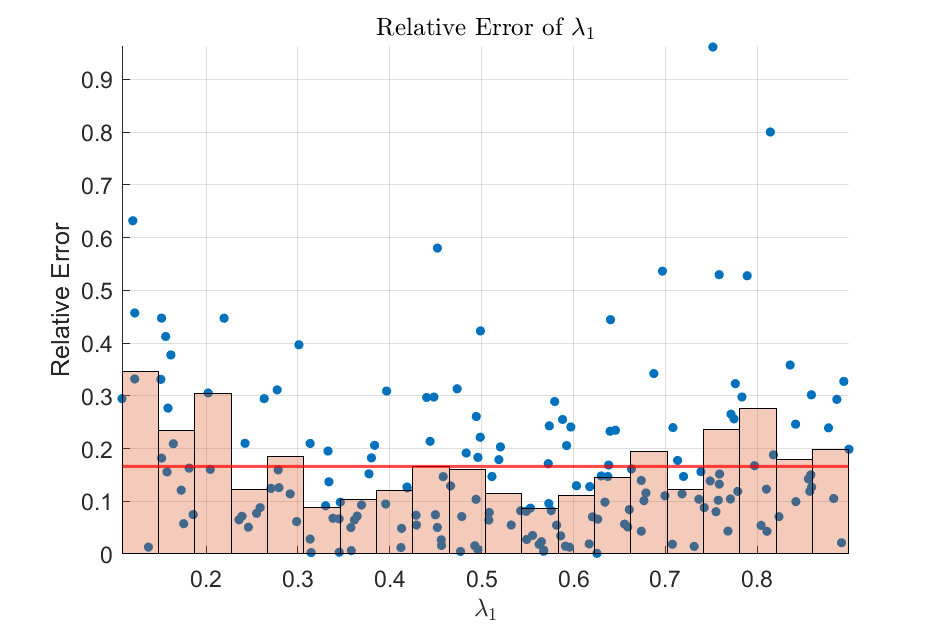}
    \caption{Relative errors of estimators across randomized parameter settings. Parameters are sampled from: \( \lambda_0 \in (-0.9, -0.1) \), \( \lambda_1 \in (0.1, 0.9) \), \( \alpha \in (0.5, 0.9) \), \( \beta \in (0.1, 0.9) \), and \( k \in (1.5, 6.5) \). The sample size is fixed at \( n = 5 \times 10^6 \). Blue dots represent individual simulation runs; histogram bars represent bin-wise mean errors; the red line represents the overall mean error.}
    \label{fig:robust_estimators}
\end{figure}

Although our estimators are theoretically independent of the carrying capacity scaling factor \( k \), it remains necessary to examine whether variation in \( k \) indirectly affect their performance. Intuitively, a smaller value of \( k \) corresponds to stricter resource constraints, which could lead to stronger non-linear effects and potentially reduce the accuracy of the deterministic ODE approximation. However, as shown in Figure~\ref{fig:mean_rel_err_vs_k}, the mean relative error across all four estimators remains low over a wide range of $k$, with no evident degradation in accuracy.

% Although our estimators do not directly depend on the carrying capacity scaling factor \( k \), it is still important to examine whether variation in \( k \) has any secondary effects on estimation performance. Intuitively, a smaller \( k \) implies a tighter carrying capacity, which could lead to stronger non-linear effects and thus potentially worse approximation by the deterministic ODE system. 

% However, Figure~\ref{fig:mean_rel_err_vs_k} demonstrates that the average relative error across all four estimators remains stable across a wide range of \( k \), with no strong trend indicating degradation in accuracy. In particular, even when \( k \) is small (e.g., near 1.5), the relative error remains comparable to larger \( k \)-values. This robustness can be attributed to the asymptotic regime: as \( n \to \infty \), the influence of the carrying capacity becomes increasingly negligible in the normalized stochastic dynamics, which aligns with our theoretical framework that does not require knowledge of \( k \) for estimation.

% These results further validate the practical applicability of our estimators, showing that their performance remains reliable across a broad range of biologically plausible carrying capacities.
\begin{figure}[h!]
    \centering
    \includegraphics[width=0.7\linewidth]{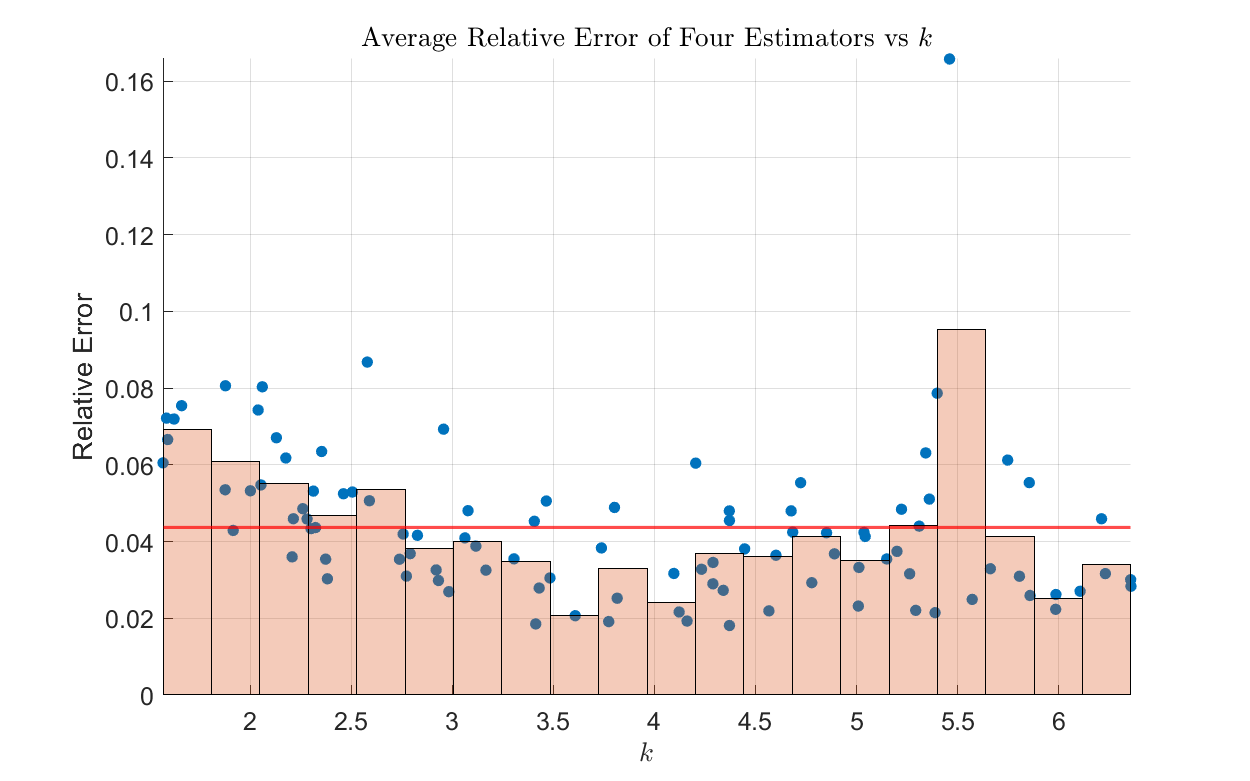}
    \caption{Effect of the carrying capacity scaling factor \( k \) on estimator performance. The vertical axis displays the average relative error across all four estimators \( (\hat{\alpha}, \hat{\beta}, \hat{\lambda}_0, \hat{\lambda}_1) \). Blue dots represent individual simulation runs; histogram bars represent bin-wise mean errors; the red line represents the overall mean error.}
    \label{fig:mean_rel_err_vs_k}
\end{figure}

% \subsection{Unobservable Clones}
% \textcolor{red}{Impractical. The majority of clones less than 2\% of n: alpha = 0.8, beta = 0.5. Try bigger alpha and smaller beta }
\newpage

\begin{appendix}
    \section{Proof of Proposition \ref{prop:zeta}}
    \label{sec:pf_prop_zeta}
\begin{lemma}
\label{lem:zeta_upper_bound}
    There exists a constant $C>0$, independent of $n$, such that 
    \begin{align*}
        \zeta_n < C+\frac{\min\left\{  1-\beta,\alpha\right\}}{\lambda_1}\log n.
    \end{align*}  
\end{lemma}  
\begin{proof}  
    % In the following analysis, we use notation and specify the function of growth rate: $\lambda_1(z) = \lambda_1(z_0,z_1)= \lambda_1(1-(z_0+z_1)/K)$. 
It is important to note that the function \( \phi \) in the ODE system~\eqref{ODEs} is not explicitly known, which precludes direct analytical treatment of the system. However, using the definition of the deterministic recurrence time \( \zeta_n \) and \textit{(A5)} of Assumption \ref{assump:f}, we can obtain an upper bound for \( \zeta_n \) via a lower bound for the solution \( y_1(t) \). In what follows, we construct an auxiliary function \( \bar{y}_1(t) \) that serves as a lower bound for \( y_1(t) \).

% To this end, we first construct an auxiliary function \( \bar{y}_1(t) \) that serves as a lower bound for the ODE solution.
% we can obtain meaningful insights by studying a lower bound of the solution \( y_1(t) \). In particular, deriving a lower bound for \( y_1(t) \) enables us to establish an upper bound for \( \zeta_n \). 

For \( \epsilon > 0 \), define \( \bar{\zeta}_n = \epsilon \log n \), and let  
\[
    \bar{\lambda}_1 = \min_{0 \leq t \leq \bar{\zeta}_n} \{ \phi(y(t)) \}.
\]
We know that \( Ky_0(t) = n e^{\lambda_0 t} \). Because \( \phi(Ky) \leq \lambda_1 \), we also have
\[
    Ky_1(t) \leq n^{\beta} e^{\lambda_1 t} + \frac{n^{1 - \alpha}}{\lambda_1 - \lambda_0} \left( e^{\lambda_1 t} - e^{\lambda_0 t} \right).
\]
Let 
\[
    g(t) = n e^{\lambda_0 t} + n^{\beta} e^{\lambda_1 t} + \frac{n^{1 - \alpha}}{\lambda_1 - \lambda_0} \left( e^{\lambda_1 t} - e^{\lambda_0 t} \right).
\]
We have
\[
    g'(t) = \lambda_0 \left( n - \frac{n^{1 - \alpha}}{\lambda_1 - \lambda_0} \right) e^{\lambda_0 t} + \lambda_1 \left( n^{\beta} + \frac{n^{1 - \alpha}}{\lambda_1 - \lambda_0} \right) e^{\lambda_1 t}.
\]
% Then,
% \[
%     g'(t) \leq 0 \Longleftrightarrow e^{(\lambda_1 - \lambda_0)t} \leq \frac{-\lambda_0 \left( n - \frac{n^{1 - \alpha}}{\lambda_1 - \lambda_0} \right)}{\lambda_1 \left( n^{\beta} + \frac{n^{1 - \alpha}}{\lambda_1 - \lambda_0} \right)}.
% \]
% For sufficiently large \( n \), we have:
% \[
%     \frac{-\lambda_0 \left( n - \frac{n^{1 - \alpha}}{\lambda_1 - \lambda_0} \right)}{\lambda_1 \left( n^{\beta} + \frac{n^{1 - \alpha}}{\lambda_1 - \lambda_0} \right)}
%     \geq \frac{-\lambda_0}{\lambda_1 + \frac{\lambda_1}{\lambda_1 - \lambda_0}} n^{\min\{1 - \beta, \alpha\}} + \frac{\lambda_0}{\lambda_1}
%     \geq \frac{-\lambda_0}{2\lambda_1 + \frac{2\lambda_1}{\lambda_1 - \lambda_0}} n^{\min\{1 - \beta, \alpha\}}.
% \]
One can verify that for 
\[
    t < \frac{\min\{1 - \beta, \alpha\}}{\lambda_1 - \lambda_0} \log n + \frac{\log \left( \frac{-\lambda_0}{2\lambda_1 + \frac{2\lambda_1}{\lambda_1 - \lambda_0}} \right)}{\lambda_1 - \lambda_0},
\]
the inequality \( g'(t) \leq 0 \) holds. In conclusion, we establish that for sufficiently large \( n \) (specifically, for \( n \) larger than a constant depending on \( \lambda_0, \lambda_1, \alpha \), and \( \beta \)), the following holds:  
If 
\[
t \leq \frac{\min\{1 - \beta, \alpha\}}{2(\lambda_1 - \lambda_0)} \log n,
\]
then \( g'(t) < 0 \). This implies that for any \( \epsilon < \frac{\min\{1 - \beta, \alpha\}}{2(\lambda_1 - \lambda_0)} \), if \( t< \bar{\zeta}_n\), we have
\begin{align}
    \label{eq:upper_bound_y0plusy1}
y_0(t) + y_1(t) \leq g(t)/K \leq g(0)/K = n/K + n^\beta/K.
\end{align}
Moreover, by \textit{(A3)} of Assumption~\ref{assump:f}, we have \( \bar{\lambda}_1 = \phi(n/K, n^\beta/K) \).

We now construct an auxiliary trajectory \( \bar{y}_1(t) \), defined as the solution to the following piecewise system:
\begin{align}
\label{ODEs:y_bar}
    \begin{cases}
        \dfrac{d\bar{y}_1}{dt} = \bar{\lambda}_1 \bar{y}_1 + n^{-\alpha} y_0(t), & \text{for } t \leq \bar{\zeta}_n, \\
        \dfrac{d\bar{y}_1}{dt} = \lambda_1 \left( 1 - (y_0(t) + \bar{y}_1(t)) \right) \bar{y}_1(t), & \text{for } t > \bar{\zeta}_n,
    \end{cases}
\end{align}
subject to the initial condition \( K \bar{y}_1(0) = n^\beta \). By \textit{(A5)} of Assumption~\ref{assump:f},
\[
\frac{d\bar{y}_1}{dt} \leq \frac{dy_1}{dt}
\]
whenever \( \bar{y}_1(t) = y_1(t) \). This monotonicity property ensures that \( \bar{y}_1(t) \leq y_1(t) \) for all \( t \geq 0 \). 

We now proceed to analyze the behavior of $\bar{y}_1(t)$. For the first phase (\( t \leq \bar{\zeta}_n \)), we can solve the equation explicitly:
\begin{align}
    K \bar{y}_1(t) = \frac{n^{1 - \alpha}}{\lambda_0 - \bar{\lambda}_1} e^{\lambda_0 t} 
    + \left( n^\beta + \frac{n^{1 - \alpha}}{\bar{\lambda}_1 - \lambda_0} \right) e^{\bar{\lambda}_1 t}.
\end{align}
Hence, at time \( \bar{\zeta}_n \), we have:
\begin{align}
\label{eq:bar_y1_bar_zeta}
    K \bar{y}_1(\bar{\zeta}_n) 
    = n^{\beta + \bar{\lambda}_1 \epsilon} 
    + \frac{1}{\bar{\lambda}_1 - \lambda_0} \left( n^{1 - \alpha + \bar{\lambda}_1 \epsilon} - n^{1 - \alpha + \lambda_0 \epsilon} \right)<n.
\end{align}
We then consider the second phase (\( t > \bar{\zeta}_n \)). Define \( \bar{y} = \bar{y}_1(\bar{\zeta}_n) \) for convenience. Then, for \( t > 0 \), the solution satisfies:
\begin{align}
    \bar{y}_1(t + \bar{\zeta}_n) 
    &= \frac{\bar{y} e^{\lambda_1 t} \exp\left( \frac{\lambda_1 n}{-\lambda_0 K} e^{\lambda_0 t} \right)}{\bar{y} \lambda_1 \int_0^t e^{\lambda_1 u} \exp\left( \frac{\lambda_1 n}{-\lambda_0 K} e^{\lambda_0 u} \right) du + \exp\left( \frac{\lambda_1 n}{-\lambda_0 K} \right)}.
\end{align}
To simplify the expression, let \( \mu = \frac{\lambda_1 n}{-\lambda_0 K} \). Then, for any \( s \in (0, t) \), we bound the integral as follows:
\begin{align*}
    \int_0^t e^{\lambda_1 u} \exp\left( \mu e^{\lambda_0 u} \right) du 
    &\leq \int_0^{s} e^{\lambda_1 u} e^{\mu} du + \int_{s}^t e^{\lambda_1 u} \exp\left(\mu e^{\lambda_0 s}\right) du \\
    &= \frac{e^{\mu}}{\lambda_1} \left( e^{\lambda_1 s} - 1 \right)
    + \frac{\exp\left(\mu e^{\lambda_0 s}\right)}{\lambda_1} \left( e^{\lambda_1 t} - e^{\lambda_1 s} \right).
\end{align*}
Therefore, we can obtain the following lower bound:
\begin{align*}
    \bar{y}_1(t + \bar{\zeta}_n) 
    &\geq \frac{\bar{y} e^{\lambda_1 t}}{\bar{y} \lambda_1 \int_0^t e^{\lambda_1 u} \exp\left( \mu e^{\lambda_0 u} \right) du + e^{\mu}} \\
    &\geq \frac{e^{\lambda_1 t}}{\exp\left(\mu e^{\lambda_0 s}\right) e^{\lambda_1 t} + e^{\mu} \left( \bar{y}^{-1} - 1 \right) + e^{\lambda_1 s} \left( e^{\mu} - \exp\left(\mu e^{\lambda_0 s}\right) \right)}.
\end{align*}
Recall that \( K(n) = k n \), where \( k > 1 \). Because \( k>1 \), \( \lambda_0 < 0 \) and \(\mu>0\) are all constants, there exists a constant \( \theta > 0 \) such that
\(\exp\left( \mu e^{\lambda_0 \theta} \right) < k\). To streamline notation, define
\[
\nu := \exp\left( \mu e^{\lambda_0 \theta} \right), \quad
\xi := e^{\lambda_1 \theta} \left( e^{\mu} - \exp\left( {\mu} e^{\lambda_0 \theta} \right) \right) > 0.
\]
Then for all \( t > \theta \), we obtain the lower bound:
\begin{align}
\label{eq:lower_bound_y1}
    \bar{y}_1(t + \bar{\zeta}_n) 
    \geq \frac{e^{\lambda_1 t}}{\nu e^{\lambda_1 t} + e^{\mu}\left(\bar{y}^{-1} - 1\right) + \xi}.
\end{align}

Now suppose there exists \( \tilde{\zeta}_n > 0 \) such that \( K \bar{y}_1(\tilde{\zeta}_n + \bar{\zeta}_n) = n \), which exists because \(\exp\left( \mu e^{\lambda_0 \theta} \right) < k\). Then $\tilde{\zeta}_n$ must satisfy:
\begin{align*}
    e^{\lambda_1 \tilde{\zeta}_n} \leq \frac{e^{\mu}\left(\bar{y}^{-1} - 1\right) + \xi}{k - \nu}.
\end{align*}
% Also, if we let  
% \[
% \bar{\lambda} := 1 - \frac{\bar{\lambda}_1}{2(\lambda_1 - \lambda_0)} > 0,
% \quad \text{and} \quad \epsilon < \frac{\min\{1 - \beta, \alpha\}}{2(\lambda_1 - \lambda_0)}.
% \]
% then we will have:
% \begin{align*}
%     \frac{1}{\bar{y}} - 1 
%     &= \frac{K}{n} \cdot \frac{n}{K \bar{y}} - 1 \\
%     &= \frac{K}{n} \cdot \frac{1}{n^{\beta + \bar{\lambda}_1 \epsilon - 1} 
%     + \dfrac{1}{\bar{\lambda}_1 - \lambda_0} \left( n^{-\alpha + \bar{\lambda}_1 \epsilon} - n^{-\alpha + \lambda_0 \epsilon} \right)} - 1 \\
%     &\geq C n^{\bar{\lambda} \min\{1 - \beta, \alpha\}},
% \end{align*}
Moreover, one may verify that for sufficiently small $\epsilon>0$, \(\bar{y}^{-1} \rightarrow \infty\) as $n\rightarrow \infty$. Hence, for large enough \( n \), there exists a positive constant $C$ such that
\[
 e^{\lambda_1 \tilde{\zeta}_n} \leq C\bar{y}^{-1} .
\]

% Also, it's easy to verify that for sufficiently small $\epsilon$, \(\bar{y}^{-1} \xrightarrow[]{n} \infty\). Therefore, for large \( n \) (independent of $\epsilon$ as long as $\epsilon$ smaller than a threshold), we have the bound:
% \[
%  e^{\lambda_1 \tilde{\zeta}_n} \leq C\bar{y}^{-1} .
% \]

Next, we consider two cases depending on the relative magnitude of  \( 1 - \beta \) and \( \alpha \):

\noindent\textbf{(1)} \( \boldsymbol{1 - \beta < \alpha } \) 

In this case, we have
\[
\bar{y}^{-1}  \leq k \cdot n^{1 - \beta - \bar{\lambda}_1 \epsilon},
\]
which implies
\[
e^{\lambda_1 \tilde{\zeta}_n} \leq Ck \cdot n^{1 - \beta - \bar{\lambda}_1 \epsilon}.
\]
Taking logarithms yields
\[
\tilde{\zeta}_n \leq \frac{1}{\lambda_1} \log\left( Ck \right) + \frac{1 - \beta - \bar{\lambda}_1 \epsilon}{\lambda_1} \log n.
\]
Since \( \bar{y}_1(t) \leq y_1(t) \), it follows that
\[
\zeta_n < \bar{\zeta}_n + \tilde{\zeta}_n 
\leq \frac{\lambda_1 - \bar{\lambda}_1}{\lambda_1} \epsilon \log n 
+ \frac{1}{\lambda_1} \log\left( Ck \right)
+ \frac{1 - \beta}{\lambda_1} \log n.
\]
Taking the limit as \( \epsilon \to 0 \), we conclude:
\[
\zeta_n \leq \frac{1}{\lambda_1} \log\left( Ck \right)
 + \frac{1 - \beta}{\lambda_1} \log n.
\]

\noindent\textbf{(2)} \( \boldsymbol{1 - \beta \geq \alpha } \) 

In this case, we have:
\[
\bar{y}^{-1} \leq k \cdot \frac{\bar{\lambda}_1 - \lambda_0}{n^{-\alpha + \bar{\lambda}_1 \epsilon} - n^{-\alpha + \lambda_0 \epsilon}},
\]
which implies
\begin{align*}
    e^{\lambda_1 \tilde{\zeta}_n} 
    &\leq Ck \cdot \frac{\bar{\lambda}_1 - \lambda_0}{n^{-\alpha + \bar{\lambda}_1 \epsilon} - n^{-\alpha + \lambda_0 \epsilon}} \\
    &= Ck \cdot \frac{(\bar{\lambda}_1 - \lambda_0) n^{\alpha - \bar{\lambda}_1 \epsilon}}{1 - n^{(\lambda_0 - \bar{\lambda}_1) \epsilon}}.
\end{align*}
Setting \( \epsilon = \frac{1}{\log n} \), we derive
\[
e^{\lambda_1 \tilde{\zeta}_n} \leq Ck \cdot n^{\alpha} \cdot e^{-\bar{\lambda}_1} \cdot\frac{\bar{\lambda}_1 - \lambda_0}{1 - e^{\lambda_0 - \bar{\lambda}_1}}.
\]
Taking logarithms yields
\[
\tilde{\zeta}_n \leq \frac{1}{\lambda_1} \log\left( Ck \right) 
+ \frac{\alpha}{\lambda_1} \log n - \frac{\bar{\lambda}_1}{\lambda_1} + \frac{1}{\lambda_1} \log\left(\frac{\bar{\lambda}_1 - \lambda_0}{1 - e^{\lambda_0 - \bar{\lambda}_1}}\right).
\]
Because \( \bar{\zeta}_n = \epsilon \log n = \frac{1}{\log n} \log n = 1 \), we conclude
\[
\zeta_n < \bar{\zeta}_n + \tilde{\zeta}_n 
\leq 1 - \frac{\bar{\lambda}_1}{\lambda_1}
+ \frac{1}{\lambda_1} \log\left(\frac{\bar{\lambda}_1 - \lambda_0}{1 - e^{\lambda_0 - \bar{\lambda}_1}}\right)
+ \frac{1}{\lambda_1} \log\left(Ck \right)
+ \frac{\alpha}{\lambda_1} \log n.
\]

In conclusion, define the constant
\[
\bar{C} := \max\left\{ 
\frac{1}{\lambda_1} \log\left( Ck \right),\ 
1 - \frac{\bar{\lambda}_1}{\lambda_1}
    + \frac{1}{\lambda_1} \log\left(\frac{\bar{\lambda}_1 - \lambda_0}{1 - e^{\lambda_0 - \bar{\lambda}_1}}\right)
    + \frac{1}{\lambda_1} \log\left( Ck \right)
\right\}.
\]
Then, in either case, we obtain the unified upper bound:
\[
\boxed{
\zeta_n < \bar{C} + \frac{\min\{1 - \beta, \alpha\}}{\lambda_1} \log n.
}
\]
\qed            
\end{proof}

In the proof of Lemma \ref{lem:zeta_upper_bound}, combining equations~\eqref{eq:bar_y1_bar_zeta} and~\eqref{eq:lower_bound_y1} and taking \( \epsilon = 1/\log n \), we derive the following upper bound for the inverse of \( y_1(t) \):
\begin{align}
\label{eq:upper_bound_1_over_y1}
\frac{1}{y_1(t)} \leq 
\begin{cases}
    c_1 n^{1 - \beta}, & \text{for } t \leq 1, \\
    c_2 + c_3 n^{\min\{1 - \beta, \alpha\}} e^{-\lambda_1 t}, & \text{for } t > 1,
\end{cases}
\end{align}
where \( c_1, c_2, c_3 \) are positive constants. This bound will be instrumental to the subsequent analysis. 

In Lemma \ref{lem:zeta_upper_bound}, we established an upper bound for \( \zeta_n \). Using similar arguments, a corresponding lower bound can be derived.
\begin{lemma}
\label{lem:zeta_lower_bound}
There exists a constant \( c > 0 \), independent of \( n \), such that for all sufficiently large \( n \),
\[
\zeta_n > -c + \frac{\min\left\{ 1 - \beta, \alpha \right\}}{\lambda_1} \log n.
\]
\end{lemma}

\begin{proof}
To derive a lower bound for \( \zeta_n \), we construct an auxiliary function \( \hat{y}_1(t) \) that bounds \( y_1(t) \) from above. Define \( \hat{y}_1(t) \) as the solution to the differential equation:
\begin{align}
\label{ODEs3}
\frac{d\hat{y}_1}{dt} = \lambda_1 \hat{y}_1 + n^{-\alpha} y_0(t),
\end{align}
with initial condition \( K\hat{y}_1(0) = n^{\beta} \). Since \( \lambda_1 \) corresponds to the intrinsic net growth rate of resistant cells in the absence of competition, this choice ensures that \( \hat{y}_1(t) \) dominates \( y_1(t) \), i.e., 
\[
y_1(t) \leq \hat{y}_1(t), \quad \text{for all } t \geq 0.
\]

Solving \eqref{ODEs3} yields:
\begin{align}
K \hat{y}_1(t) = \frac{n^{1 - \alpha}}{\lambda_0 - \lambda_1} e^{\lambda_0 t} 
+ \left( n^{\beta} + \frac{n^{1 - \alpha}}{\lambda_1 - \lambda_0} \right) e^{\lambda_1 t}.
\end{align}
Define \( \hat{\zeta}_n \) as the time at which \( K\hat{y}_1(\hat{\zeta}_n) = n \). Then:
\begin{align*}
e^{\lambda_1 \hat{\zeta}_n} 
&= \frac{n - \frac{n^{1 - \alpha}}{\lambda_0 - \lambda_1} e^{\lambda_0 \hat{\zeta}_n}}{n^{\beta} + \frac{n^{1 - \alpha}}{\lambda_1 - \lambda_0}} \\
&\geq \frac{n}{n^{\beta} + \frac{n^{1 - \alpha}}{\lambda_1 - \lambda_0}} \\
&\geq \frac{\lambda_1 - \lambda_0}{1 + \lambda_1 - \lambda_0} \cdot n^{\min\{1 - \beta, \alpha\}}.
\end{align*}
Taking logarithms gives:
\[
\lambda_1 \hat{\zeta}_n \geq \log\left( \frac{\lambda_1 - \lambda_0}{1 + \lambda_1 - \lambda_0} \right) + \min\{1 - \beta, \alpha\} \log n.
\]
Define 
\[
c := \frac{1}{\lambda_1} \log\left( \frac{1 + \lambda_1 - \lambda_0}{\lambda_1 - \lambda_0} \right).
\]
and we have
\[
\hat{\zeta}_n \geq -c + \frac{\min\left\{ 1 - \beta, \alpha \right\}}{\lambda_1} \log n.
\]
Since \( \zeta_n \geq \hat{\zeta}_n \), the desired lower bound follows.
\qed
\end{proof}

As an immediate consequence of Lemma~\ref{lem:zeta_upper_bound} and Lemma~\ref{lem:zeta_lower_bound}, we obtain the asymptotic characterization of the deterministic recurrence time stated in Proposition~\ref{prop:zeta}:
\[
\lim_{n \to \infty} \frac{\zeta_n}{\log n} = \frac{\min\left\{ 1 - \beta, \alpha \right\}}{\lambda_1}.
\]

% As a direct consequence of Lemma~\ref{lem:zeta_upper_bound} and Lemma~\ref{lem:zeta_lower_bound}, we conclude the result of Proposition~\ref{prop:zeta}, which establishes the precise asymptotic behavior of the deterministic recurrence time \( \zeta_n \):
% \[
% \lim_{n \to \infty} \frac{\zeta_n}{\log n} = \frac{\min\left\{ 1 - \beta, \alpha \right\}}{\lambda_1}.
% \]

\section{Proof of Proposition \ref{prop:ybeta}}
\label{sec:pf_prop_ybeta}
\begin{proof}
First, note that
\[
y_{\beta}(t) = y_{\beta}(0) \exp\left( \int_0^t \phi(Ky(s)) \, ds \right),
\]
which implies
\begin{align}
K y_{\beta}(\zeta_n) = n^{\beta} \exp\left( \int_0^{\zeta_n} \phi(Ky(s)) \, ds \right).
\end{align}
From the ODE~(\ref{ODEs}), we have
\[
\frac{d}{dt} \log y_1(t) = \phi(Ky(t)) + \frac{y_0(t)}{y_1(t)} \cdot n^{-\alpha}.
\]
Integrating from $0$ to \( t \) yields
\[
\log y_1(t) - \log y_1(0) = \int_0^t \left( \phi(Ky(s)) + \frac{y_0(s)}{y_1(s)} \cdot n^{-\alpha} \right) ds.
\]
Therefore,
\begin{align}
K y_1(\zeta_n) = n^{\beta} \exp\left( \int_0^{\zeta_n} \phi(Ky(s)) \, ds \right) \exp\left( \int_0^{\zeta_n} \frac{y_0(s)}{y_1(s)} \cdot n^{-\alpha} \, ds \right).
\end{align}
Because \( Ky_1(\zeta_n) = n \), it follows that
\begin{align}
\label{eq:y_beta_zeta}
K y_{\beta}(\zeta_n) = n \cdot \exp\left( -n^{-\alpha} \int_0^{\zeta_n} \frac{y_0(s)}{y_1(s)} \, ds \right).
\end{align}

To analyze equation~\eqref{eq:y_beta_zeta}, we apply the upper bound for \( 1/y_1(t) \) from~\eqref{eq:upper_bound_1_over_y1}, obtaining:
\begin{align*}
    \int_0^{\zeta_n} \frac{y_0(s)}{y_1(s)} \, ds 
    &= \int_0^1 \frac{y_0(s)}{y_1(s)} \, ds + \int_1^{\zeta_n} \frac{y_0(s)}{y_1(s)} \, ds \\
    &\leq C_1 n^{1 - \beta} + C_2 \int_1^{\zeta_n} \left( e^{\lambda_0 s} + n^{1 - \beta} e^{(\lambda_0 - \lambda_1)s} \right) ds \\
    &\leq C_3 n^{1 - \beta},
\end{align*}
where \( C_1, C_2, C_3 \) are positive constants and the second inequality holds for sufficiently large \( n \). Furthermore, by reusing the auxiliary function \( \hat{y}_1 \) defined in~\eqref{ODEs3} and selecting \( \epsilon > 0 \) such that \( \zeta_n > \epsilon \log n \) for sufficiently large \( n \), we obtain:
\begin{align*}
    \int_0^{\zeta_n} \frac{y_0(s)}{y_1(s)} \, ds 
    &\geq \int_0^{\epsilon \log n} \frac{y_0(s)}{y_1(s)} \, ds 
    \geq \int_0^{\epsilon \log n} \frac{y_0(s)}{\hat{y}_1(s)} \, ds \\
    &= \int_0^{\epsilon \log n} \frac{1}{
        \frac{n^{-\alpha}}{\lambda_0 - \lambda_1}
        + \left( n^{\beta - 1} + \frac{n^{-\alpha}}{\lambda_1 - \lambda_0} \right) e^{(\lambda_1 - \lambda_0)s}
    } \, ds.
\end{align*}
Define \( a = \frac{n^{-\alpha}}{\lambda_1 - \lambda_0} \), 
\( b = n^{\beta - 1} + \frac{n^{-\alpha}}{\lambda_1 - \lambda_0} \), 
and \( c = \lambda_1 - \lambda_0 \). Then:
\begin{align*}
    \int_0^{\zeta_n} \frac{y_0(s)}{y_1(s)} \, ds 
    \geq \int_0^{\epsilon \log n} \frac{1}{-a + b e^{c s}} \, ds.
\end{align*}
Substitute \( u = e^{c s} \), hence \( ds = \frac{du}{c u} \), yields:
\begin{align*}
    \int_0^{\epsilon \log n} \frac{1}{-a + b e^{cs}} \, ds 
    &= \int_1^{n^{c\epsilon}} \frac{1}{-a + b u} \cdot \frac{du}{c u} \\
    &= \frac{1}{ac} \left( \int_1^{n^{c\epsilon}} \frac{1}{u - \frac{a}{b}} \, du 
    - \int_1^{n^{c\epsilon}} \frac{1}{u} \, du \right) \\
    &= \frac{1}{ac} \left( 
    \log \left( \frac{n^{c\epsilon} - \frac{a}{b}}{1 - \frac{a}{b}} \right) 
    - \log n^{c\epsilon}
    \right),
\end{align*}
where
\begin{align*}
    \frac{n^{c\epsilon} - \frac{a}{b}}{1 - \frac{a}{b}} 
    &= \left( 1 + \frac{n^{1 - \alpha - \beta}}{\lambda_1 - \lambda_0} \right) 
    \left( n^{c\epsilon} - \frac{n^{1 - \alpha}}{(\lambda_1 - \lambda_0) n^{\beta} + n^{1 - \alpha}} \right) \\
    &= n^{c\epsilon} 
    + \frac{n^{1 - \alpha - \beta + c\epsilon}}{\lambda_1 - \lambda_0} 
    - \frac{n^{1 - \alpha - \beta}}{\lambda_1 - \lambda_0}.
\end{align*}
Thus,
\begin{align*}
    \int_0^{\epsilon \log n} \frac{1}{-a + b e^{cs}} \, ds 
    &= \frac{1}{ac} \left( 
    \log \left( \frac{n^{c\epsilon} - \frac{a}{b}}{1 - \frac{a}{b}} \right) 
    - \log n^{c\epsilon}
    \right) \\
    &= n^{\alpha} \log \left( 
    1 + \left( \frac{n^{1 - \alpha - \beta}}{\lambda_1 - \lambda_0} 
    - \frac{n^{1 - \alpha - \beta}}{\lambda_1 - \lambda_0} n^{-c\epsilon} \right) 
    \right) \\
    &\geq \frac{n^{1 - \beta}}{2(\lambda_1 - \lambda_0)},
\end{align*}
where the inequality follows from Taylor Expansion and holds for sufficiently large \( n \).

Therefore, combining the above analysis with equation (\ref{eq:y_beta_zeta}), we obtain the following bounds:
\begin{align}
\label{eq:bound_log_ybeta}
    \frac{n^{1 - \alpha - \beta}}{2(\lambda_1 - \lambda_0)} 
    \leq \log \left( \frac{n}{K y_{\beta}(\zeta_n)} \right) 
    \leq C_3 n^{1 - \alpha - \beta}.
\end{align}
It can be verified that we may take
\[
C_3 = \frac{2}{\bar{\lambda}_1 - \lambda_0},
\]
where
\[
\bar{\lambda}_1 := \min_{0 \leq t \leq \zeta_n} \phi(y(t)) > 0.
\]
Taking logarithms once more yields the limit:
\begin{align*}
    \lim_{n \to \infty} \frac{\log \log \left( \frac{n}{K y_{\beta}(\zeta_n)} \right)}{\log n} 
    = 1 - \alpha - \beta.
\end{align*}
\qed
\end{proof}

\section{Proof of Theorem \ref{thm:ratio_conv}}
\label{sec:pf_thm_ratio_conv}
\begin{proof}
The proof of Theorem~\ref{thm:ratio_conv} is decomposed into three components, corresponding to the convergence of the following ratios: the sensitive cell population \( X_0(t)/y_0(t) \), the total resistant cell population \( X_1(t)/y_1(t) \), and the pre-existing resistant clone \( X_{\beta}(t)/y_{\beta}(t) \). Although each subpopulation evolves under different dynamical constraints, the proofs follow the same framework. Each case is treated in detail in the subsequent subsections.

% The proof of Theorem~\ref{thm:ratio_conv} will be divided into three parts, corresponding to the ratios for the sensitive cells \( X_0(t)/y_0(t) \), the total resistant population \( X_1(t)/y_1(t) \), and the pre-existing resistant clone \( X_{\beta}(t)/y_{\beta}(t) \). Although these components evolve under different growth mechanisms, they share a common proof strategy. We will address each case separately in the following subsections.

\subsection*{Sensitive cells}
We adopt the scaled process representation from~\cite{bansaye2024sharp} and express the ratio \( X_0(t)/y_0(t) \) as a semimartingale:
% We adopt the representation of the scaled process as formulated in~\cite{bansaye2024sharp}, and express the ratio \( X_0(t)/y_0(t) \) as a semimartingale:
\begin{align}
\frac{X_0(t)}{y_0(t)} &= 1 
- \int_0^t \frac{X_0(s)}{y_0(s)} \lambda_0 \, ds 
+ \int_0^t \int_0^{\infty} \frac{1}{K y_0(s-)} \mathbbm{1}_{\left\{ u \leq K X_0(s-) r_0 \right\}} \mathcal{N}_0^b(ds, du) \nonumber \\
&\quad - \int_0^t \int_0^{\infty} \frac{1}{K y_0(s-)} \mathbbm{1}_{\left\{ u \leq K X_0(s-) d_0 \right\}} \mathcal{N}_0^d(ds, du) \nonumber \\
&= 1 + E_0(t),
\end{align}
% where the fluctuation term \( E_0(t) \) is defined as:
where \( E_0(t) \) is given by:
\begin{align}
E_0(t) &= \int_0^t \int_0^{\infty} \frac{1}{K y_0(s-)} \mathbbm{1}_{\left\{ u \leq K X_0(s-) r_0 \right\}} \tilde{\mathcal{N}}_0^b(ds, du) \nonumber \\
&\quad - \int_0^t \int_0^{\infty} \frac{1}{K y_0(s-)} \mathbbm{1}_{\left\{ u \leq K X_0(s-) d_0 \right\}} \tilde{\mathcal{N}}_0^d(ds, du).
\end{align}
Here, \( \tilde{\mathcal{N}}_0^{\cdot}(ds, du) = \mathcal{N}_0^{\cdot}(ds, du) - ds\,du \) denotes the compensated Poisson martingale measures. 

From Theorem A.3 of \cite{bansaye2015stochastic}, we know that \( E_0(t) \) is a square-integrable martingale with quadratic variation:
\begin{align}
\langle E_0 \rangle_t = \int_0^t \frac{X_0(s)}{K y_0(s)^2}(r_0 + d_0) \, ds.
\end{align}

\noindent
By the Burkholder–Davis–Gundy inequality and Jensen's inequality, we obtain:
\begin{align*}
\mathbb{E} \left[ \sup_{t \leq \zeta_n + \delta} \left| \frac{X_0(t)}{y_0(t)} - 1 \right| \right]
&= \mathbb{E} \left[ \sup_{t \leq \zeta_n + \delta} \left| E_0(t) \right| \right] \\
&\leq C_1 \, \mathbb{E} \left[ \langle E_0 \rangle_{\zeta_n + \delta}^{1/2} \right] \\
&= C_1 \, \mathbb{E} \left[ \left( \int_0^{\zeta_n + \delta} \frac{X_0(s)}{K y_0(s)^2}(r_0 + d_0) \, ds \right)^{1/2} \right] \\
&\leq C_1 \left( \int_0^{\zeta_n + \delta} \frac{\mathbb{E}[X_0(s)]}{K y_0(s)^2}(r_0 + d_0) \, ds \right)^{1/2} \\
&= C_2 K^{-1/2} \left( \int_0^{\zeta_n + \delta} \frac{1}{y_0(s)} \, ds \right)^{1/2} \\
&\leq C_3 n^{-1/2} e^{-\lambda_0 \zeta_n / 2},
\end{align*}
where $C_1$, $C_2$, and $C_3$ are positive constants. The final inequality uses the fact that $\mathbb{E}[X_0(s)] = y_0(s) = n e^{\lambda_0 s}/K$.
% and we used the fact that \( \mathbb{E}[X_0(s)] = y_0(s) =ne^{\lambda_0s}/K\), since \( Z_0(t) \) is a subcritical branching process with exponential decay.
We then apply Lemma~\ref{lem:zeta_upper_bound} together with the assumptions that \(1 - \beta < -\frac{\lambda_1}{\lambda_0}, 1 - \beta < \alpha \) to conclude that 
\[
n^{-1/2} e^{-\lambda_0 \zeta_n / 2} \to 0 \quad \text{as } n \to \infty.
\]
% Finally, by Lemma~\ref{lem:zeta_upper_bound} and the assumption that \(1 - \beta < -\frac{\lambda_1}{\lambda_0}, 1 - \beta < \alpha \), it follows that
% \[
% n^{-1/2} e^{-\lambda_0 \zeta_n / 2} \to 0 \quad \text{as } n \to \infty.
% \]
% \begin{remark}
%     \textbf{An upper bound of $\zeta_n$: } We will need to have an approximation of $\zeta_n$ to know if $n^{-1/2}e^{-\lambda_0\zeta_n/2} \rightarrow 0$. Recall that previously we had $\zeta_n \sim \frac{\alpha}{\lambda_1}\log n$, by setting $\alpha <- \lambda_1/\lambda_0$, $n^{-1/2}e^{-\lambda_0\zeta_n/2}$ should converge to 0.
% \end{remark}
Therefore, for any \( \epsilon > 0 \), we have:
\[
\lim_{n \to \infty} \mathbb{P} \left( \sup_{t \leq \zeta_n+\delta} \left| \frac{X_0(t)}{y_0(t)} - 1 \right| > \epsilon \right) 
\leq \lim_{n \to \infty} \frac{\mathbb{E} \left[ \sup_{t \leq \zeta_n+\delta} \left| \frac{X_0(t)}{y_0(t)} - 1 \right|  \right]}{\epsilon} = 0.
\]

% \subsection*{Total resistant population}
\subsection*{Resistant cells}
Similarly the ratio \( X_1(t)/y_1(t) \) can be expressed as a semimartingale:
% Similarly, the scaled process for the total resistant population \( X_1(t)/y_1(t) \) can also be expressed as a semimartingale:
\begin{align}
\frac{X_1(t)}{y_1(t)} &= 1 
- \int_0^t \left( \frac{X_1(s)}{y_1(s)} \phi(Ky(s)) + \frac{y_0(s) X_1(s)}{y_1(s)^2} n^{-\alpha} \right) ds \nonumber \\
&\quad + \int_0^t \int_0^{\infty} \frac{1}{K y_1(s-)} \mathbbm{1}_{\{ u \leq K X_1(s-) f(X(s-)) \}} \mathcal{N}_1^b(ds, du) \nonumber \\
&\quad - \int_0^t \int_0^{\infty} \frac{1}{K y_1(s-)} \mathbbm{1}_{\{ u \leq K X_1(s-) d_1 \}} \mathcal{N}_1^d(ds, du) \nonumber \\
&\quad + \int_0^t \int_0^{\infty} \frac{1}{K y_1(s-)} \mathbbm{1}_{\{ u \leq K X_0(s-) n^{-\alpha} \}} \mathcal{N}_0^m(ds, du) \nonumber \\
&= 1 + E_1(t) 
+ \int_0^t \frac{X_1(s)}{y_1(s)} \left( \phi(KX(s)) - \phi(Ky(s)) \right) ds\nonumber\\ 
&\quad + \int_0^t \left( \frac{X_0(s)}{y_1(s)} - \frac{y_0(s) X_1(s)}{y_1(s)^2} \right) n^{-\alpha} ds,
\end{align}
where 
\( E_1(t) \) is given by:
\begin{align}
E_1(t) &= \int_0^t \int_0^{\infty} \frac{1}{K y_1(s-)} \mathbbm{1}_{\{ u \leq K X_1(s-) f(X(s-)) \}} \tilde{\mathcal{N}}_1^b(ds, du) \nonumber \\
&\quad - \int_0^t \int_0^{\infty} \frac{1}{K y_1(s-)} \mathbbm{1}_{\{ u \leq K X_1(s-) d_1 \}} \tilde{\mathcal{N}}_1^d(ds, du) \nonumber \\
&\quad + \int_0^t \int_0^{\infty} \frac{1}{K y_1(s-)} \mathbbm{1}_{\{ u \leq K X_0(s-) n^{-\alpha} \}} \tilde{\mathcal{N}}_0^m(ds, du).
\end{align}
It's easy to verify that \( E_1(t) \) is a square-integrable martingale with quadratic variation:
\begin{align}
\langle E_1 \rangle_t 
= \int_0^t \frac{1}{K y_1(s)} \left( \frac{X_1(s)}{y_1(s)} (f(X(s)) + d_1) + \frac{X_0(s)}{y_1(s)} n^{-\alpha} \right) ds.
\end{align}
We then obtain the bound:
\begin{align*}
\sup_{t \leq \zeta_n + \delta} \left| \frac{X_1(t)}{y_1(t)} - 1 \right| 
&\leq \sup_{t \leq \zeta_n + \delta} |E_1(t)| 
+ \int_0^{\zeta_n + \delta} \frac{X_1(s)}{y_1(s)} \left| \phi(KX(s)) - \phi(Ky(s)) \right| ds \\
&\quad + n^{-\alpha} \int_0^{\zeta_n + \delta} \left| \frac{X_0(s)}{y_1(s)} - \frac{y_0(s) X_1(s)}{y_1(s)^2} \right| ds\\
&\leq \sup_{t \leq \zeta_n + \delta} |E_1(t)| 
+ \int_0^{\zeta_n + \delta} C \frac{X_1(s)}{y_1(s)} \left( \left| X_0(s) - y_0(s) \right| + \left| X_1(s) - y_1(s) \right| \right) ds \\
&\quad + n^{-\alpha} \int_0^{\zeta_n + \delta} \left| \frac{X_0(s)}{y_1(s)} - \frac{y_0(s) X_1(s)}{y_1(s)^2} \right| ds\\
&\leq \sup_{t \leq \zeta_n + \delta} |E_1(t)| 
+ \int_0^{\zeta_n + \delta} C \frac{X_1(s)}{y_1(s)} \left( \left| X_0(s) - y_0(s) \right| + \left| X_1(s) - y_1(s) \right| \right) ds \\
&\quad + n^{-\alpha} \int_0^{\zeta_n + \delta} \frac{y_0(s)}{y_1(s)} \left| \frac{X_0(s)}{y_0(s)} - 1 \right| ds 
+ n^{-\alpha} \int_0^{\zeta_n + \delta} \frac{y_0(s)}{y_1(s)} \left| \frac{X_1(s)}{y_1(s)} - 1 \right| ds\\
&\leq \sup_{t \leq \zeta_n + \delta} |E_1(t)| 
 \\
&\quad + \int_0^{\zeta_n + \delta} C \frac{X_1(s)}{y_1(s)} \left| X_0(s) - y_0(s) \right| ds + n^{-\alpha} \int_0^{\zeta_n + \delta} \frac{y_0(s)}{y_1(s)} \left| \frac{X_0(s)}{y_0(s)} - 1 \right| ds \\
&\quad + \int_0^{\zeta_n + \delta} \left( C X_1(s) + n^{-\alpha} \frac{y_0(s)}{y_1(s)} \right) 
\sup_{r \leq s} \left| \frac{X_1(r)}{y_1(r)} - 1 \right| ds,
\end{align*}
where (i) $C$ is some positive constant; (ii) the second inequality comes from the Lipschitz continuity property of \( \phi(Kx) \) (\textit{(A1)} of Assumption~\ref{assump:f}); (iii) the third inequality is obtained by applying the triangle inequality.
% where the second inequality comes from Assumption~\ref{assump:f}\textit{(A1)}, the Lipschitz continuity property of \( \phi(Kx) \) and the third inequality comes from triangle inequality.

By Gronwall’s inequality, we obtain the bound
\begin{align}
    \sup_{t \leq \zeta_n + \delta} \left| \frac{X_1(t)}{y_1(t)} - 1 \right| 
    &\leq \exp\left( \int_0^{\zeta_n + \delta} \left( C X_1(s) + n^{-\alpha} \frac{y_0(s)}{y_1(s)} \right) ds \right) \times \left( \sup_{t \leq \zeta_n + \delta} |E_1(t)|
     \right.\\
    &\quad \left. + \int_0^{\zeta_n + \delta} C \frac{X_1(s)}{y_1(s)} \left| X_0(s) - y_0(s) \right| ds + n^{-\alpha} \int_0^{\zeta_n + \delta} \frac{y_0(s)}{y_1(s)} \left| \frac{X_0(s)}{y_0(s)} - 1 \right| ds \right).\nonumber 
\end{align}
To establish our result, we proceed to analyze the three integral terms and the martingale supremum on the right-hand side.
% To proceed, we analyze the four terms on the right-hand side individually:

\textbf{(1)}\quad $\boldsymbol{\exp\left( \int_0^{\zeta_n + \delta} \left( C X_1(s) + n^{-\alpha} \frac{y_0(s)}{y_1(s)} \right) ds \right)}$

Let \(\hat{Z}_1(t) \) represent a branching process with intrinsic growth rate \( \lambda_1 = r_1 - d_1 \) and immigration from the sensitive population at rate \( n^{-\alpha} Z_0(s) \). Because \( r_1 = f(0, 0) \geq f(Z/K) \) for all \( Z \in \mathbb{R}^+ \times \mathbb{R}^+ \), it follows that 
\begin{align}
\label{eq:upper_bound_E_X1}
\mathbb{E}[X_1(t)] \leq \mathbb{E}[\hat{Z}_1(t)]/K = \frac{n^{1 - \alpha}}{K(\lambda_1 - \lambda_0)} \left( e^{\lambda_1 t} - e^{\lambda_0 t} \right) + \frac{n^{\beta}}{K} e^{\lambda_1 t},
\end{align}
which is a standard result for branching processes with immigration \cite{durrett2015branching}. Applying the upper bound on $\zeta_n$ from Lemma~\ref{lem:zeta_upper_bound}, specifically $\zeta_n < C+\frac{\min\left\{  1-\beta,\alpha\right\}}{\lambda_1}\log n$ for some constant $C>0$ (with a slight abuse of notation, we allow this constant to be redefined and it may differ from the constant $C$ used previously), we obtain
\begin{align*}
    \int_0^{\zeta_n + \delta} \mathbb{E}[X_1(s)] \, ds 
    &\leq \int_0^{ C + \frac{1 - \beta}{\lambda_1} \log n } 
    \left( \frac{n^{1 - \alpha}}{K(\lambda_1 - \lambda_0)} + \frac{n^{\beta}}{K} \right) e^{\lambda_1 s} \, ds \\
    &= O(1).
\end{align*}
Hence, the integral is uniformly bounded in \( n \). Furthermore, applying inequality (\ref{eq:upper_bound_1_over_y1}) and the assumption $1 - \beta < \alpha$, we obtain:
\begin{align*}
    n^{ - \alpha}\int_0^{\zeta_n + \delta} \frac{y_0(s)}{y_1(s)} \, ds 
    &\leq C_1 n^{1 - \alpha - \beta} \int_0^1 e^{\lambda_0 s} \, ds 
    + C_2 n^{-\alpha} \int_1^{\zeta_n + \delta} e^{\lambda_0 s} \, ds \\
    &\quad + C_3 n^{1 - \alpha - \beta} \int_1^{\zeta_n + \delta} e^{(\lambda_0 - \lambda_1)s} \, ds \\
    &= O(n^{1 - \alpha - \beta}).
\end{align*}
We therefore conclude that the following expectation is uniformly bounded in \( n \):
\begin{align}
\label{eq:integral_X1}
    \limsup_{n \to \infty} 
    \mathbb{E} \left[ 
    \int_0^{\zeta_n + \delta} \left( 
    C X_1(s) + n^{-\alpha} \frac{y_0(s)}{y_1(s)} 
    \right) ds 
    \right] < \infty.
\end{align}
\textbf{(2)}\quad $\boldsymbol{\sup\limits_{t \leq \zeta_n + \delta} |E_1(t)|}$

% We now estimate the martingale term. By the Burkholder–Davis–Gundy inequality and Jensen’s inequality, we have:
We now proceed to estimate the martingale term. Applying the Burkholder–Davis–Gundy inequality followed by Jensen's inequality yields:
\begin{align*}
    \mathbb{E}\left[ \sup_{t \leq \zeta_n + \delta} \left| E_1(t) \right| \right] 
    &\leq C \, \mathbb{E}\left[ \langle E_1 \rangle_{\zeta_n + \delta}^{1/2} \right] 
    \leq C \left( \mathbb{E}\left[ \langle E_1 \rangle_{\zeta_n + \delta} \right] \right)^{1/2} \\
    &= C \left( \mathbb{E} \left[ \int_0^{\zeta_n + \delta} \frac{1}{K y_1(s)} \left( 
    \frac{X_1(s)}{y_1(s)} (f(X(s)) + d_1) 
    + \frac{X_0(s)}{y_1(s)} n^{-\alpha} 
    \right) ds \right] \right)^{1/2} \\
    &\leq C \left( \mathbb{E} \left[ \int_0^{\zeta_n + \delta} \frac{1}{K y_1(s)} \left( 
    \frac{X_1(s)}{y_1(s)} (r_1 + d_1) 
    + \frac{X_0(s)}{y_1(s)} n^{-\alpha} 
    \right) ds \right] \right)^{1/2},
\end{align*}
where the final inequality follows from the inequality \( f(X(s)) \leq r_1 \).
% From (\ref{eq:upper_bound_1_over_y1}), we further bound the expectations inside the square root. First,
To bound the expectation in the previous expression, we apply the upper bound from (\ref{eq:upper_bound_1_over_y1}). We first estimate the integral involving the expectation of $X_1(s)$:
\begin{align*}
\frac{1}{K} \int_0^{\zeta_n + \delta} \frac{\mathbb{E}[X_1(s)]}{y_1(s)^2} \, ds 
&\leq C_1 n^{-\beta} \int_0^1 e^{\lambda_1 s} ds 
+ C_2 n^{\beta - 2} \int_1^{\zeta_n + \delta} e^{\lambda_1 s} ds 
+ C_3 n^{-\beta} \int_1^{\zeta_n + \delta} e^{-\lambda_1 s} ds \\
&= O(n^{-\beta}).
% \to 0 \quad \text{as } n \to \infty.
\end{align*}
Similarly, for the second term, with the fact $\EE[X_0]=y_0$, we obtain:
\begin{align*}
\frac{n^{-\alpha}}{K} \int_0^{\zeta_n + \delta} \frac{y_0(s)}{y_1(s)^2} \, ds 
&\leq C_1 n^{1 - \alpha - 2\beta} \int_0^1 e^{\lambda_0 s} ds 
+ C_2 n^{-1 - \alpha} \int_1^{\zeta_n + \delta} e^{\lambda_0 s} ds \\
&\quad + C_3 n^{1 - \alpha - 2\beta} \int_1^{\zeta_n + \delta} e^{(-2\lambda_1 + \lambda_0)s} ds \\
&= O(n^{1 - \alpha - 2\beta}) + O(n^{-1 - \alpha}).
\end{align*}
Therefore, we conclude that
\[
\mathbb{E}\left[ \sup_{t \leq \zeta_n + \delta} \left| E_1(t) \right| \right] = O(n^{-\beta/2}).
\]

\textbf{(3)}\quad $\boldsymbol{\int_0^{\zeta_n + \delta} \frac{X_1(s)}{y_1(s)} |X_0(s) - y_0(s)| \, ds}$

Applying Hölder's inequality and Cauchy–Schwarz inequality to the expectation yields:
\begin{align*}
    & \quad \mathbb{E} \left[ \int_0^{\zeta_n + \delta} \frac{X_1(s)}{y_1(s)} \left| X_0(s) - y_0(s) \right| ds \right] \\
    &\leq \mathbb{E} \left[ \left( \int_0^{\zeta_n + \delta} \left( \frac{X_1(s)}{y_1(s)} \right)^2 ds \right)^{1/2} 
    \left( \int_0^{\zeta_n + \delta} \left( X_0(s) - y_0(s) \right)^2 ds \right)^{1/2} \right] \\
    &\leq \left( \mathbb{E} \left[ \int_0^{\zeta_n + \delta} \left( \frac{X_1(s)}{y_1(s)} \right)^2 ds \right] \right)^{1/2} 
    \left( \mathbb{E} \left[ \int_0^{\zeta_n + \delta} \left( X_0(s) - y_0(s) \right)^2 ds \right] \right)^{1/2} \\
    &= \left( \int_0^{\zeta_n + \delta} \frac{\mathbb{E}[X_1(s)^2]}{y_1(s)^2} ds \right)^{1/2}
    \left( \int_0^{\zeta_n + \delta} \mathbb{E}\left[ \left( X_0(s) - y_0(s) \right)^2 \right] ds \right)^{1/2} \\
    &\leq C \left( \int_0^{\zeta_n + \delta} \frac{\mathbb{E}[X_1(s)^2]}{y_1(s)^2} ds \right)^{1/2} \cdot \frac{1}{\sqrt{n}},
\end{align*}
where the last inequality follows from a second moment calculation for subcritical branching processes (see \cite{durrett2015branching}). In particular, the second moment of $X_0(t)$ is given by
% In particular, we have ((see \cite{bp}))
\begin{align}
&\mathbb{E}[X_0(t)^2] = \frac{n}{K^2} \cdot \frac{r_0 + d_0}{\lambda_0} e^{\lambda_0 t} \left( e^{\lambda_0 t} - 1 \right) 
+ \frac{n^2}{K^2} e^{2 \lambda_0 t}.
\end{align}
Consequently, the variance satisfies the bound
\begin{align}
    \mathbb{E}\left[ \left( X_0(s) - y_0(s) \right)^2 \right]\leq \frac{n}{K^2} \cdot \frac{r_0 + d_0}{|\lambda_0|} e^{\lambda_0 t}.
\end{align}

% Following a similar argument as (\ref{eq:upper_bound_E_X1}), we know
Following an argument analogous to that used in (\ref{eq:upper_bound_E_X1}), we establish the bound
\begin{align}
\label{eq:upper_bound_E_X1_square}
\mathbb{E}[X_1(s)^2] \leq \frac{1}{K^2} \, \mathbb{E}[\hat{Z}_1(s)^2] 
\leq C \left( n^{2 - 2\alpha} e^{2 \lambda_1 s} + n^{2\beta} e^{2 \lambda_1 s} \right)/K^2.
\end{align}
% Combined with (\ref{eq:upper_bound_1_over_y1}),
Substituting this bound, together with the bound from (\ref{eq:upper_bound_1_over_y1}), yields
\begin{align*}
\int_0^{\zeta_n + \delta} \frac{\mathbb{E}[X_1(s)^2]}{y_1(s)^2} ds 
&\leq C_1 \int_0^1 e^{2\lambda_1 s} ds 
+ C_2 n^{2\beta - 2} \int_1^{\zeta_n + \delta} e^{2\lambda_1 s} ds 
+ C_3 \int_1^{\zeta_n + \delta} ds \\
&= O(\log n).
\end{align*}
Consequently, we obtain the final estimate:
% which gives the final estimate:
\[
\mathbb{E} \left[ \int_0^{\zeta_n + \delta} \frac{X_1(s)}{y_1(s)} \left| X_0(s) - y_0(s) \right| ds \right] = O(n^{-1/2} \sqrt{\log n}).
\]

\textbf{(4)}\quad $\boldsymbol{ n^{-\alpha} \int_0^{\zeta_n + \delta} \frac{y_0(s)}{y_1(s)} \left| \frac{X_0(s)}{y_0(s)} - 1 \right| ds}$

% We now bound the mutation correction term. Using Hölder's inequality and  (\ref{eq:upper_bound_1_over_y1}), we have:
Applying Hölder's inequality and the bound from (\ref{eq:upper_bound_1_over_y1}) gives:
\begin{align*}
    & \quad \mathbb{E} \left[ n^{-\alpha} \int_0^{\zeta_n + \delta} \frac{y_0(s)}{y_1(s)} \left| \frac{X_0(s)}{y_0(s)} - 1 \right| ds \right] \\
    &= n^{-\alpha} \int_0^{\zeta_n + \delta} \frac{1}{y_1(s)} \mathbb{E} \left[ \left| X_0(s) - y_0(s) \right| \right] ds \\
    &\leq n^{-\alpha} \left( \int_0^{\zeta_n + \delta} \frac{1}{y_1(s)^2} ds \right)^{1/2} 
    \left( \int_0^{\zeta_n + \delta} \mathbb{E} \left[ \left( X_0(s) - y_0(s) \right)^2 \right] ds \right)^{1/2} \\
    &\leq C n^{-\alpha - 1/2} \left( \int_0^{\zeta_n + \delta} \frac{1}{y_1(s)^2} ds \right)^{1/2}.
    % &= O(n^{1/2 - \alpha - \beta}).
\end{align*}
Applying (\ref{eq:upper_bound_1_over_y1}) again to the remaining integral:
\begin{align*}
    \int_0^{\zeta_n + \delta} \frac{1}{y_1(s)^2} ds 
    &\leq C_1 n^{2 - 2\beta} \int_0^1 ds 
    + C_2 \int_1^{\zeta_n + \delta} ds 
    + C_3 n^{2 - 2\beta} \int_1^{\zeta_n + \delta} e^{-2\lambda_1 s} ds \\
    &= O(n^{2 - 2\beta}).
\end{align*}
Combining these results, we obtain:
\[
\mathbb{E} \left[ n^{-\alpha} \int_0^{\zeta_n + \delta} \frac{y_0(s)}{y_1(s)} \left| \frac{X_0(s)}{y_0(s)} - 1 \right| ds \right] = O(n^{1/2 - \alpha - \beta}).
\]

% Now we can combine the four terms and analyze their joint behavior. For simplification, we define the following quantities:
We now combine the four terms analyzed above to derive their joint asymptotic behavior. To simplify the presentation, we define the following two quantities:
\begin{align*}
    A(n) &:= n^u\left(\sup_{t \leq \zeta_n + \delta} |E_1(t)|
    + \int_0^{\zeta_n + \delta} C \frac{X_1(s)}{y_1(s)} \left| X_0(s) - y_0(s) \right| ds  
    + n^{-\alpha} \int_0^{\zeta_n + \delta} \frac{y_0(s)}{y_1(s)} \left| \frac{X_0(s)}{y_0(s)} - 1 \right| ds\right), \\
    B(n) &:= \int_0^{\zeta_n + \delta} \left( 
    C X_1(s) + n^{-\alpha} \frac{y_0(s)}{y_1(s)} 
    \right) ds.
\end{align*}
From the previous analysis, specifically, the bounds of order $O(n^{-\beta/2})$, $O(n^{-1/2} \sqrt{\log n})$, and $O(n^{1/2 - \alpha - \beta})$, we obtain that for any exponent $u<\beta/2$, the following holds:
% From previous estimates, we have: \textcolor{red}{(add detail for $u$)}
\[
\lim_{n \to \infty} \mathbb{E}[A(n)] = 0 \quad \text{and} \quad \limsup_{n \to \infty} \mathbb{E}[B(n)] < \infty.
% \quad\text{when}\quad u<\beta/2.
\]
Because \( A(n) \geq 0 \), Markov’s inequality implies that \( A(n) \xrightarrow{P} 0 \). Now, for any \( \epsilon > 0 \) and \( \delta > 0 \), we bound the probability as follows:
\[
\mathbb{P}\left( A(n) \cdot \exp(B(n)) > \epsilon \right) 
\leq \mathbb{P}(A(n) > \delta) + \mathbb{P}\left( \exp(B(n)) > \frac{\epsilon}{\delta} \right).
\]
The first term converges to zero as \( n \to \infty \) by the convergence in probability of \( A(n)\). For the second term, applying Markov’s inequality gives:
\[
\mathbb{P}\left( \exp(B(n)) > \frac{\epsilon}{\delta} \right) 
= \mathbb{P}\left( B(n) > \log\left( \frac{\epsilon}{\delta} \right) \right) 
\leq \frac{\mathbb{E}[B(n)]}{\log\left( \frac{\epsilon}{\delta} \right)}.
\]
Because \( \sup_n \mathbb{E}[B(n)] \leq C < \infty \), we can make this bound arbitrarily small by choosing \( \delta > 0 \) sufficiently small (thus making the logarithmic term arbitrarily large). Therefore,
\[
\mathbb{P}\left( A(n) \cdot \exp(B(n)) > \epsilon \right) \to 0,
\]
which implies \( A(n) \cdot \exp(B(n)) \xrightarrow{P} 0 \).

Recalling the Gronwall-type inequality
\[
n^u\sup_{t \leq \zeta_n + \delta} \left| \frac{X_1(t)}{y_1(t)} - 1 \right| \leq A(n) \cdot \exp(B(n)),
\]
we thus conclude that
\[
\lim_{n \to \infty} \mathbb{P} \left( n^u\sup_{t \leq \zeta_n+\delta} \left| \frac{X_1(t)}{y_1(t)} - 1 \right| > \epsilon \right) = 0,
\]
which establishes the desired convergence in probability.

\subsection*{Pre-existing resistant clone}
The ratio $X_{\beta}(t)/y_{\beta}(t)$ can be expressed as a semimartingale:
\begin{align}
    \frac{X_{\beta}(t)}{y_{\beta}(t)} 
    &= 1 
    - \int_0^t \frac{X_{\beta}(s)}{y_{\beta}(s)} \phi(K y(s)) \, ds \nonumber \\
    &\quad + \int_0^t \int_0^{\infty} \frac{1}{K y_{\beta}(s-)} \mathbbm{1}_{\left\{ u \leq K X_{\beta}(s-) f( X(s-)) \right\}} \mathcal{N}_1^b(ds, du) \nonumber \\
    &\quad - \int_0^t \int_0^{\infty} \frac{1}{K y_{\beta}(s-)} \mathbbm{1}_{\left\{ u \leq K X_{\beta}(s-) d_1 \right\}} \mathcal{N}_1^d(ds, du\nonumber\\
    &= 1 + E_{\beta}(t) 
    + \int_0^t \frac{X_{\beta}(s)}{y_{\beta}(s)}\left( \phi(KX(s)) - \phi(Ky(s)) \right) ds,
\end{align}
where \( E_{\beta}(t) \) is given by:
\begin{align}
    E_{\beta}(t) 
    &= \int_0^t \int_0^{\infty} \frac{1}{K y_{\beta}(s-)} 
    \mathbbm{1}_{\left\{ u \leq K X_{\beta}(s-) f(X(s-)) \right\}} 
    \tilde{\mathcal{N}}_1^b(ds, du) \nonumber\\
    &\quad - \int_0^t \int_0^{\infty} \frac{1}{K y_{\beta}(s-)} 
    \mathbbm{1}_{\left\{ u \leq K X_{\beta}(s-) d_1 \right\}} 
    \tilde{\mathcal{N}}_1^d(ds, du).
\end{align}
Then we have
\begin{align*}
    &\sup_{t \leq \zeta_n + \delta} \left| \frac{X_{\beta}(t)}{y_{\beta}(t)} - 1 \right|\\ 
    \leq& \sup_{t \leq \zeta_n + \delta} \left| E_{\beta}(t) \right| 
    + \int_0^{\zeta_n + \delta} \frac{X_{\beta}(s)}{y_{\beta}(s)} \left| \phi(KX(s)) - \phi(Ky(s)) \right| ds \\
    \leq& \sup_{t \leq \zeta_n + \delta} \left| E_{\beta}(t) \right| 
    + \int_0^{\zeta_n + \delta} C \frac{X_{\beta}(s)}{y_{\beta}(s)} \left( |X_0(s) - y_0(s)| + |X_1(s) - y_1(s)| \right) ds \\
    \leq &\sup_{t \leq \zeta_n + \delta} \left| E_{\beta}(t) \right| 
    + \int_0^{\zeta_n + \delta} C \frac{X_{\beta}(s)}{y_{\beta}(s)} \left( |X_0(s) - y_0(s)| + |X_{\beta}(s) - y_{\beta}(s)| + |X_m(s) - y_m(s)| \right) ds \\
    \leq &\sup_{t \leq \zeta_n + \delta} \left| E_{\beta}(t) \right| 
    + \int_0^{\zeta_n + \delta} C \frac{X_{\beta}(s)}{y_{\beta}(s)} |X_0(s) - y_0(s)| ds 
    + \int_0^{\zeta_n + \delta} C \frac{X_{\beta}(s)}{y_{\beta}(s)} |X_m(s) - y_m(s)| ds \\
    &\quad + \int_0^{\zeta_n + \delta} C X_{\beta}(s) \sup_{r \leq s} \left| \frac{X_{\beta}(r)}{y_{\beta}(r)} - 1 \right| ds,
\end{align*}
where $y_m(t):=y_1(t)-y_{\beta}(t)$. Applying Gronwall’s inequality to the system yields the following bound:
\begin{align*}
    \sup_{t \leq \zeta_n+\delta} \left| \frac{X_{\beta}(t)}{y_{\beta}(t)} - 1 \right| 
    &\leq C\left( 
        \sup_{t \leq \zeta_n+\delta} \left| E_{\beta}(t) \right| 
        + \int_0^{\zeta_n+\delta} \frac{X_{\beta}(s)}{y_{\beta}(s)} |X_0(s) - y_0(s)| ds 
        + \int_0^{\zeta_n+\delta} \frac{X_{\beta}(s)}{y_{\beta}(s)} |X_m(s) - y_m(s)| ds 
    \right) \\
    &\quad \times \exp\left( \int_0^{\zeta_n+\delta} C X_{\beta}(s) ds \right).
\end{align*}
We now bound the expectation of the martingale term. Applying the Burkholder–Davis–Gundy inequality followed by Jensen's inequality gives:
\begin{align*}
    \mathbb{E} \left[ \sup_{t \leq \zeta_n + \delta} |E_{\beta}(t)| \right] 
    &\leq C \, \mathbb{E} \left[ \langle E_{\beta} \rangle_{\zeta_n + \delta}^{1/2} \right] 
    \leq C \left( \mathbb{E} \left[ \langle E_{\beta} \rangle_{\zeta_n + \delta} \right] \right)^{1/2} \\
    &= C \left( \mathbb{E} \left[ \int_0^{\zeta_n + \delta} \frac{1}{K y_{\beta}(s)} \cdot \frac{X_{\beta}(s)}{y_{\beta}(s)} \left( f(X(s)) + d_1 \right) ds \right] \right)^{1/2} \\
    &\leq C \left( \int_0^{\zeta_n + \delta} \frac{ \mathbb{E}[X_{\beta}(s)] }{K y_{\beta}(s)^2} (r_1 + d_1) \, ds \right)^{1/2} \\
    &= C (r_1 + d_1)^{1/2} \left( \int_0^{\zeta_n + \delta} \frac{ \mathbb{E}[X_{\beta}(s)] }{K y_{\beta}(s)^2} \, ds \right)^{1/2}.
\end{align*}

Since the mutation term in the ODE system~\eqref{ODEs} vanishes as \( \alpha \to \infty \), the function \( y_1(t) \) converges to \( y_{\beta}(t) \). This convergence implies that the established upper bound for \( 1/y_1(t) \) in~\eqref{eq:upper_bound_1_over_y1} yields a corresponding bound for \( 1/y_{\beta}(t) \) in the limit:
% As observed from the ODE system~\eqref{ODEs}, the functions \( y_1(t) \) and \( y_{\beta}(t) \) coincide in the limit \( \alpha \to \infty \), since the mutation term vanishes. Therefore, by taking the limit \( \alpha \to \infty \) in the upper bound for \( 1/y_1(t) \) given in~\eqref{eq:upper_bound_1_over_y1}, we obtain a corresponding upper bound for \( 1/y_{\beta}(t) \). Specifically, we have:
\begin{align}
\label{eq:upper_bound_1_over_ybeta}
    \frac{1}{y_{\beta}(t)} \leq c_1 + c_2 n^{1 - \beta} e^{-\lambda_1 t},
\end{align}
where the constant \( c_2 \) has been adjusted to account for the behavior on \( t < 1 \).

By an analogous argument, taking the limit \( \alpha \to \infty \) or \( \beta \to -\infty \) in estimates~\eqref{eq:upper_bound_E_X1} and~\eqref{eq:upper_bound_E_X1_square} yields the following moment bounds for the pre-existing and mutation-derived resistant populations:
% Following the same reasoning, by taking the limit \( \alpha \to \infty \) or \( \beta \to -\infty \) in the estimates~\eqref{eq:upper_bound_E_X1} and~\eqref{eq:upper_bound_E_X1_square}, we obtain the following bounds for the moments of the pre-existing and mutation-derived resistant populations:
\begin{align}
\label{eq:first_mom_Xbeta}
    \mathbb{E}[X_{\beta}(t)] &\leq C n^{\beta - 1} e^{\lambda_1 t}, \\
    \label{eq:second_mom_Xbeta}
    \mathbb{E}[X_{\beta}(t)^2] &\leq C n^{2\beta - 2} e^{2\lambda_1 t}, \\
    \label{eq:first_mom_Xm}
    \mathbb{E}[X_m(t)] &\leq C n^{-\alpha} e^{\lambda_1 t},\\
    \label{eq:second_mom_Xm}
    \mathbb{E}[X_m(t)^2] &\leq C n^{-2\alpha} e^{2\lambda_1 t},
\end{align}
where \( C > 0 \) is a constant independent of \( n \) and \( t \). We therefore establish the bound
\begin{align*}
    \int_0^{\zeta_n + \delta} \frac{\mathbb{E}[X_{\beta}(s)]}{K y_{\beta}(s)^2} ds &= O(n^{-\beta}),
    % \Longrightarrow \mathbb{E} \left[ \sup_{t \leq \zeta_n + \delta} |E_{\beta}(t)| \right] &= O(n^{-\beta/2}).
\end{align*}
which implies
\begin{align}
    \mathbb{E} \left[ \sup_{t \leq \zeta_n + \delta} |E_{\beta}(t)| \right] &= O(n^{-\beta/2}).
\end{align}

To bound the second and third error terms, we apply the Cauchy–Schwarz inequality:
\begin{align*}
    \mathbb{E} \left[ \int_0^{\zeta_n+\delta} \frac{X_{\beta}(s)}{y_{\beta}(s)} \left| X_0(s) - y_0(s) \right| ds \right]
    &\leq \left( \int_0^{\zeta_n+\delta} \frac{\mathbb{E}[X_{\beta}(s)^2]}{y_{\beta}(s)^2} ds \right)^{1/2}
    \left( \int_0^{\zeta_n+\delta} \mathbb{E} \left[ \left( X_0(s) - y_0(s) \right)^2 \right] ds \right)^{1/2} \\
    &= O(n^{-1/2} \sqrt{\log n}), \\[1em]
    \mathbb{E} \left[ \int_0^{\zeta_n+\delta} \frac{X_{\beta}(s)}{y_{\beta}(s)} \left| X_m(s) - y_m(s) \right| ds \right]
    &\leq \left( \int_0^{\zeta_n+\delta} \frac{\mathbb{E}[X_{\beta}(s)^2]}{y_{\beta}(s)^2} ds \right)^{1/2}
    \left( \int_0^{\zeta_n+\delta} \mathbb{E} \left[ \left( X_m(s) - y_m(s) \right)^2 \right] ds \right)^{1/2} \\
    &= O(n^{1-\alpha-\beta} \sqrt{\log n}).
\end{align*}
These bounds follow from the estimate
\begin{align*}
    \int_0^{\zeta_n+\delta} \frac{\mathbb{E}[X_{\beta}(s)^2]}{y_{\beta}(s)^2} ds &= O(\log n),
\end{align*}
and
\begin{align}
    \mathbb{E} \left[ \left( X_m(s) - y_m(s) \right)^2 \right]
    \leq \mathbb{E} \left[ X_m(s)^2 \right] + y_m(s)^2
    \leq C n^{-2\alpha} e^{2\lambda_1 s},
\end{align}
where the upper bound on \( y_m(s)^2 \) follows from the assumption \( \phi(Ky(s)) \leq \lambda_1 \). Lastly, since $X_{\beta}(s)\le X_1(s)$ for all $s$, the boundedness result from (\ref{eq:integral_X1}) implies
\begin{align}
    \limsup_{n \to \infty} 
    \mathbb{E} \left[ 
        \int_0^{\zeta_n + \delta} X_{\beta}(s) \, ds 
    \right]
    < 
    \limsup_{n \to \infty} 
    \mathbb{E} \left[ 
        \int_0^{\zeta_n + \delta} X_1(s) \, ds 
    \right] < \infty.
\end{align}

To simplify the presentation, we define the following quantities:
\begin{align*}
    A(n) &:= n^u\left(\sup_{t \leq \zeta_n+\delta} \left| E_{\beta}(t) \right| 
        + \int_0^{\zeta_n+\delta} \frac{X_{\beta}(s)}{y_{\beta}(s)} |X_0(s) - y_0(s)| ds 
        + \int_0^{\zeta_n+\delta} \frac{X_{\beta}(s)}{y_{\beta}(s)} |X_m(s) - y_m(s)| ds \right), \\
    B(n) &:= \int_0^{\zeta_n+\delta} C X_{\beta}(s) ds.
\end{align*}
From the previous analysis, we obtain that for any exponent $u <\min\{ \beta/2 ,\alpha+\beta-1\}$, the following holds:
\[
\lim_{n \to \infty} \mathbb{E}[A(n)] = 0 \quad \text{and} \quad \limsup_{n \to \infty} \mathbb{E}[B(n)] < \infty.
% \quad\text{when}\quad  u <\min\{ \beta/2 ,\alpha+\beta-1\}.
\]
Following the same Gronwall inequality argument applied to the total resistant population, we conclude that
\[
\lim_{n \to \infty} \mathbb{P} \left( n^u\sup_{t \leq \zeta_n+\delta} \left| \frac{X_{\beta}(t)}{y_{\beta}(t)} - 1 \right| > \epsilon \right) = 0.
\]
\qed
\end{proof}

\section{Proof of Proposition \ref{prop:convergence_of_gamma_n}}
\label{sec:pf_prop_gamma_conv}
\begin{proof}
Let \( \epsilon_n := \epsilon n^{-u} \). We decompose the probability into two parts:
\begin{align}
\mathbb{P}\left( |\gamma_n - \zeta_n| > \epsilon_n \right) 
= \mathbb{P}\left( \gamma_n > \zeta_n + \epsilon_n \right) + \mathbb{P}\left( \gamma_n < \zeta_n - \epsilon_n \right).
\end{align}

\textbf{(1)} We begin by bounding the first term. From the definition of \( \gamma_n \) in (\ref{def:gamma}), it follows that:
\begin{align*}
\mathbb{P}\left( \gamma_n > \zeta_n + \epsilon_n \right) 
&= \mathbb{P}\left( \sup_{0 \leq t \leq \zeta_n + \epsilon_n} X_1(t) < \frac{n}{K} \right) \\
&= \mathbb{P}\left( \sup_{0 \leq t \leq \zeta_n + \epsilon_n} \frac{X_1(t)}{y_1(t)} \cdot y_1(t) < \frac{n}{K} \right) \\
&\leq \mathbb{P}\left( \inf_{0 \leq t \leq \zeta_n + \epsilon_n} \frac{X_1(t)}{y_1(t)} \cdot \sup_{0 \leq t \leq \zeta_n + \epsilon_n} y_1(t) < \frac{n}{K} \right).
\end{align*}
Because \( y_1(t) \) is a monotonic increasing, its supremum over the interval is attained at the endpoint:
\[
\sup_{0 \leq t \leq \zeta_n + \epsilon_n} y_1(t) = y_1(\zeta_n + \epsilon_n).
\]
Substituting this expression yields the bound:
\begin{align}
\mathbb{P}\left( \gamma_n > \zeta_n + \epsilon_n \right) 
\leq \mathbb{P}\left( \inf_{0 \leq t \leq \zeta_n + \epsilon_n} \frac{X_1(t)}{y_1(t)} < \frac{n}{K y_1(\zeta_n + \epsilon_n)} \right).
\end{align}
%    We now define an auxiliary ODE system with solution \( \bar{y}_1(t) \), given by:
% \begin{align*}
% \begin{cases}
% \displaystyle \frac{d\bar{y}_1}{dt} = \lambda_1(1 - (\bar{y}_0(t) + \bar{y}_1(t))) \bar{y}_1(t), \\
% \bar{y}_1(0) = \frac{n}{K},
% \end{cases}
% \end{align*}
% where we set \( \bar{y}_0(t) := \frac{n}{K} e^{\lambda_0 \zeta_n} e^{\lambda_0 t} \leq \frac{n}{K} e^{\lambda_0 \zeta_n} \) for all \( t \geq 0 \). 
% \smallskip
% \noindent
% It's not hard to find that we have \( y_1(\zeta_n + t) \geq \bar{y}_1(t) \).
% \smallskip
% \noindent

By Lemma~\ref{lem:zeta_lower_bound}, there exists a constant \( c > 0 \) such that for all sufficiently large \( n \), the following holds:
\begin{align}
\label{eq:zeta_n_lower}
\zeta_n > \zeta_n - \epsilon_n > c \log n, \text{ and } \frac{n}{K} n^{c\lambda_0} < \frac{1}{4}\left(1 - \frac{n}{K}\right).
\end{align}
% This implies the following upper bound for \( y_0(t + \zeta_n) \):
This yields the upper bound
\begin{align}
\label{eq:upper_y_0}
y_0(t+\zeta_n) = \frac{n}{K} e^{\lambda_0 \zeta_n} e^{\lambda_0 t} \leq \frac{n}{K} n^{c\lambda_0} < \frac{1}{4} \left(1 - \frac{n}{K} \right), \quad \text{for } t \geq 0.
\end{align}

% Now consider the auxiliary ODE defined in~\eqref{ODEs:y_bar}. We focus on the regime where \( \bar{y}_1(t+\zeta_n) \leq \frac{n}{K} + \frac{1}{4}(1 - \frac{n}{K}) \). From the upper bound \eqref{eq:upper_y_0}, we deduce:
Now consider the auxiliary ODE system defined in~\eqref{ODEs:y_bar}. We examine the regime where \( \bar{y}_1(t+\zeta_n) \leq \frac{n}{K} + \frac{1}{4}(1 - \frac{n}{K}) \). Using the upper bound from \eqref{eq:upper_y_0}, we obtain:
\[
1 - (y_0(t+\zeta_n) + \bar{y}_1(t+\zeta_n)) 
> 1 - \left( \frac{1}{4} (1 - \frac{n}{K}) + \frac{n}{K} + \frac{1}{4}(1 - \frac{n}{K}) \right) 
= \frac{1}{2} \left(1 - \frac{n}{K} \right).
\]
Consequently, the following differential inequality holds:
\begin{align*}
\frac{d\bar{y}_1(t+\zeta_n)}{dt} 
\geq \frac{1}{2} \lambda_1 \left(1 - \frac{n}{K} \right) \bar{y}_1(t+\zeta_n).
\end{align*}
Integrating this inequality yields the lower bound:
\begin{align*}
\bar{y}_1(t+\zeta_n) \geq \min\left\{  \bar{y}_1(\zeta_n) \cdot e^{ \frac{1}{2} \lambda_1 \left(1 - \frac{n}{K} \right)t }, \; \frac{n}{K} + \frac{1}{4} \left(1 - \frac{n}{K} \right) \right\}.
\end{align*}

Although the auxiliary ODE solution satisfies \( \bar{y}_1(\zeta_n) \leq \frac{n}{K} \), we impose the initial condition \( \bar{y}_1(\zeta_n) = \frac{n}{K} \) to match the known value \( y_1(\zeta_n) = \frac{n}{K} \). This choice preserves the lower bound for all \( t \geq \zeta_n \), yielding:
\begin{align}
y_1(\zeta_n + \epsilon_n) \geq \bar{y}_1(\zeta_n+\epsilon_n) 
\geq \min\left\{ \frac{n}{K} \cdot e^{ \frac{1}{2} \lambda_1 \left(1 - \frac{n}{K} \right)\epsilon_n }, \; \frac{n}{K} + \frac{1}{4} \left(1 - \frac{n}{K} \right) \right\}.
\end{align}
% From this lower bound on \( y_1(\zeta_n + \epsilon_n) \), it follows that:
From this bound, it follows that:
\[
\frac{n}{K y_1(\zeta_n + \epsilon_n)} \leq \max\left\{ e^{\frac{1}{2} \lambda_1 \left( \frac{n}{K} - 1 \right) \epsilon_n}, \; \frac{1}{1 + \frac{1}{4} \left( \frac{K}{n} - 1 \right)} \right\}.
\]
For sufficiently large \( n \), the right-hand side is bounded above by \( 1 - \varepsilon n^{-u} \) for some \( \varepsilon > 0 \). Hence,
\[
\mathbb{P}(\gamma_n > \zeta_n + \epsilon_n)
\leq \mathbb{P}\left( \inf_{0 \leq t \leq \zeta_n + \epsilon_n} \frac{X_1(t)}{y_1(t)} < 1 - \varepsilon n^{-u} \right).
\]
By Theorem~\ref{thm:ratio_conv}, the right-hand side converges to zero as \( n \to \infty \). We conclude that
\[
\lim_{n \to \infty} \mathbb{P}(\gamma_n > \zeta_n + \epsilon_n) = 0.
\]

\textbf{(2)} For the second term, we have:
\begin{align*}
\mathbb{P}(\gamma_n < \zeta_n - \epsilon_n) 
&= \mathbb{P}\left( \sup_{0 \leq t \leq \zeta_n - \epsilon_n} X_1(t) > \frac{n}{K} \right) \\
&= \mathbb{P}\left( \sup_{0 \leq t \leq \zeta_n - \epsilon_n} \frac{X_1(t)}{y_1(t)} \cdot y_1(t) > \frac{n}{K} \right) \\
&\leq \mathbb{P}\left( \sup_{0 \leq t \leq \zeta_n - \epsilon_n} \frac{X_1(t)}{y_1(t)} \cdot y_1(\zeta_n - \epsilon_n) > \frac{n}{K} \right). \\
\end{align*}
To bound \( y_1(\zeta_n - \epsilon_n) \), consider the interval \( t \in [\zeta_n - \epsilon_n, \zeta_n] \), where the dynamics satisfy:
\begin{align}
\frac{dy_1}{dt} = y_1 \cdot \phi(Ky) + y_0 \cdot n^{-\alpha}
\geq \lambda_1 \left( 1 - (y_0 + y_1) \right) y_1.
\end{align}
Because \( y_0(t) \) is decreasing and \( y_1(t) \) is increasing, it follows that:
\[
y_0(t) + y_1(t) \leq y_0(\zeta_n - \epsilon_n) + y_1(\zeta_n) 
\leq \frac{n}{K} e^{\lambda_0(\zeta_n - \epsilon_n)} + \frac{n}{K}.
\]
Using the bounds from \eqref{eq:zeta_n_lower} and \eqref{eq:upper_y_0}, we obtain that for sufficiently large \( n \), this sum is bounded above by $\frac{n}{K} + \frac{1}{4} \left( 1 - \frac{n}{K} \right)$. Therefore,
\[
\frac{dy_1}{dt} \geq \frac{3}{4} \lambda_1 \left( 1 - \frac{n}{K} \right) y_1.
\]
Integrating this inequality backward from \( \zeta_n \) yields:
\begin{align}
y_1(\zeta_n - \epsilon_n) \leq e^{\frac{3}{4} \lambda_1 \left( \frac{n}{K} - 1 \right) \epsilon_n} y_1(\zeta_n)
= e^{\frac{3}{4} \lambda_1 \left( \frac{n}{K} - 1 \right) \epsilon_n} \cdot \frac{n}{K}.
\end{align}
Therefore, the probability can be bounded as:
\begin{align*}
\mathbb{P}(\gamma_n < \zeta_n - \epsilon_n)
&\leq \mathbb{P}\left( \sup_{0 \leq t \leq \zeta_n - \epsilon_n} \frac{X_1(t)}{y_1(t)} > 
e^{\frac{3}{4} \lambda_1 \left( 1 - \frac{n}{K} \right) \epsilon_n} \right).
\end{align*}
For sufficiently large \( n \), the exponential lower bound derived above satisfies
\begin{align*}
    e^{\frac{3}{4} \lambda_1 \left( 1 - \frac{n}{K} \right) \epsilon_n}>1 + \varepsilon n^{-u}
\end{align*}
for some constant \( \varepsilon > 0 \). It follows from Theorem~\ref{thm:ratio_conv} that
\[
\lim_{n \to \infty} \mathbb{P}(\gamma_n < \zeta_n - \epsilon_n) = 0.
\]
\qed
\end{proof}

\section{Proof of Proposition \ref{prop:bound_of_In}}\label{sec:pf_prop_bound_In}
\begin{proof}
By Theorem~\ref{thm:ratio_conv} and Proposition~\ref{prop:convergence_of_gamma_n}, for any \( \varepsilon, \delta > 0 \), there exists \( n_0 > 0 \) such that for all \( n > n_0 \),
\[
\mathbb{P} \left( \sup_{t \leq \zeta_n+\delta} \frac{X_0(t) + X_1(t)}{y_0(t) + y_1(t)} < 1 + \varepsilon, \gamma_n<\zeta_n+\delta \right) > 1-\varepsilon.
\]
From inequality (\ref{eq:upper_bound_y0plusy1}) in Lemma~\ref{lem:zeta_upper_bound}, there further exists \( n_1 > 0 \) such that for all \( n > n_1 \),
 \[
\sup_{t\leq \zeta_n+\delta} K y_0(t) + K y_1(t) \leq \max \left\{ n + n^{\beta}, \, K y_0(\bar{\zeta}_n) + n e^{\lambda_1\delta} \right\},
 \]
where \( K y_0(\bar{\zeta}_n) = n^{1 + \lambda_0 \epsilon} \). Therefore, for sufficiently small $\delta>0$ and $\varepsilon$, there exists \( n_2 > 0 \) such that for all \( n > n_2 \),
\[
(1 + \varepsilon) \cdot \max \left\{ n + n^{\beta}, \, K y_0(\bar{\zeta}_n) + ne^{\lambda_1\delta} \right\} 
< \frac{1}{2} \left( K + n \right).
\]
Combining these results, for all \( n > \max\{n_0, n_1, n_2\} \), we have:
\[
\mathbb{P} \left( \sup_{t \leq \zeta_n+\delta} (X_0(t) + X_1(t)) < \frac{1}{2} \left( 1 + \frac{n}{K} \right), \gamma_n<\zeta_n+\delta \right) > 1-\varepsilon.
\]

Define the event
\[
\Omega_n := \left\{ \omega \,\middle|\, \sup_{t \leq \zeta_n+\delta} (X_0(t) + X_1(t)) \leq \frac{1}{2} \left( 1 + \frac{n}{K} \right), \gamma_n<\zeta_n+\delta \right\}.
\]
We have established that $\mathbb{P}(\Omega_n) \rightarrow 1$ as $n\rightarrow \infty$. By \textit{(A5)} of Assumption~\ref{assump:f}, for all $\omega\in \Omega_n$, the birth rate \( f(X_0, X_1) \) is bounded away from the death rate $d_1$. More precisely, define
\[
r_1^{\min} := \min_{\omega \in \Omega_n} f(K X_0(t), K X_1(t)) > d_1.
\]
To establish bounds on the number of surviving resistant clones in the original stochastic process, we introduce two auxiliary processes.

First, we define an \emph{upper envelope process}, denoted by \( \hat{Z}_0(s), \hat{Z}_1(s) \), and let \( \hat{I}_n(s) \) represent the number of surviving resistant clones at time \( s \) in this process. The upper envelope process evolves according to the same dynamics as the original process, except that resistant cells (those arising from mutations of sensitive cells) and their descendants experience no death events (i.e., have zero death rate). This modification ensures that any mutant clone that arises will survive indefinitely.

Next, we define a \emph{lower envelope process}, denoted by \( \bar{Z}_0(s), \bar{Z}_1(s) \), and let \( \bar{I}_n(s) \) represent the number of surviving resistant clones at time \( s \) in this process. In this lower envelope process, each resistant cell originating from mutation undergoes a birth–death process with a constant, state-independent birth rate \( r_1^{\min} \) and death rate \( d_1 \).

We now formally construct couplings between the original process and these two envelope processes.

\textbf{Upper envelope coupling:} We couple the upper envelope and original process so that each mutation in the upper envelope process simultaneously induces a mutation in the original process. Because mutant clones in the upper envelope process do not go extinct, we have \( \hat{I}_n(t) \geq I_n(t) \) for all \( t \geq 0 \).

\textbf{Lower envelope coupling:} The lower envelope process is similarly coupled to the original process through the following construction.
\begin{enumerate}
    \item Each mutation event in the lower envelope process triggers a mutation in the original process, ensuring that clones are generated in parallel in both processes.
    \item Let \( \bar{Z}_{1,i}(s) \) and \( Z_{1,i}(s) \) denote the population sizes of the \( i \)-th resistant clone in the lower envelope and original processes, respectively.
    \item For a birth event in clone \( Z_{1,i}(s) \), draw a uniform random variable \( U \sim \mathrm{Unif}[0,1] \). A corresponding birth event occurs in clone \( \bar{Z}_{1,i}(s) \) if
    \[
    U < \frac{\bar{Z}_{1,i}(s) r_1^{\min}}{Z_{1,i}(s) f(X_0(s), X_1(s))}.
    \]
    \item For each death event in clone \( Z_{1,i}(s) \), draw \( U \sim \mathrm{Unif}[0,1] \), and induce a death event in \( \bar{Z}_{1,i}(s) \) if
    \[
    U < \frac{\bar{Z}_{1,i}(s)}{Z_{1,i}(s)}.
    \]
\end{enumerate}
This coupling guarantees that \( \bar{Z}_{1,i}(s) \leq Z_{1,i}(s) \) for all \( s \in [0, t] \), because the two processes share the death events when their population sizes are equal, but the lower envelop process experiences fewer birth events. Therefore, under the event \( \Omega_n \), we have \( \bar{I}_n(t) \leq I_n(t) \).

From Theorem~2 in \cite{leder2024parameter}, it follows that:
\begin{align*}
\lim_{n \to \infty} \mathbb{P} \left( \left| \frac{1}{n^{1 - \alpha}} \hat{I}_n(\gamma_n) + \frac{1}{\lambda_0} \right| > \epsilon \right) &= 0, \text{ and}\\
\lim_{n \to \infty} \mathbb{P} \left( \left| \frac{1}{n^{1 - \alpha}} \bar{I}_n(\gamma_n) + \frac{r_1^{\min} - d_1}{\lambda_0 r_1^{\min}} \right| > \epsilon \right) &= 0.
\end{align*}
Define the constants
\begin{align*}
    c_I = \frac{1}{2} \cdot \frac{r_1^{\min} - d_1}{|\lambda_0 |r_1^{\min}}, \quad C_I = 2\cdot \frac{1}{|\lambda_0|}.
\end{align*}
It follows that
\begin{align*}
\mathbb{P} \left( c_I n^{1 - \alpha} \leq I_n(\gamma_n) \leq C_I n^{1 - \alpha} \right) 
&\geq \mathbb{P} \left( c_I n^{1 - \alpha} \leq I_n(\gamma_n) \leq C_I n^{1 - \alpha}, \Omega_n \right)  \xrightarrow[]{n \to \infty} 1.
\end{align*}
\qed

\end{proof}

\section{Proof of Proposition \ref{prop:conv_Xbeta}}
\label{sec:pf_prop_conv_Xbeta}
\begin{proof}
% To begin with, we express the dynamics of \( X_{\beta}(t) \) and \( X_1(t) \) as semimartingales involving compensated Poisson integrals:
Define the event
\begin{align*}
    \Omega_n = \left\{
        \gamma_n < \zeta_n + \delta,\;
        \sup_{t \leq \zeta_n + \delta} \left| \frac{X_0(t)}{y_0(t)} - 1 \right| < \epsilon,\;
        \sup_{t \leq \zeta_n + \delta} \left| \frac{X_1(t)}{y_1(t)} - 1 \right| < \epsilon,\;
        \sup_{t \leq \zeta_n + \delta} \left| \frac{X_{\beta}(t)}{y_{\beta}(t)} - 1 \right| < \epsilon
    \right\},
\end{align*}
which ensures that all subpopulations remain close to their deterministic counterparts up to time $\zeta_n + \delta$. The analysis below takes place on the event $\Omega_n$.

We first express \( X_{\beta}(t) \) and \( X_1(t) \) as semimartingales:
\begin{align}
\label{eqn:SDE_Xbeta}
X_{\beta}(t) &= X_{\beta}(0) + M_{\beta}(t) + \int_0^t X_{\beta}(s) \phi(KX(s)) \, ds, \\
\label{eqn:SDE_X1}
X_1(t) &= X_1(0) + M_1(t) + \int_0^t X_1(s) \phi(KX(s)) \, ds + n^{-\alpha} \int_0^t X_0(s) \, ds,
\end{align}
where the martingale terms \( M_{\beta}(t) \) and \( M_1(t) \) are given by:
\begin{align*}
M_{\beta}(t) &= \frac{1}{K} \int_0^t \int_0^{\infty} \mathbbm{1}_{\left\{ u \leq K X_{\beta}(s-) f(X(s-)) \right\}} \tilde{N}_1^b(ds, du) \\
&\quad - \frac{1}{K} \int_0^t \int_0^{\infty} \mathbbm{1}_{\left\{ u \leq K X_{\beta}(s-) d_1 \right\}} \tilde{N}_1^d(ds, du), \\
M_1(t) &= \frac{1}{K} \int_0^t \int_0^{\infty} \mathbbm{1}_{\left\{ u \leq K X_1(s-) f(X(s-)) \right\}} \tilde{N}_1^b(ds, du) \\
&\quad - \frac{1}{K} \int_0^t \int_0^{\infty} \mathbbm{1}_{\left\{ u \leq K X_1(s-) d_1 \right\}} \tilde{N}_1^d(ds, du) \\
&\quad + \frac{1}{K} \int_0^t \int_0^{\infty} \mathbbm{1}_{\left\{ u \leq K X_0(s-) n^{-\alpha} \right\}} \tilde{N}_0^m(ds, du).
\end{align*}

We define $\tau_{\beta} =\inf \left\{t : X_{\beta}(t) \leq 1/K\right\}$, $\tau_1 =\inf \left\{t : X_1(t) \leq 1/K\right\}$. Applying Itô's formula for semimartingales \cite{protter2012stochastic}  $X_{\beta}(t)$ and $X_1(t)$ yields:
%Differentiating both sides of \eqref{eqn:SDE_Xbeta} and \eqref{eqn:SDE_X1}, dividing by \( X_{\beta}(t) \) and \( X_1(t) \) respectively, and integrating, we obtain:
\begin{align}
\label{eq:log_X_beta}
\log X_{\beta}(t\wedge\tau_{\beta}) &= \log X_{\beta}(0) + \bar{M}_{\beta}(t\wedge\tau_{\beta}) + \int_0^{t\wedge\tau_{\beta}} \phi(KX(s)) \, ds + Q_{\beta}(t\wedge\tau_{\beta}), \\
\label{eq:log_X_1}
\log X_1(t\wedge\tau_1) &= \log X_1(0) + \bar{M}_1(t\wedge\tau_1) + \int_0^{t\wedge\tau_1} \phi(KX(s)) \, ds + n^{-\alpha} \int_0^{t\wedge\tau_1} \frac{X_0(s)}{X_1(s)} \, ds + Q_1(t\wedge\tau_1),
\end{align}
where
\begin{align*}
\bar{M}_{\beta}(t) &= \frac{1}{K} \int_0^t \int_0^{\infty} \frac{1}{X_{\beta}(s-)} \mathbbm{1}_{\left\{ u \leq K X_{\beta}(s-) f(X(s-)) \right\}} \tilde{N}_1^b(ds, du) \\
&\quad - \frac{1}{K} \int_0^t \int_0^{\infty} \frac{1}{X_{\beta}(s-)} \mathbbm{1}_{\left\{ u \leq K X_{\beta}(s-) d_1 \right\}} \tilde{N}_1^d(ds, du), \\
\bar{M}_1(t) &= \frac{1}{K} \int_0^t \int_0^{\infty} \frac{1}{X_1(s-)} \mathbbm{1}_{\left\{ u \leq K X_1(s-) f(X(s-)) \right\}} \tilde{N}_1^b(ds, du) \\
&\quad - \frac{1}{K} \int_0^t \int_0^{\infty} \frac{1}{X_1(s-)} \mathbbm{1}_{\left\{ u \leq K X_1(s-) d_1 \right\}} \tilde{N}_1^d(ds, du) \\
&\quad + \frac{1}{K} \int_0^t \int_0^{\infty} \frac{1}{X_1(s-)} \mathbbm{1}_{\left\{ u \leq K X_0(s-) n^{-\alpha} \right\}} \tilde{N}_0^m(ds, du).
\end{align*}
and 
\begin{align*}
Q_{\beta}(t) &= \int_0^t \int_0^{\infty} \left( \log\left(X_{\beta}(s-)+\frac{1}{K}\right)-\log\left(X_{\beta}(s-)\right)-\frac{1}{KX_{\beta}(s-)} \right)\mathbbm{1}_{\left\{ u \leq K X_{\beta}(s-) f(X(s-)) \right\}} N_1^b(ds, du) \\
&\quad + \int_0^t \int_0^{\infty} \left( \log\left(X_{\beta}(s-)-\frac{1}{K}\right)-\log\left(X_{\beta}(s-)\right)+\frac{1}{KX_{\beta}(s-)} \right)\mathbbm{1}_{\left\{ u \leq K X_{\beta}(s-) d_1 \right\}} N_1^d(ds, du), \\
Q_1(t) &= \int_0^t \int_0^{\infty} \left( \log\left(X_1(s-)+\frac{1}{K}\right)-\log\left(X_1(s-)\right)-\frac{1}{KX_1(s-)} \right)\mathbbm{1}_{\left\{ u \leq K X_1(s-) f(X(s-)) \right\}} N_1^b(ds, du) \\
&\quad + \int_0^t \int_0^{\infty} \left( \log\left(X_1(s-)-\frac{1}{K}\right)-\log\left(X_1(s-)\right)+\frac{1}{KX_1(s-)} \right)\mathbbm{1}_{\left\{ u \leq K X_1(s-) d_1 \right\}} N_1^d(ds, du)\\
&\quad + \int_0^t \int_0^{\infty} \left( \log\left(X_1(s-)+\frac{1}{K}\right)-\log\left(X_1(s-)\right)-\frac{1}{KX_1(s-)} \right)\mathbbm{1}_{\left\{ u \leq K X_0(s-) n^{-\alpha} \right\}} N_0^m(ds, du).
\end{align*}

Since $\inf_{t\leq \zeta_n+\delta} y_{\beta}(t) = \inf_{t\leq \zeta_n+\delta} y_1(t) = n^{\beta}/K$, under the event $\Omega_n$, it follows that $\gamma_n < \zeta_n + \delta < \min\{\tau_{\beta},\tau_1\}$.
From the definition of $ \gamma_n $, where $n = KX_1(\gamma_n)$, and by substituting $t$ with $\gamma_n$ into (\ref{eq:log_X_beta}) and (\ref{eq:log_X_1}) and exponentiating both sides, we obtain
% Through the definition of \( \gamma_n \), we know:
\begin{align*}
    n = KX_1(\gamma_n) 
    &= n^{\beta} \exp\left( \bar{M}_1(\gamma_n) \right) 
       \exp\left( \int_0^{\gamma_n} \phi(KX(s)) \, ds \right)
       \exp\left( n^{-\alpha} \int_0^{\gamma_n} \frac{X_0(s)}{X_1(s)} \, ds \right) \exp\left(Q_1(\gamma_n)\right).
\end{align*}
Therefore,
\begin{align*}
    KX_{\beta}(\gamma_n) 
    &= n^{\beta} \exp\left( \bar{M}_{\beta}(\gamma_n) \right) 
       \exp\left( \int_0^{\gamma_n} \phi(KX(s)) \, ds \right) \exp\left( Q_{\beta}(\gamma_n) \right) \\
    &= n \cdot \exp\Bigg( 
        \bar{M}_{\beta}(\gamma_n) 
        - \bar{M}_1(\gamma_n) 
        - n^{-\alpha} \int_0^{\gamma_n} \frac{X_0(s)}{X_1(s)} \, ds + Q_{\beta}(\gamma_n)-Q_1(\gamma_n) \Bigg).
\end{align*}
Taking logarithm yields
\begin{align}
    -\log \left( \frac{KX_{\beta}(\gamma_n)}{n} \right) 
    &=  n^{-\alpha} \int_0^{\gamma_n} \frac{X_0(s)}{X_1(s)} \, ds + \bar{M}_1(\gamma_n) -\bar{M}_{\beta}(\gamma_n) + Q_1(\gamma_n)-Q_{\beta}(\gamma_n).
    \nonumber
\end{align}
    % By Burkholder–Davis–Gundy inequality, we have 
    % \begin{align*}
    %     \mathbb{E} \left[ \sup_{s \leq t} \left| E_{\beta}(s) \right| \right]  &\leq C_1  \mathbb{E} \left[ \langle E_{\beta} \rangle_{t}^{1/2} \right] \\
    %     &\leq C_2 \mathbb{E} \left[ \left( \int_0^t \frac{1}{KX_{\beta}(s)}ds \right)^{1/2} \right]\\
    %      \mathbb{E} \left[ \sup_{s \leq t} \left| E_1(s) \right| \right]  &\leq C_3  \mathbb{E} \left[ \langle E_1 \rangle_{t}^{1/2} \right] \\
    %     &\leq C_4 \mathbb{E} \left[ \left( \int_0^t \left(\frac{1}{KX_{\beta}(s)}+\frac{n^{-\alpha}X_0(s)}{KX_1(s)^2}\right)ds \right)^{1/2} \right]
    % \end{align*}

We first consider the first term in the right hand side.
Under the event $\Omega_n$, using inequality (\ref{eq:upper_bound_1_over_y1}), we obtain the following bounds:
% then under the event \( \Omega_n \), by (\ref{eq:upper_bound_1_over_y1}) we have the following bounds:
\begin{align*}
    n^{-\alpha} \int_0^{\gamma_n} \frac{X_0(s)}{X_1(s)} \, ds
    &\leq n^{-\alpha} \int_0^{\gamma_n} \frac{y_0(s)(1 + \epsilon)}{y_1(s)(1 - \epsilon)} \, ds 
    \leq C n^{1 - \alpha - \beta}, \\
    n^{-\alpha} \int_0^{\gamma_n} \frac{X_0(s)}{X_1(s)} \, ds
    &\geq n^{-\alpha} \int_0^{\gamma_n} \frac{y_0(s)(1 - \epsilon)}{y_1(s)(1 + \epsilon)} \, ds 
    \geq c n^{1 - \alpha - \beta}.
    %, \\
    %n^{-\alpha}\int_0^{\gamma_n}\frac{X_0(s)}{K X_1(s)^2} \, ds
    %&\leq n^{-\alpha} \int_0^{\gamma_n} \frac{y_0(s)(1 - \epsilon)}{K y_1(s)^2(1 + \epsilon)^2} \, ds
    %= O(n^{1 - \alpha - 2\beta}).
\end{align*}
The explicit values for the constants $c$ and $C$ can be derived from the analysis in Proposition~\ref{prop:ybeta} (specifically, from equation (\ref{eq:bound_log_ybeta})), yielding the choices:
\begin{align}
\label{eq:setting_C}
c = \frac{1}{2}(\lambda_1 - \lambda_0), \qquad
C = \frac{2}{\bar{\lambda}_1 - \lambda_0}.
\end{align}
% To be precise, we observe from the proof of Proposition~\ref{prop:ybeta} (\ref{eq:bound_log_ybeta}) that we can choose
% \begin{align}
% \label{eq:setting_C}
% c = \frac{1}{2}(\lambda_1 - \lambda_0), \qquad
% C = \frac{2}{\bar{\lambda}_1 - \lambda_0}.
% \end{align}

% By (\ref{eq:upper_bound_1_over_y1}), (\ref{eq:upper_bound_1_over_ybeta}) and (\ref{eq:first_mom_Xm}), we also have the probabilistic estimate, for any $\theta>0$:
% \begin{align*}
%     &\mathbb{P}\left( \int_0^{\gamma_n} \frac{(f(X(s)) + d_1) X_m(s)}{K X_1(s) X_{\beta}(s)} \, ds >\theta n^{1 - \alpha - \beta} \right) \\
%     \leq\;& \mathbb{P}\left( \int_0^{\gamma_n} \frac{(f(X(s)) + d_1) X_m(s)}{K X_1(s) X_{\beta}(s)} \, ds >\theta n^{1 - \alpha - \beta},\; \Omega_n \right)
%     + \mathbb{P}(\Omega_n^C) \\
%     \leq\;& \mathbb{P}\left( \int_0^{\zeta_n + \delta} \frac{X_m(s)}{K y_1(s) y_{\beta}(s)} \, ds 
%     > \frac{(1 - \epsilon)^2\theta}{r_1 + d_1} n^{1 - \alpha - \beta} \right)
%     + \mathbb{P}(\Omega_n^C) \\
%     =\;& O\left( n^{\alpha + \beta - 1} \int_0^{\zeta_n + \delta} \frac{\mathbb{E}[X_m(s)]}{K y_1(s) y_{\beta}(s)} \, ds\right)
%     + \mathbb{P}(\Omega_n^C) \\
%     =\;& O(n^{-\beta}) + \mathbb{P}(\Omega_n^C) \xrightarrow[n \to \infty]{} 0.
% \end{align*}

The difference $\bar{M}_1(\gamma_n) - \bar{M}_{\beta}(\gamma_n)$ can be expressed as a sum of stochastic integrals with respect to the compensated Poisson measures for birth, death, and mutation events:
\begin{align*}
    \bar{M}_1(\gamma_n) - \bar{M}_{\beta}(\gamma_n) &= 
    \frac{1}{K} \int_0^{\gamma_n} \int_0^{\infty} 
        \frac{-X_m(s-)}{X_1(s-)X_{\beta}(s-)} 
        \mathbbm{1}_{\left\{ u \leq K X_{\beta}(s-) f(X(s-)) \right\}} \tilde{N}_1^b(ds, du) \\
    &\quad + \frac{1}{K} \int_0^{\gamma_n} \int_0^{\infty} 
        \frac{1}{X_1(s-)} 
        \mathbbm{1}_{\left\{ K X_{\beta}(s-) f(X(s-)) < u \leq K X_1(s-) f(X(s-)) \right\}} \tilde{N}_1^b(ds, du) \\
    &\quad - \frac{1}{K} \int_0^{\gamma_n} \int_0^{\infty} 
        \frac{-X_m(s-)}{X_1(s-)X_{\beta}(s-)} 
        \mathbbm{1}_{\left\{ u \leq K X_{\beta}(s-) d_1 \right\}} \tilde{N}_1^d(ds, du) \\
    &\quad - \frac{1}{K} \int_0^{\gamma_n} \int_0^{\infty} 
        \frac{1}{X_1(s-)} 
        \mathbbm{1}_{\left\{ K X_{\beta}(s-) d_1 < u \leq K X_1(s-) d_1 \right\}} \tilde{N}_1^d(ds, du) \\
    &\quad + \frac{1}{K} \int_0^{\gamma_n} \int_0^{\infty} 
        \frac{1}{X_1(s-)} 
        \mathbbm{1}_{\left\{ u \leq K X_0(s-) n^{-\alpha} \right\}} \tilde{N}_1^m(ds, du).
\end{align*}
Thus, it suffices to analyze the following three terms:
\begin{align}
    D_1(t) &:= \frac{1}{K} \int_0^t \int_0^{\infty} 
        \frac{X_m(s-)}{X_1(s-) X_{\beta}(s-)} 
        \mathbbm{1}_{\left\{ u \leq K X_{\beta}(s-) \right\}} \tilde{N}(ds, du), \\
    D_2(t) &:= \frac{1}{K} \int_0^t \int_0^{\infty} 
        \frac{1}{X_1(s-)} 
        \mathbbm{1}_{\left\{ K X_{\beta}(s-) < u \leq K X_1(s-) \right\}} \tilde{N}(ds, du), \\
    D_3(t) &:= \frac{1}{K} \int_0^t \int_0^{\infty} 
        \frac{1}{X_1(s-)} 
        \mathbbm{1}_{\left\{ u \leq K X_0(s-) n^{-\alpha} \right\}} \tilde{N}(ds, du),
\end{align}
where \( \tilde{N}(ds, du) \) denotes the corresponding compensated Poisson martingale measure. We begin by establishing a bound for \( D_1(t) \). For $t\le \gamma_n$, and conditional on the event $\Omega_n$, we have
\begin{align*}
    D_1(t) 
    &\leq \frac{1}{K} \int_0^{\zeta_n+\delta} \int_0^{\infty} 
    \frac{X_m(s-)}{y_1(s)y_{\beta}(s)(1 - \epsilon)^2} 
    \mathbbm{1}_{\{ u \leq K y_{\beta}(s)(1 + \epsilon) \}} 
    \tilde{N}(ds, du) 
    =: \bar{D}_1(\zeta_n + \delta).
\end{align*}
To bound the expectation of $\bar{D}_1(\zeta_n + \delta)$, we apply Jensen's inequality and (\ref{eq:second_mom_Xm}):
% By Jensen's inequality and (\ref{eq:second_mom_Xbeta}), we obtain:
\begin{align*}
    \mathbb{E}[\bar{D}_1(\zeta_n + \delta)] 
    &\leq \mathbb{E} \left[ \bar{D}_1(\zeta_n + \delta)^2 \right] ^{1/2} \leq \mathbb{E} \left[ \langle \bar{D}_1 \rangle_{\zeta_n + \delta} \right]^{1/2}\\
    &= C\left( \int_0^{\zeta_n + \delta} \frac{\mathbb{E}[X_m(s)^2]}{K y_1(s)^2 y_{\beta}(s)} ds \right)^{1/2}
    = O(n^{1 - \alpha - 3\beta/2}).
\end{align*}
We now bound the remaining terms \( D_2(t) \) and \( D_3(t) \) using a similar argument. For $D_2(t)$, we have
\begin{align*}
    D_2(t) 
    &= \frac{1}{K} \int_0^t \int_0^{\infty} 
    \frac{1}{X_1(s-)} 
    \mathbbm{1}_{\{ u \leq K X_m(s-) \}} 
    \tilde{N}(ds, du) \\
    &\leq \frac{1}{K} \int_0^{\zeta_n+\delta} \int_0^{\infty} 
    \frac{1}{y_1(s)(1 - \epsilon)} 
    \mathbbm{1}_{\{ u \leq K X_m(s-) \}} 
    \tilde{N}(ds, du) 
    =: \bar{D}_2(\zeta_n + \delta),
\end{align*}
and
\begin{align*}
    \mathbb{E}[\bar{D}_2(\zeta_n + \delta)] 
    &\leq  \mathbb{E} \left[ \langle \bar{D}_2 \rangle_{\zeta_n + \delta} \right]^{1/2} = C\left( \int_0^{\zeta_n + \delta} \frac{\mathbb{E}[X_m(s)]}{K y_1(s)^2} ds \right)^{1/2}
    = O(n^{1/2 - \alpha/2 - \beta}).
\end{align*}
For \( D_3(t) \), we have
\begin{align*}
    D_3(t) 
    &\leq \frac{1}{K} \int_0^{\zeta_n+\delta} \int_0^{\infty} 
    \frac{1}{y_1(s)(1 - \epsilon)} 
    \mathbbm{1}_{\{ u \leq K y_0(s)(1 + \epsilon) n^{-\alpha} \}} 
    \tilde{N}(ds, du) 
    =: \bar{D}_3(\zeta_n + \delta),
\end{align*}
and
\begin{align*}
    \mathbb{E}[\bar{D}_3(\zeta_n + \delta)] 
    &\leq  \mathbb{E} \left[ \langle \bar{D}_3 \rangle_{\zeta_n + \delta}^{1/2} \right] = C\left( \int_0^{\zeta_n + \delta} \frac{y_0(s) n^{-\alpha}}{K y_1(s)^2} ds \right)^{1/2}
    = O(n^{1/2 - \alpha/2 - \beta}).
\end{align*}
Therefore, for any $\theta>0$ and \( i = 1, 2, 3 \), Markov's inequality yields:
\begin{align*}
    \mathbb{P}(D_i(t) > \theta n^{1 - \alpha - \beta}) 
    &\leq \mathbb{P}(\Omega_n^c) + \mathbb{P}(D_i(t) >\theta n^{1 - \alpha - \beta}, \Omega_n) \\
    &\leq \mathbb{P}(\Omega_n^c) + \mathbb{P}(\bar{D}_i(\zeta_n + \delta) >\theta n^{1 - \alpha - \beta}) \\
    &\leq \mathbb{P}(\Omega_n^c) + \theta^{-1}n^{\alpha + \beta - 1} \mathbb{E}[\bar{D}_i(\zeta_n + \delta)] \xrightarrow[n \to \infty]{} 0,
\end{align*}
which implies that for any $\theta>0$,
\begin{align}
\label{eq:log_bound_Xbeta}
    \lim_{n \to \infty} \mathbb{P}\left( \bar{M}_1(\gamma_n) - \bar{M}_{\beta}(\gamma_n) >\theta n^{1 - \alpha - \beta} \right) = 0.
\end{align}

Lastly, we analyze the term $Q_1(\gamma_n) - Q_{\beta}(\gamma_n)$. By Taylor's theorem, we have 
\begin{align*}
    Q_1(\gamma_n) - Q_{\beta}(\gamma_n) &=  
    \int_0^{\gamma_n} \int_0^{\infty} 
        \left(-\frac{1}{2K^2X_1(s-)^2} +O\left(\frac{1}{K^3X_1(s-)^3}\right) \right)
        \mathbbm{1}_{\left\{ u \leq K X_1(s-) f(X(s-)) \right\}} N_1^b(ds, du) \\
        &\quad +     \int_0^{\gamma_n} \int_0^{\infty} 
        \left(-\frac{1}{2K^2X_1(s-)^2} +O\left(\frac{1}{K^3X_1(s-)^3}\right) \right)
        \mathbbm{1}_{\left\{ u \leq K X_1(s-) d_1 \right\}} N_1^d(ds, du) \\
        &\quad +     \int_0^{\gamma_n} \int_0^{\infty} 
        \left(-\frac{1}{2K^2X_1(s-)^2} +O\left(\frac{1}{K^3X_1(s-)^3}\right) \right)
        \mathbbm{1}_{\left\{ u \leq K X_0(s-) n^{-\alpha} \right\}} N_0^m(ds, du) \\
   &\quad - \int_0^{\gamma_n} \int_0^{\infty} 
        \left(-\frac{1}{2K^2X_{\beta}(s-)^2} +O\left(\frac{1}{K^3X_{\beta}(s-)^3}\right) \right)
        \mathbbm{1}_{\left\{ u \leq K X_{\beta}(s-) f(X(s-)) \right\}} N_1^b(ds, du) \\
    &\quad - \int_0^{\gamma_n} \int_0^{\infty} 
        \left(-\frac{1}{2K^2X_{\beta}(s-)^2} +O\left(\frac{1}{K^3X_{\beta}(s-)^3}\right) \right)
        \mathbbm{1}_{\left\{ u \leq K X_{\beta}(s-) d_1 \right\}} N_1^d(ds, du) \\
    &= -\frac{1}{2}\int_0^{\gamma_n} \int_0^{\infty} 
        \frac{1}{K^2X_1(s-)^2} \mathbbm{1}_{\left\{ u \leq K X_1(s-) f(X(s-)) \right\}} \tilde{N}_1^b(ds, du)\\
        &\quad - \frac{1}{2}\int_0^{\gamma_n} \int_0^{\infty} 
        \frac{1}{K^2X_1(s-)^2} \mathbbm{1}_{\left\{ u \leq K X_1(s-) d_1 \right\}} \tilde{N}_1^d(ds, du)\\
        &\quad - \frac{1}{2}\int_0^{\gamma_n} \int_0^{\infty} 
        \frac{1}{K^2X_1(s-)^2} \mathbbm{1}_{\left\{ u \leq K X_0(s-) n^{-\alpha} \right\}} \tilde{N}_0^m(ds, du)\\
        &\quad + \frac{1}{2}\int_0^{\gamma_n} \int_0^{\infty} 
        \frac{1}{K^2X_{\beta}(s-)^2} \mathbbm{1}_{\left\{ u \leq K X_{\beta}(s-) f(X(s-)) \right\}} \tilde{N}_1^b(ds, du)\\
        &\quad + \frac{1}{2}\int_0^{\gamma_n} \int_0^{\infty} 
        \frac{1}{K^2X_{\beta}(s-)^2} \mathbbm{1}_{\left\{ u \leq K X_{\beta}(s-) d_1 \right\}} \tilde{N}_1^d(ds, du)\\
        &\quad -\frac{1}{2}\int_0^{\gamma_n}\frac{f(X(s)+d_1)}{KX_1(s)}ds+ \frac{1}{2}\int_0^{\gamma_n}\frac{f(X(s)+d_1)}{KX_{\beta}(s)}ds - \frac{1}{2}\int_0^{\gamma_n}\frac{X_0(s)n^{-\alpha}}{KX_1(s)^2}ds + R_n\\
        & = o\left(D_1(\gamma_n)\right) + o\left(D_2(\gamma_n)\right)+o\left(D_3(\gamma_n)\right) +R_n \\
        &\quad  +\int_0^{\gamma_n} \frac{(f(X(s)) + d_1) X_m(s)}{K X_1(s) X_{\beta}(s)} \, ds - \frac{1}{2}\int_0^{\gamma_n}\frac{X_0(s)n^{-\alpha}}{KX_1(s)^2}ds .\\
\end{align*}
where \(R_n\) denotes a negligible remainder term. The last equality comes from the fact on the event $\Omega_n$, $KX_1(s)\geq KX_{\beta}(s)\geq (1-\epsilon)n^{\beta}$ for $s\leq \zeta_n+\delta$. Thus, it suffices to analyze the last two terms. On the event \( \Omega_n \), by (\ref{eq:upper_bound_1_over_y1}) we have
\begin{align*}
n^{-\alpha}\int_0^{\gamma_n}\frac{X_0(s)}{K X_1(s)^2} \, ds
    &\leq n^{-\alpha} \int_0^{\gamma_n} \frac{y_0(s)(1 - \epsilon)}{K y_1(s)^2(1 + \epsilon)^2} \, ds
    = O(n^{1 - \alpha - 2\beta}).    
\end{align*}
Applying bounds from (\ref{eq:upper_bound_1_over_y1}), (\ref{eq:upper_bound_1_over_ybeta}), and the moment estimate (\ref{eq:first_mom_Xm}), we obtain for any \(\theta>0\):
% Moreover, by (\ref{eq:upper_bound_1_over_y1}), (\ref{eq:upper_bound_1_over_ybeta}) and (\ref{eq:first_mom_Xm}), we obtain the following probabilistic estimate for any \(\theta>0\):
 \begin{align*}
     &\mathbb{P}\left( \int_0^{\gamma_n} \frac{(f(X(s)) + d_1) X_m(s)}{K X_1(s) X_{\beta}(s)} \, ds >\theta n^{1 - \alpha - \beta} \right) \\
     \leq\;& \mathbb{P}\left( \int_0^{\gamma_n} \frac{(f(X(s)) + d_1) X_m(s)}{K X_1(s) X_{\beta}(s)} \, ds >\theta n^{1 - \alpha - \beta},\; \Omega_n \right)
     + \mathbb{P}(\Omega_n^C) \\
     \leq\;& \mathbb{P}\left( \int_0^{\zeta_n + \delta} \frac{X_m(s)}{K y_1(s) y_{\beta}(s)} \, ds 
     > \frac{(1 - \epsilon)^2\theta}{r_1 + d_1} n^{1 - \alpha - \beta} \right)
     + \mathbb{P}(\Omega_n^C) \\
     =\;& O\left( n^{\alpha + \beta - 1} \int_0^{\zeta_n + \delta} \frac{\mathbb{E}[X_m(s)]}{K y_1(s) y_{\beta}(s)} \, ds\right)
     + \mathbb{P}(\Omega_n^C) \\
     =\;& O(n^{-\beta}) + \mathbb{P}(\Omega_n^C) \xrightarrow[n \to \infty]{} 0,
 \end{align*}
which implies \(Q_1(\gamma_n) - Q_{\beta}(\gamma_n) =o\left( n^{1-\alpha-\beta} \right)\). Consequently,
\begin{align*}
    \lim_{n \to \infty} \mathbb{P}\left( c n^{1 - \alpha - \beta} < -\log \left( \frac{K X_{\beta}(\gamma_n)}{n} \right) < C n^{1 - \alpha - \beta} \right) = 1.
\end{align*}
\qed
\end{proof}

    % Under the event of $\Omega_n$, we have 
    % \begin{align*}
    %     X_m(s) &= X_1(s)-X_{\beta}(s) \\
    %     &= (X_1(s)-y_1(s))-(X_{\beta}(s)-y_{\beta}(s))+ y_1(s)-y_{\beta}(s)\\
    %     & \leq |X_1(s)-y_1(s)|-|X_{\beta}(s)-y_{\beta}(s)|+ y_m(s)\\
    %     &\leq y_m(s) + y_1(s)\epsilon + y_{\beta}(s)\epsilon
    % \end{align*}

\section{Proof of Theorem \ref{thm: estimators}}
\label{sec:pf_thm_estimators}
\begin{proof} In what follows, We prove the consistency for the estimators $\hat{\alpha}$, $\hat{\beta}$, $\hat{\lambda}_0$, and $\hat{\lambda}_1$.
\\

\noindent\textbf{(1)} \textbf{\boldmath$\hat{\alpha}$:} From Proposition~\ref{prop:bound_of_In}, we obtain
\begin{align*}
    \lim_{n\to\infty}\mathbb{P} \left( c n^{1 - \alpha} \leq I_n(\gamma_n) \leq C n^{1 - \alpha} \right) = 1.
\end{align*}
Taking logarithms yields
\begin{align*}
    \lim_{n\to\infty}\mathbb{P} \left( \alpha-\log_n c \leq 1 - \log_n I_n(\gamma_n) \leq \alpha-\log_n C \right) = 1.
\end{align*}
Since $\log_n c \rightarrow 0$ and $\log_n C \rightarrow 0$ as $n \rightarrow \infty$, it follows that
\[
\hat{\alpha} - \alpha = 1 - \log_n I_n(\gamma_n) - \alpha \xrightarrow{p} 0, 
\]
establishing the consistency of the estimator $\hat{\alpha}$.
\\

\noindent\textbf{(2)} \textbf{\boldmath$\hat{\beta}$:} From the proof of Proposition~\ref{prop:ybeta} (specifically equation (\ref{eq:y_beta_zeta})), we obtain
\begin{align*}
   n - K y_{\beta}(\zeta_n) 
   &= n \left( 1 - \exp\left( -n^{-\alpha} \int_0^{\zeta_n} \frac{y_0}{y_1} \, ds \right) \right) \\
   &= O\left( n \left( 1 - \exp\left( -\frac{n^{1 - \alpha - \beta}}{2(\lambda_1(s))- \lambda_0(s))} \right) \right) \right) \\
   &= O\left( n^{2 - \alpha - \beta} \right),
\end{align*}
This asymptotic bound further implies $\frac{K y_{\beta}(\zeta_n)}{n} \rightarrow 1$ as $n\rightarrow \infty$.

Now consider:
\begin{align*}
    \hat{\beta} - \beta 
    &= 1 - \hat{\alpha} - \frac{\log \log \left( \frac{n}{Z_{\beta}(\gamma_n)} \right)}{\log n} - \beta \\
    &= (1 - \alpha) + (\alpha - \hat{\alpha}) 
    - \frac{\log \log \left( \frac{n}{Z_{\beta}(\gamma_n)} \right)}{\log n}
    + \frac{\log \log \left( \frac{n}{K y_{\beta}(\zeta_n)} \right)}{\log n}
    - \frac{\log \log \left( \frac{n}{K y_{\beta}(\zeta_n)} \right)}{\log n}
    - \beta \\
    &\leq |\alpha - \hat{\alpha}| 
    + \left| 1 - \alpha - \beta - \frac{\log \log \left( \frac{n}{K y_{\beta}(\zeta_n)} \right)}{\log n} \right|
    + \left| \frac{\log \log \left( \frac{n}{Z_{\beta}(\gamma_n)} \right)}{\log n} 
    - \frac{\log \log \left( \frac{n}{K y_{\beta}(\zeta_n)} \right)}{\log n} \right|.
\end{align*}
By Proposition~\ref{prop:ybeta} and the established convergence \( \hat{\alpha} \xrightarrow{p} \alpha \), the first two terms converge to zero in probability. It therefore suffices to analyze the asymptotic behavior of the remaining term:
\begin{align*}
   \frac{\log \log \left( \frac{n}{Z_{\beta}(\gamma_n)} \right)}{\log n} - \frac{\log \log \left( \frac{n}{K y_{\beta}(\zeta_n)} \right)}{\log n} 
   &= \frac{1}{\log n} \log \left( \frac{ \log \left( \frac{Z_{\beta}(\gamma_n) - n}{n} + 1 \right) }{ \log \left( \frac{K y_{\beta}(\zeta_n) - n}{n} + 1 \right) } \right) \\
   &= \frac{1}{\log n} \log \left( \frac{ Z_{\beta}(\gamma_n) - n + o(Z_{\beta}(\gamma_n) - n) }{ K y_{\beta}(\zeta_n) - n + o(K y_{\beta}(\zeta_n) - n) } \right).
\end{align*}
By Proposition~\ref{prop:conv_Xbeta}, we have
\begin{align*}
    \lim_{n \to \infty} \mathbb{P} \left( c n^{1 - \alpha - \beta} 
    < -\log \left( \frac{Z_{\beta}(\gamma_n)}{n} \right) 
    < C n^{1 - \alpha - \beta} \right) = 1.
\end{align*}
This implies
\begin{align*}
    \lim_{n \to \infty} \mathbb{P} \left( 
    n - n \exp(-c n^{1 - \alpha - \beta}) 
    < n - Z_{\beta}(\gamma_n) 
    < n - n \exp(-C n^{1 - \alpha - \beta}) 
    \right) = 1.
\end{align*}
Applying a Taylor expansion to the exponential terms yields
\begin{align*}
    \lim_{n \to \infty} 
    \mathbb{P} \left( 
    \frac{c}{2} < 
    \frac{n - Z_{\beta}(\gamma_n)}{n^{2 - \alpha - \beta}} 
    < 2C 
    \right) = 1.
\end{align*}
From the proofs of Proposition~\ref{prop:ybeta} and Proposition~\ref{prop:conv_Xbeta}, particularly drawing on equations (\ref{eq:bound_log_ybeta}) and  (\ref{eq:setting_C}), we establish that for the same constants \( c, C > 0 \), the following bounds hold:
\begin{align*}
    \frac{c}{2} < 
    \liminf_{n \to \infty} \frac{n - K y_{\beta}(\zeta_n)}{n^{2 - \alpha - \beta}} 
    \leq 
    \limsup_{n \to \infty} \frac{n - K y_{\beta}(\zeta_n)}{n^{2 - \alpha - \beta}} 
    < 2C.
\end{align*}
Therefore, both \( n - Z_{\beta}(\gamma_n) \) and \( n - K y_{\beta}(\zeta_n) \) are of order \( n^{2 - \alpha - \beta} \) with high probability, and their ratio remains bounded away from zero and infinity. Thus, for any \( \epsilon > 0 \),
\begin{align*}
    \lim_{n \to \infty} 
    \mathbb{P} \left( 
        \frac{1}{\log n} \log \left( \frac{ Z_{\beta}(\gamma_n) - n + o(Z_{\beta}(\gamma_n) - n) }{ K y_{\beta}(\zeta_n) - n + o(K y_{\beta}(\zeta_n) - n) } \right) > \epsilon 
    \right) = 0,
\end{align*}
which completes the proof.
\\

\noindent\textbf{(3)} \textbf{\boldmath$\hat{\lambda}_0$:}
We now analyze the convergence of the estimator $\hat{\lambda}_0$. Consider the following decomposition:
\begin{align*}
    |\hat{\lambda}_0 - \lambda_0| 
    &= \left| \frac{1}{\gamma_n} \log \frac{Z_0(\gamma_n)}{n} - \lambda_0 \right| \\
    &= \left| \frac{1}{\gamma_n} \log \frac{K X_0(\gamma_n)}{n} 
    - \frac{1}{\gamma_n} \log \frac{K y_0(\gamma_n)}{n} 
    + \frac{1}{\gamma_n} \log \frac{K y_0(\gamma_n)}{n} - \lambda_0 \right|.
\end{align*}
Because \( K y_0(\gamma_n) = n e^{\lambda_0 \gamma_n} \), the last two terms combine to yield zero:
\begin{align*}
    |\hat{\lambda}_0 - \lambda_0|
    &= \left| \frac{1}{\gamma_n} \log \frac{X_0(\gamma_n)}{y_0(\gamma_n)} \right| 
    = \frac{1}{\gamma_n} \left| \log \left( \frac{X_0(\gamma_n)}{y_0(\gamma_n)} \right) \right| 
    = \frac{1}{\gamma_n} \left| \log \left( 1 + \left( \frac{X_0(\gamma_n)}{y_0(\gamma_n)} - 1 \right) \right) \right|.
\end{align*}
By Theorem~\ref{thm:ratio_conv} and Proposition~\ref{prop:convergence_of_gamma_n}, we have:
\begin{align*}
    &\lim_{n \to \infty} \mathbb{P}(|\gamma_n - \zeta_n| > \delta) = 0, \text{ and} \\
    &\lim_{n \to \infty} \mathbb{P} \left( \left| \frac{X_0(\gamma_n)}{y_0(\gamma_n)} - 1 \right| > \varepsilon,\; |\gamma_n - \zeta_n| < \delta \right) = 0.
\end{align*}
Hence, for any $\epsilon>0$,
\begin{align*}
    \mathbb{P}(|\hat{\lambda}_0 - \lambda_0| > \epsilon) 
    &\leq \mathbb{P}(|\hat{\lambda}_0 - \lambda_0| > \epsilon,\; |\gamma_n - \zeta_n| < \delta) 
    + \mathbb{P}(|\gamma_n - \zeta_n| \geq \delta) \\
    &= \mathbb{P} \left( \frac{1}{\gamma_n} \left| \log \left( 1 + \left( \frac{X_0(\gamma_n)}{y_0(\gamma_n)} - 1 \right) \right) \right| > \epsilon,\; |\gamma_n - \zeta_n| < \delta \right) 
    + \mathbb{P}(|\gamma_n - \zeta_n| \geq \delta) \\
    &\leq \mathbb{P} \left(  \left| \frac{X_0(\gamma_n)}{y_0(\gamma_n)} - 1 \right| > (\zeta_n - \delta)\epsilon/2,\; |\gamma_n - \zeta_n| < \delta\right) + \mathbb{P}(|\gamma_n - \zeta_n| \geq \delta) \xrightarrow{n \to \infty} 0.
\end{align*}
\\

\noindent\textbf{(4)} \textbf{\boldmath$\hat{\lambda}_1$:} Lastly, we analyze the convergence of the estimator $\hat{\lambda}_1$. Consider the following decomposition:
\begin{align*}
    \hat{\lambda}_1 - \lambda_1 
    &= \frac{1 - \hat{\beta}}{\gamma_n} \log n - \lambda_1 \\
    &= \left( \frac{1 - \hat{\beta}}{\gamma_n} - \frac{1 - \beta}{\gamma_n} \right) \log n 
    + \left( \frac{1 - \beta}{\gamma_n} - \frac{1 - \beta}{\zeta_n} \right) \log n 
    + \left( \frac{1 - \beta}{\zeta_n} \log n - \lambda_1 \right) \\
    &\leq \frac{\log n}{\gamma_n} |\hat{\beta} - \beta| 
    + \left| \frac{\log n}{\gamma_n} - \frac{\log n}{\zeta_n} \right| (1 - \beta) 
    + \left| \frac{1 - \beta}{\zeta_n} \log n - \lambda_1 \right|.
\end{align*}
Since we have established that \( \hat{\beta} \xrightarrow{p} \beta \), and since Proposition~\ref{prop:zeta} and Proposition~\ref{prop:convergence_of_gamma_n} imply \( \gamma_n \xrightarrow{p} \zeta_n \) with \( \zeta_n =\Theta(\log n) \), it follows that each term on the right-hand side converges to $0$ in probability.
Thus, we conclude:
\[
\hat{\lambda}_1 \xrightarrow{p} \lambda_1.
\]

\end{proof} 
\end{appendix}

\newpage
\bibliographystyle{plain}
\bibliography{ref}

\begin{thebibliography}{10}

\bibitem{bansaye2024sharp}
Vincent Bansaye, Xavier Erny, and Sylvie M{\'e}l{\'e}ard.
\newblock Sharp approximation and hitting times for stochastic invasion processes.
\newblock {\em Stochastic Processes and their Applications}, 178:104458, 2024.

\bibitem{bansaye2015stochastic}
Vincent Bansaye and Sylvie M{\'e}l{\'e}ard.
\newblock {\em Stochastic models for structured populations}, volume~16.
\newblock Springer, 2015.

\bibitem{beltran2019role}
Himisha Beltran, Andrew Hruszkewycz, Howard~I Scher, Jeffrey Hildesheim, Jennifer Isaacs, Evan~Y Yu, Kathleen Kelly, Daniel Lin, Adam Dicker, Julia Arnold, et~al.
\newblock The role of lineage plasticity in prostate cancer therapy resistance.
\newblock {\em Clinical cancer research}, 25(23):6916--6924, 2019.

\bibitem{benzekry2014classical}
S{\'e}bastien Benzekry, Clare Lamont, Afshin Beheshti, Amanda Tracz, John~ML Ebos, Lynn Hlatky, and Philip Hahnfeldt.
\newblock Classical mathematical models for description and prediction of experimental tumor growth.
\newblock {\em PLoS computational biology}, 10(8):e1003800, 2014.

\bibitem{brouard2023genetic}
Vianney Brouard.
\newblock Genetic composition of supercritical branching populations under power law mutation rates.
\newblock {\em arXiv preprint arXiv:2309.12055}, 2023.

\bibitem{chaft2011disease}
Jamie~E Chaft, Geoffrey~R Oxnard, Camelia~S Sima, Mark~G Kris, Vincent~A Miller, and Gregory~J Riely.
\newblock Disease flare after tyrosine kinase inhibitor discontinuation in patients with egfr-mutant lung cancer and acquired resistance to erlotinib or gefitinib: implications for clinical trial design.
\newblock {\em Clinical cancer research}, 17(19):6298--6303, 2011.

\bibitem{dagogo2018tumour}
Ibiayi Dagogo-Jack and Alice~T Shaw.
\newblock Tumour heterogeneity and resistance to cancer therapies.
\newblock {\em Nature reviews Clinical oncology}, 15(2):81--94, 2018.

\bibitem{ding2012clonal}
Li~Ding, Timothy~J Ley, David~E Larson, Christopher~A Miller, Daniel~C Koboldt, John~S Welch, Julie~K Ritchey, Margaret~A Young, Tamara Lamprecht, Michael~D McLellan, et~al.
\newblock Clonal evolution in relapsed acute myeloid leukaemia revealed by whole-genome sequencing.
\newblock {\em Nature}, 481(7382):506--510, 2012.

\bibitem{durrett2015branching}
Richard Durrett.
\newblock Branching process models of cancer.
\newblock In {\em Branching process models of cancer}, pages 1--63. Springer, 2015.

\bibitem{ethier2009markov}
Stewart~N Ethier and Thomas~G Kurtz.
\newblock {\em Markov processes: characterization and convergence}.
\newblock John Wiley \& Sons, 2009.

\bibitem{gatenby2009adaptive}
Robert~A Gatenby, Ariosto~S Silva, Robert~J Gillies, and B~Roy Frieden.
\newblock Adaptive therapy.
\newblock {\em Cancer research}, 69(11):4894--4903, 2009.

\bibitem{gunnarsson2025limit}
Einar~Bjarki Gunnarsson, Kevin Leder, and Xuanming Zhang.
\newblock Limit theorems for the site frequency spectrum of neutral mutations in an exponentially growing population.
\newblock {\em Stochastic Processes and their Applications}, 182:104565, 2025.

\bibitem{hobor2024mixed}
Sebastijan Hobor, Maise Al~Bakir, Crispin~T Hiley, Marcin Skrzypski, Alexander~M Frankell, Bjorn Bakker, Thomas~BK Watkins, Aleksandra Markovets, Jonathan~R Dry, Andrew~P Brown, et~al.
\newblock Mixed responses to targeted therapy driven by chromosomal instability through p53 dysfunction and genome doubling.
\newblock {\em Nature communications}, 15(1):4871, 2024.

\bibitem{ichihara2015shades}
Eiki Ichihara and Christine~M Lovly.
\newblock Shades of t790m: intratumor heterogeneity in egfr-mutant lung cancer.
\newblock {\em Cancer discovery}, 5(7):694--696, 2015.

\bibitem{johnson2014mutational}
Brett~E Johnson, Tali Mazor, Chibo Hong, Michael Barnes, Koki Aihara, Cory~Y McLean, Shaun~D Fouse, Shogo Yamamoto, Hiroki Ueda, Kenji Tatsuno, et~al.
\newblock Mutational analysis reveals the origin and therapy-driven evolution of recurrent glioma.
\newblock {\em Science}, 343(6167):189--193, 2014.

\bibitem{johnson2023clonerate}
Brian Johnson, Yubo Shuai, Jason Schweinsberg, and Kit Curtius.
\newblock clonerate: fast estimation of single-cell clonal dynamics using coalescent theory.
\newblock {\em Bioinformatics}, 39(9):btad561, 2023.

\bibitem{lambert2005branching}
Amaury Lambert.
\newblock The branching process with logistic growth.
\newblock 2005.

\bibitem{leder2024parameter}
Kevin Leder, Ruping Sun, Zicheng Wang, and Xuanming Zhang.
\newblock Parameter estimation from single patient, single time-point sequencing data of recurrent tumors.
\newblock {\em Journal of Mathematical Biology}, 89(5):51, 2024.

\bibitem{leder2021clonal}
Kevin Leder and Zicheng Wang.
\newblock Clonal diversity at cancer recurrence.
\newblock {\em arXiv preprint arXiv:2108.13472}, 2021.

\bibitem{lewinsohn2023state}
Maya~A Lewinsohn, Trevor Bedford, Nicola~F M{\"u}ller, and Alison~F Feder.
\newblock State-dependent evolutionary models reveal modes of solid tumour growth.
\newblock {\em Nature Ecology \& Evolution}, 7(4):581--596, 2023.

\bibitem{perez2024seom}
Jose~Alejandro Perez-Fidalgo, Fernando G{\'a}lvez-Montosa, Eva~Mar{\'\i}a Guerra, Ainhoa Madariaga, Aranzazu Manzano, Cristina Martin-Lorente, Maria~Jes{\'u}s Rubio-P{\'e}rez, Jesus Alarc{\'o}n, Mar{\'\i}a~Pilar Barretina-Ginesta, and Lydia Gaba.
\newblock Seom--geico clinical guideline on epithelial ovarian cancer (2023).
\newblock {\em Clinical and Translational Oncology}, 26(11):2758--2770, 2024.

\bibitem{piotrowska2015heterogeneity}
Zofia Piotrowska, Matthew~J Niederst, Chris~A Karlovich, Heather~A Wakelee, Joel~W Neal, Mari Mino-Kenudson, Linnea Fulton, Aaron~N Hata, Elizabeth~L Lockerman, Anuj Kalsy, et~al.
\newblock Heterogeneity underlies the emergence of egfr t790 wild-type clones following treatment of t790m-positive cancers with a third-generation egfr inhibitor.
\newblock {\em Cancer discovery}, 5(7):713--722, 2015.

\bibitem{prodhomme2023strong}
Adrien Prodhomme.
\newblock Strong gaussian approximation of metastable density-dependent markov chains on large time scales.
\newblock {\em Stochastic Processes and their Applications}, 160:218--264, 2023.

\bibitem{protter2012stochastic}
Philip~E Protter.
\newblock Stochastic differential equations.
\newblock In {\em Stochastic integration and differential equations}, pages 249--361. Springer, 2012.

\bibitem{raatz2021impact}
Michael Raatz, Saumil Shah, Guranda Chitadze, Monika Br{\"u}ggemann, and Arne Traulsen.
\newblock The impact of phenotypic heterogeneity of tumour cells on treatment and relapse dynamics.
\newblock {\em PLoS Computational Biology}, 17(2):e1008702, 2021.

\bibitem{rogers2024re}
Susanne Rogers, Markus Gross, Ekin Ermis, Gizem Cosgun, Brigitta~G Baumert, Thomas Mader, Christina Schroeder, Nicoletta Lomax, Sara Alonso, Adela Ademaj, et~al.
\newblock Re-irradiation for recurrent glioblastoma: a pattern of care analysis.
\newblock {\em BMC neurology}, 24(1):462, 2024.

\bibitem{salehi2021clonal}
Sohrab Salehi, Farhia Kabeer, Nicholas Ceglia, Mirela Andronescu, Marc~J Williams, Kieran~R Campbell, Tehmina Masud, Beixi Wang, Justina Biele, Jazmine Brimhall, et~al.
\newblock Clonal fitness inferred from time-series modelling of single-cell cancer genomes.
\newblock {\em Nature}, 595(7868):585--590, 2021.

\bibitem{samur2023high}
Mehmet~Kemal Samur, Marco Roncador, Anil Aktas~Samur, Mariateresa Fulciniti, Abdul~Hamid Bazarbachi, Raphael Szalat, Masood~A Shammas, Adam~S Sperling, Paul~G Richardson, Florence Magrangeas, et~al.
\newblock High-dose melphalan treatment significantly increases mutational burden at relapse in multiple myeloma.
\newblock {\em Blood}, 141(14):1724--1736, 2023.

\bibitem{simpson2025challenges}
Kathryn~L Simpson, Dominic~G Rothwell, Fiona Blackhall, and Caroline Dive.
\newblock Challenges of small cell lung cancer heterogeneity and phenotypic plasticity.
\newblock {\em Nature Reviews Cancer}, pages 1--16, 2025.

\bibitem{vaz2023recurrent}
Maria~Angeles Vaz-Salgado, Mar{\'\i}a Villamayor, V{\'\i}ctor Albarr{\'a}n, V{\'\i}ctor Al{\'\i}a, Pilar Sotoca, Jes{\'u}s Chamorro, Diana Rosero, Ana~M Barrill, Mercedes Mart{\'\i}n, Eva Fernandez, et~al.
\newblock Recurrent glioblastoma: a review of the treatment options.
\newblock {\em Cancers}, 15(17):4279, 2023.

\bibitem{wang2022single}
L~Wang, J~Jung, H~Babikir, K~Shamardani, S~Jain, X~Feng, N~Gupta, S~Rosi, S~Chang, D~Raleigh, et~al.
\newblock A single-cell atlas of glioblastoma evolution under therapy reveals cell-intrinsic and cell-extrinsic therapeutic targets. nat cancer 3: 1534--1552, 2022.

\bibitem{wang2025stochastic}
Ren-Yi Wang, Marek Kimmel, and Guodong Pang.
\newblock Stochastic dynamics of two-compartment cell proliferation models with regulatory mechanisms for hematopoiesis.
\newblock {\em Journal of Mathematical Biology}, 91(2):18, 2025.

\bibitem{williams2018quantification}
Marc~J Williams, Benjamin Werner, Timon Heide, Christina Curtis, Chris~P Barnes, Andrea Sottoriva, and Trevor~A Graham.
\newblock Quantification of subclonal selection in cancer from bulk sequencing data.
\newblock {\em Nature genetics}, 50(6):895--903, 2018.

\bibitem{youk2024quantitative}
Jeonghwan Youk, Hyun~Woo Kwon, Joonoh Lim, Eunji Kim, Taewoo Kim, Ryul Kim, Seongyeol Park, Kijong Yi, Chang~Hyun Nam, Sara Jeon, et~al.
\newblock Quantitative and qualitative mutational impact of ionizing radiation on normal cells.
\newblock {\em Cell Genomics}, 4(2), 2024.

\bibitem{zhang2017integrating}
Jingsong Zhang, Jessica~J Cunningham, Joel~S Brown, and Robert~A Gatenby.
\newblock Integrating evolutionary dynamics into treatment of metastatic castrate-resistant prostate cancer.
\newblock {\em Nature communications}, 8(1):1816, 2017.

\bibitem{zhang2012activation}
Zhenfeng Zhang, Jae~Cheol Lee, Luping Lin, Victor Olivas, Valerie Au, Thomas LaFramboise, Mohamed Abdel-Rahman, Xiaoqi Wang, Alan~D Levine, Jin~Kyung Rho, et~al.
\newblock Activation of the axl kinase causes resistance to egfr-targeted therapy in lung cancer.
\newblock {\em Nature genetics}, 44(8):852--860, 2012.

\end{thebibliography}

\vspace{0.5in}
\end{document}